\documentclass[preprint]{imsart}

\pdfoutput=1

\RequirePackage{xcolor}
\RequirePackage{mathrsfs,mathtools,latexsym, enumerate, amsmath, amsthm, amssymb,hyperref,amsfonts, pdftricks,tikz,framed,color,enumerate,xcolor,bbm}
\RequirePackage{romanbar}
\RequirePackage{times}
\RequirePackage{multirow}
\RequirePackage{framed}
\RequirePackage{MnSymbol} 
\RequirePackage{pgfplots}
\RequirePackage{stmaryrd}
\RequirePackage{listings,mleftright}
\usetikzlibrary{arrows.meta, decorations.markings,math,calc} 
\RequirePackage{algorithmic}
\RequirePackage[ngerman,main=english]{babel}
\RequirePackage{tikz}
\usetikzlibrary{arrows.meta} 
\usetikzlibrary{calc}
\RequirePackage{pgfplots}
\tikzset{
    disc/.style={fill=black!40, draw=black,thick}
}
\tikzset{
    discg/.style={fill=#1, draw=black,thick}
}
\tikzset{
    vertex/.style={circle, draw, fill=black!50, inner sep=0pt, minimum width=5pt}
}
\tikzset{
     mystar/.style={star,star points=5,star point ratio=2,fill=blue!20,draw,inner sep=1.5pt}
}
\def\centerarc[#1](#2)(#3:#4:#5)
    { \draw[#1] ($(#2)+({#5*cos(#3)},{#5*sin(#3)})$) arc (#3:#4:#5); }

\usetikzlibrary{decorations.markings, calc, arrows} 
\RequirePackage{ifthen}
\RequirePackage{xspace} 

\startlocaldefs
\newcommand{\rn}[1]{\Romanbar{#1}}
\newcommand{\Z}{\ensuremath \mathbb{Z}}
\newcommand{\norm}[1]{\left\lvert #1 \right\rvert} 
\newcommand{\mnorm}[1]{\left\| #1 \right\|} 
\newcommand{\tr}{\operatorname{Trace}}

\newcommand{\lb}[1]{{\lfloor #1 \rfloor_h}}
\newcommand{\ub}[1]{{\lceil #1 \rceil_h}}

\newcommand{\law}{\operatorname{Law}}
\newcommand{\J}{\boldsymbol{\mathcal{J}}}

\newcommand{\W}{\boldsymbol{W}}

\newcommand{\differential}{D}

\newcommand{\R}{\mathbb{R}}

\definecolor{ncolor}{rgb}{0.0, 0.5, 0.0}
\definecolor{mcolor}{rgb}{0.6, 0.1, 0.6}

\def\J{\mathrm{Couplings}}

\definecolor{midnight_color}{rgb}{0.000000,0.000000,0.501961}
\definecolor{collision_color}{rgb}{0.990445,0.502228,0.032891}
\definecolor{hamiltonian_color}{rgb}{0.000111,0.001760,0.998218}
\definecolor{rutgers_red}{rgb}{0.756199,0.000000,0.115559}

\theoremstyle{plain}
\newtheorem{prototheorem}{Theorem}

\newtheorem{protoconj}[prototheorem]{Conjecture}

\newtheorem{thm}[prototheorem]{Theorem}
\newtheorem{lem}[prototheorem]{Lemma}
\newtheorem{defn}[prototheorem]{Definition}
\newtheorem{cor}[prototheorem]{Corollary}
\newtheorem{assumption}[prototheorem]{Assumption}

\newtheorem{rem}[prototheorem]{Remark}

\DeclarePairedDelimiterX{\infdivx}[2]{(}{)}{%
  #1\;\delimsize\|\;#2%
}
\newcommand{\KL}{\mathrm{KL}\infdivx}

\DeclarePairedDelimiterX{\distx}[2]{(}{)}{%
  #1  ,  #2%
}
\newcommand{\TV}{\mathrm{TV}\distx}

\endlocaldefs

\begin{document}
\begin{frontmatter}

\title{Mixing Time Guarantees for Unadjusted Hamiltonian Monte Carlo}
\runtitle{Mixing Time Guarantees for uHMC}

\begin{aug}
\Author[a]{\fnms{Nawaf} \snm{Bou-Rabee}\ead[label=e1]{nawaf.bourabee@rutgers.edu}}
\and
\Author[b]{\fnms{Andreas} \snm{Eberle}\ead[label=e2]{eberle@uni-bonn.de}}
\address[a]{Department of Mathematical Sciences \\ Rutgers University Camden \\ 311 N 5th Street \\ Camden, NJ 08102 USA \\ \href{mailto:nawaf.bourabee@rutgers.edu}{nawaf.bourabee@rutgers.edu}}
\address[b]{Institut f\"{u}r Angewandte Mathematik \\  Universit\"{a}t Bonn \\ Endenicher Allee 60 \\  53115 Bonn, Germany \\ \href{mailto:eberle@uni-bonn.de}{eberle@uni-bonn.de}}
\runAuthor{N.~Bou-Rabee \and A.~Eberle}
\end{aug}

\begin{abstract}
We provide quantitative upper bounds on the total variation mixing time of the Markov chain corresponding to the unadjusted Hamiltonian Monte Carlo (uHMC) algorithm. For two general classes of models and fixed time discretization step size $h$, the mixing
time is shown to depend only logarithmically on the dimension. Moreover, we provide quantitative
upper bounds on the total variation distance 
between the invariant measure of the uHMC chain
and the true target measure. As a consequence,
we show that 
an $\varepsilon$-accurate approximation of the target distribution $\mu$ 
in total variation distance can be achieved 
by uHMC for a broad class of models with 
$O\left(d^{3/4}\varepsilon^{-1/2}\log (d/\varepsilon )\right)$
gradient evaluations, and for mean field models
with weak interactions with
$O\left(d^{1/2}\varepsilon^{-1/2}\log (d/\varepsilon )\right)$
gradient evaluations. 
The proofs are based on the construction of successful couplings 
for uHMC that realize the upper bounds. 
\end{abstract}



\end{frontmatter}


\section{Introduction}


A central problem in Markov chain Monte Carlo (MCMC) is to determine the number of MCMC steps that guarantees that the distribution of the chain is a good approximation of its invariant probability measure.  This ``mixing time'' is commonly measured in terms of the total variation distance. Many tools have been developed for bounding mixing times, see \cite{levin2009markov} for an overview in the case of discrete state spaces. 
For Markov processes on continuous state spaces,
these tools 
include geometric and analytic approaches based on conductance and isoperimetric inequalities \cite{MontenegroTetali,Vempala}, spectral gaps and functional inequalities \cite{SaloffCoste, bakry2013analysis}, and hypocoercivity \cite{VillaniHypo,DolbeaultMouhotSchmeiser,ArmstrongMourrat,LuWang}, as well as probabilistic approaches that are mainly based on couplings, possibly in combination with drift/minorization conditions, see for example \cite{Lindvall,LindvallRogers, MeTw2009,hairer2006ergodicity,HairerMattinglyScheutzow,Eb2016A,eberle2019couplings,EberleMajka2019}.
A focus of current research has become to understand the mixing properties of Hamiltonian Monte Carlo (HMC) in representative models that exhibit both high-dimensionality and non-convexity. 


HMC is an MCMC method that is based on deterministic Hamiltonian dynamics combined with momentum randomizations \cite{DuKePeRo1987,Ne2011}. With few exceptions, the Hamiltonian dynamics is not analytically solvable, and therefore, it is normally discretized in time using an evenly spaced grid with time step size $h > 0$ and a numerical integrator; most often velocity Verlet  \cite{HaLuWa2010,LeRe2004,BoSaActaN2018,MaWe2001}.  
Time discretization leads to an error in the invariant measure of the HMC method based on Verlet, which can in principle be reduced by taking a smaller time step size or removed by adding a Metropolis accept-reject step per HMC step.  The latter gives rise to \emph{adjusted} HMC (or Metropolis-adjusted HMC, or Metropolized HMC), and the method without adjustment is called \emph{unadjusted} HMC (uHMC).  The method based on the exact Hamiltonian dynamics is called \emph{exact} HMC and is mainly used as a theoretical tool.   


In recent years, there has been considerable progress in developing quantitative Wasserstein convergence bounds for HMC. For strongly logconcave distributions, a synchronous coupling was used to derive $L^1$-Wasserstein convergence bounds for exact HMC in \cite{mangoubi2017rapid}, and this result was subsequently sharpened in \cite{chen2019optimal}.  These papers use perturbative arguments to study the effect of time discretization error. 
For more general target distributions, a coupling-based argument for unadjusted and adjusted HMC was developed in \cite{BoEbZi2020} to obtain contractivity in a carefully designed Wasserstein distance equivalent to the standard $L^1$-Wasserstein distance $\W^1$. By applying this coupling componentwise, this result has been extended to mean-field models \cite{BoSc2020}. Moreover, by using a two-scale coupling, the approach has also been extended to perturbations of Gaussian measures in infinite dimension \cite{BoEb2020}.




On the other hand, explicit total variation (TV) bounds for HMC are scarce. 
The few existing results \cite{mangoubi2017rapid,chen2020fast}
hold only under restrictive conditions on the 
target and/or the initial distribution
(strong logconcavity, warm start), and except for 
very special cases, the optimal order of 
upper bounds is unknown.
This paper makes a contribution towards filling this gap in the literature. It provides explicit
upper bounds on (i) a one step $\W^1$ to TV
regularization of the uHMC transition kernel,
(ii) the mixing time of the uHMC chain started 
at an arbitrary initial distribution with finite
first moment, and (iii) the TV bias
between the stationary distribution of uHMC and the true target distribution. 

Recall that by the
dual descriptions of the distances, TV bounds correspond to bounds for integrals
of arbitrary bounded measurable functions 
(``observables''), while
$\W^1$ bounds correspond to bounds for  
Lipschitz continuous functions. Thus, one
advantage of TV bounds is that they require less
regularity of the observables and in particular 
apply to indicator functions. Moreover, in contrast to $\W^1$ distances, the TV distance is
scale invariant. If the transition kernels of
a Markov chain have strong smoothing properties,
then it is not difficult to deduce TV bounds
from $\W^1$ bounds. For HMC, however, such regularizing properties and their precise  dependence on the dimension are not obvious.

Although the results are stated below in a slightly different way, the main idea underlying the mixing time bounds for uHMC derived in this work is the construction of a successful coupling of two copies of uHMC starting from different initial conditions. This coupling builds on recently introduced couplings for uHMC which bring the two copies arbitrarily close in representative models having both non-convexity and high dimension \cite{BoEbZi2020,BoSc2020}.  After the two copies are close, a ``one-shot coupling'' is used to get them to coalesce with high probability \cite{madras2010quantitative}.   
This strategy works provided for small distances between the two copies, there is a large overlap between their corresponding distributions in the next step.
A key element of our proofs is correspondingly
a precise upper bound on this overlap, which
is equivalent to a $\W^1$ to TV
regularization bound for the uHMC transition kernel, see Lemmas \ref{lem:overlap1} and \ref{lem:regularization} below. 
The regularizing effect stems from the initial velocity randomization in each transition step
of the uHMC chain. We stress that it is not trivial to quantify the overlap precisely because each 
HMC step involves many deterministic moves.
A similar overlap bound can also
be applied in combination with a triangle inequality trick and existing bounds for the
Wasserstein bias in order to quantify the total
variation bias of the invariant measure,
see Lemmas \ref{lem:overlap2} and \ref{lem:tv_strong_error}.

Our work is closely related to the recent work 
by Chen, Dwivedi, Wainwright and Yu
\cite{chen2020fast}. In that paper, conductance methods are applied to prove TV convergence bounds for adjusted HMC assuming a warm start and that either an isoperimetric or a log-isoperimetric inequality holds. 
An important ingredient in their proofs is an overlap bound that is similar to the more refined bounds we develop for the one-shot coupling. 
For a warm start, it is shown in
\cite{chen2020fast} that an $\varepsilon$-accurate approximation can be achieved by adjusted HMC
with $O(d^{11/12} \log(\varepsilon^{-1}) )$
gradient evaluations for strongly log-concave target distributions, and with $O(d^{4/3} \log(\varepsilon^{-1}) )$ gradient evaluations for weakly non-log-concave target distributions. 
The convergence bounds for uHMC stated below 
have a substantially better dependence on the dimension $d$ and on the initial law. In particular,
we show that the mixing time of the uHMC Markov chain often depends only logarithmically on
$d$. 
The price to pay is that both the required step size and, correspondingly, the required number of gradient evaluations per transition step of the Markov chain, depend
substantially on the accuracy $\varepsilon$.
Unfortunately, it is 
currently not clear how to develop 
a coupling approach that
can provide sharp bounds for adjusted HMC.
The problem is that straightforward couplings are not efficient in combination with accept-reject
steps. See however the recent work \cite{chewi2020optimal} for promising first 
steps in this direction.

A related strategy as in our work has also recently been implemented for an ``OBABO'' discretization of second-order Langevin dynamics; see Proposition 3 and Proposition 22 of \cite{monmarche2020high}.  In this case, the proofs simplify substantially because they involve designing a one-shot coupling based on only one step of the OBABO scheme instead of many Verlet steps. The price to pay is that the resulting bounds in \cite{monmarche2020high} are
less sharp because in contrast to Lemma \ref{lem:overlap1} below, the overlap bound for OBABO applies only if the two copies are very close to each other (within distance of order $O(h^{1/2})$, where $h$ is the time step size).







\medskip


We now outline the main results of this paper.  Let $\pi(x,dy)$ denote the one-step transition kernel of exact HMC with fixed integration time $T > 0$ of the Hamiltonian flow during each transition step, and let $\tilde{\pi}(x,dy)$ denote the one-step transition kernel of uHMC operated with integration time $T$ and discretization time step size $h > 0$ satisfying $T \in h \mathbb{Z}$.   Denote by $\mu$ and $\tilde{\mu}$ the corresponding invariant probability measures of exact HMC and unadjusted HMC, respectively, and let $\TV{\nu}{\eta}$ denote the TV distance between two probability measures $\nu, \eta$ on $\mathbb{R}^d$.

The main result of this paper is a quantitative bound on $\TV{\tilde \mu}{\nu \tilde{\pi}^{m+1}}$ that holds for any $m \ge 0$ and any probability measure $\nu$ on $\mathbb{R}^d$. More precisely, by using a one-shot coupling we first obtain \begin{align}
\TV{\tilde{\mu}}{\nu \tilde{\pi}^{m+1}} \ \le \ (3/2) \, \left( T^{-2}  + 27 \, d \, L_H^2 \, T^4 \right)^{1/2}  \, \W^1(\tilde{\mu},\nu \tilde{\pi}^{m})
\end{align}
where $\W^1$ is the standard $L^1$-Wasserstein distance.  Therefore, if the uHMC chain converges geometrically in $\W^1$, i.e., 
$\W^1(\tilde{\mu},\nu \tilde{\pi}^{m}) \ \le \ M_1 e^{-c m} \W^1(\tilde \mu, \nu)$, then
\begin{align} \label{eq:uHMC_tv}
   \TV{\tilde{\mu}}{\nu \tilde{\pi}^{m+1}} \ \le \ (3/2) \, \left( T^{-2}  + 27 \, d \, L_H^2 \, T^4 \right)^{1/2}  \, M_1 \, e^{-c m}  \, \W^1(\tilde \mu, \nu) \;.
\end{align}  
As long as the constant $M_1$ and the $L^1$-Wasserstein distance between the initial distribution of the chain and the invariant measure of uHMC $\W^1(\tilde \mu, \nu)$
depend polynomially on the dimension $d$, \eqref{eq:uHMC_tv} implies that the mixing time of the uHMC chain depends at most logarithmically on $d$. Alternatively, it would also be possible to rephrase the bound in \eqref{eq:uHMC_tv} in the spirit of 
\cite{EberleMajka2019} as a contraction bound in a Kantorovich distance that has both a TV and a Wasserstein part.

In a second step, we quantify the TV distance between $\tilde \mu$ and $\mu$. By the triangle inequality 
\begin{align}
\TV{\mu}{\tilde{\mu} } \ &\le \ \TV{\mu \pi }{\mu \tilde{\pi} } + \TV{\mu \tilde{\pi} }{\tilde{\mu} \tilde{\pi} }  \nonumber \\
\ &\le \ \TV{\mu \pi }{\mu \tilde{\pi} } + (3/2) \, \left( T^{-2}  + 27 \, d \, L_H^2 \, T^4 \right)^{1/2}  \, \W^1(\mu, \tilde \mu) \;.
\end{align} 
By strong accuracy of the underlying integrator, i.e., $\W^1( \mu, \tilde{\mu}) \le M_2 \, h^2$,  and using another one-shot coupling to estimate  $\TV{\mu \pi }{\mu \tilde{\pi} }$, we obtain \begin{align}
\label{eq:uHMC_im}
   \TV{\mu}{\tilde{\mu} } \ \le \  C \, h^2 +  (3/2) \, \left( T^{-2}  + 27 \, d \, L_H^2 \, T^4 \right)^{1/2}  \, M_2 \, h^2   \;, 
\end{align} 
where $C$ is a constant given in \eqref{eq:C} which grows like $d^{3/2}$ for general $U$ and $d^{1/2}$ for quadratic $U$; and the constant $M_2$ grows like $d$ for general $U$ and $d^{1/2}$ for quadratic $U$.  For mean-field $U$, the corresponding upper bound on $\TV{\mu}{\tilde{\mu} }$ grows like $O(d h^2)$.  Applying the triangle inequality again, and inserting \eqref{eq:uHMC_tv} and \eqref{eq:uHMC_im} gives an overall convergence bound on $\TV{\mu}{\nu \tilde{\pi}^{m+1}}$. These results are stated  precisely in Theorems \ref{thm:tv_to_equilibrium} and
\ref{thm:tv_IM_accuracy}, and Corollary~\ref{cor:tv_to_target} in the next section for general $U$, and in Theorems \ref{thm:tv_to_equilibrium_mf} and
\ref{thm:tv_IM_accuracy_mf}, and Corollary~\ref{cor:tv_to_target_mf} for mean-field $U$.  The remaining sections contain detailed proofs.  In the special case $T=h$, i.e., one integration step per HMC step, one recovers TV convergence bounds for the unadjusted Langevin algorithm (uLA) \cite{durmus2019}.  
In this case, \eqref{eq:uHMC_tv} shows that the mixing time again depends logarithmically on  $d$, while \eqref{eq:uHMC_im} shows that the accuracy of the invariant measure is $O(d h)$ for general $U$.  
  
  \medskip
  
We conclude this introduction by remarking that mixing time bounds based on coupling methods might be relevant to recently developed \emph{unbiased} estimators based on couplings  \cite{heng2019unbiased}.  Both in theory and in practice, the usefulness of these unbiased estimators requires a successful coupling that realizes these bounds.  Therefore, to understand the performance of these unbiased estimators, it is crucial to obtain quantitative TV convergence bounds and geometric tail bounds on the corresponding coupling times.





\section{Main Results}

\subsection{Notation}

Let $\mathcal{P}(\mathbb{R}^d)$ denote the set of probability measures on $\mathbb{R}^d$.   Define the total variation (TV) distance between $\nu, \eta \in \mathcal{P}(\mathbb{R}^d)$  by \[
\TV{\nu}{\eta} \ := \ \sup \{ |\nu(A) - \eta(A)| : A \in \mathcal{B}(\mathbb{R}^d) \}
\] where $\mathcal{B}(\mathbb{R}^d)$ is the Borel $\sigma$-algebra on $\mathbb{R}^d$. Denote the set of all couplings of $\nu, \eta \in \mathcal{P}(\mathbb{R}^d)$ by $\J(\nu,\eta)$. A useful property of the TV distance is the following coupling characterization \begin{equation} \label{eq:tv_coupling}
\TV{\nu}{\eta} \ = \ \inf\Big\{P[X\neq Y]~:~ \law(X, Y) \in \J(\nu,\eta) \Big\} \;.
\end{equation}  
Let $\mathsf{d}$ be a metric on $\mathbb{R}^d$.
For $\nu,\eta \in \mathcal{P}(\mathbb{R}^d)$,
define the $L^1$-Wasserstein distance with respect to $\mathsf{d}$ by
\[
\W_{\mathsf{d}}^1(\nu,\eta) \ := \ \inf \Big\{ E\left[\mathsf{d}(X,Y)  \right] ~:~ \law(X, Y) \in \J(\nu,\eta) \Big\} \;.
\] In the special case where $\mathsf{d}$ is the standard Euclidean metric, we write the corresponding $L^1$-Wasserstein distance as $\W^1$.

\subsection{Short Overview of Unadjusted Hamiltonian Monte Carlo}

Unadjusted Hamiltonian Monte Carlo (uHMC) is an MCMC method for approximate sampling from a `target' probability distribution on $\mathbb{R}^d$ of the form
\begin{equation}\label{eq:mu}
\mu (dx)= \mathcal Z^{-1}\,\exp (-U(x))\, dx \;, \quad \text{$\mathcal Z=
\int\exp (-U(x))\,dx$}  \;,
\end{equation}
where $U: \mathbb{R}^{d} \to \mathbb{R}$ is
a twice continuously differentiable function
such that $\mathcal Z<\infty$. The function $U$ is
interpreted as a potential energy.  
The uHMC algorithm generates a Markov chain on $\mathbb{R}^d$ with the help of: (i) a sequence $(\xi_k)_{k \in \mathbb{N}_0}$ of i.i.d.~random variables $\xi_k \sim \mathcal N(0,I_d)$; and (ii) the velocity Verlet integrator with time step size $h > 0$ and initial condition $(x,v) \in \mathbb{R}^{2d}$ whose discrete solution takes values on an evenly spaced temporal grid $\{ t_i := i h \}_{i \in \mathbb{N}_0}$ and is interpolated by $(\tilde{q}_t(x,v), \tilde{v}_t(x,v))$ which satisfies the following differential equations
\begin{equation}
\label{velVerlet}
	\frac{d}{dt} \tilde{q}_t   \, = \,   \tilde{v}_{\lb{t}} - (t-\lb{t})  \nabla U(\tilde{q}_{\lb{t}})  , \qquad   \frac{d}{dt} \tilde{v}_t \, = \, -\frac{1}{2}\left(\nabla
	U(\tilde{q}_{\lb{t}}) + \nabla U(\tilde{q}_{\ub{t}})\right)
\end{equation}
with initial condition $(\tilde{q}_0(x,v),\tilde{v}_0(x,v)) = (x,v) \in \mathbb{R}^{2d}$. Here we have defined \begin{equation}
\label{eq:round}
\lb{t}= \max\{ s\in h\Z \ : \ s\leq t \} \mbox{\quad and\quad} \ub{t}= \min\{ s\in
h\Z \ : \ s\geq t \} \quad \text{ for $h>0$} \;.
\end{equation}
For $h=0$, we set $\lb{t}=\ub{t}=t$ and drop the tildes in the notation. Thus
\begin{equation}
    \label{eq:Hamflow}
     \frac{d}{dt} {q}_t  \,  =\,     {v}_t ,\qquad  \frac{d}{dt} {v}_t \,  =\, - \nabla
	U({q}_t),
\end{equation}
and
correspondingly, $({q}_t(x,v), {v}_t(x,v))$ is the exact Hamiltonian flow
with respect to the unit mass Hamiltonian $H(x,v) = (1/2) \norm{v}^2 + U(x)$.\smallskip 
\\
\begin{minipage}{\textwidth}
\begin{minipage}{0.50\textwidth}
\noindent
By integrating \eqref{velVerlet}, note that $\tilde{q}_t$ is a piecewise quadratic function of time that interpolates between the points $\{ (t_k, \tilde{q}_{t_k}) \}$ and satisfies 
$\left. \frac{d}{dt} \tilde{q}_t  \right|_{t=t_k+} = \tilde{v}_{t_k}$, while $\tilde{v}_t$ is a piecewise linear function of time that interpolates between the points $\{ (t_k, \tilde{v}_{t_k} ) \}$.  In general, $\tilde{q}_t$ does not satisfy $\left. \frac{d}{dt} \tilde{q}_t  \right|_{t=t_k+} = 
\left. \frac{d}{dt} \tilde{q}_t  \right|_{t=t_k-}$, as illustrated in the figure.
In the analysis of the convergence properties 
of uHMC, this continuous-time interpolation of the discrete solution produced by velocity Verlet is
convenient to work with.

\end{minipage}
\begin{minipage}{0.48\textwidth}
\begin{center}
\begin{tabular}{c}
\includegraphics[width=\textwidth]{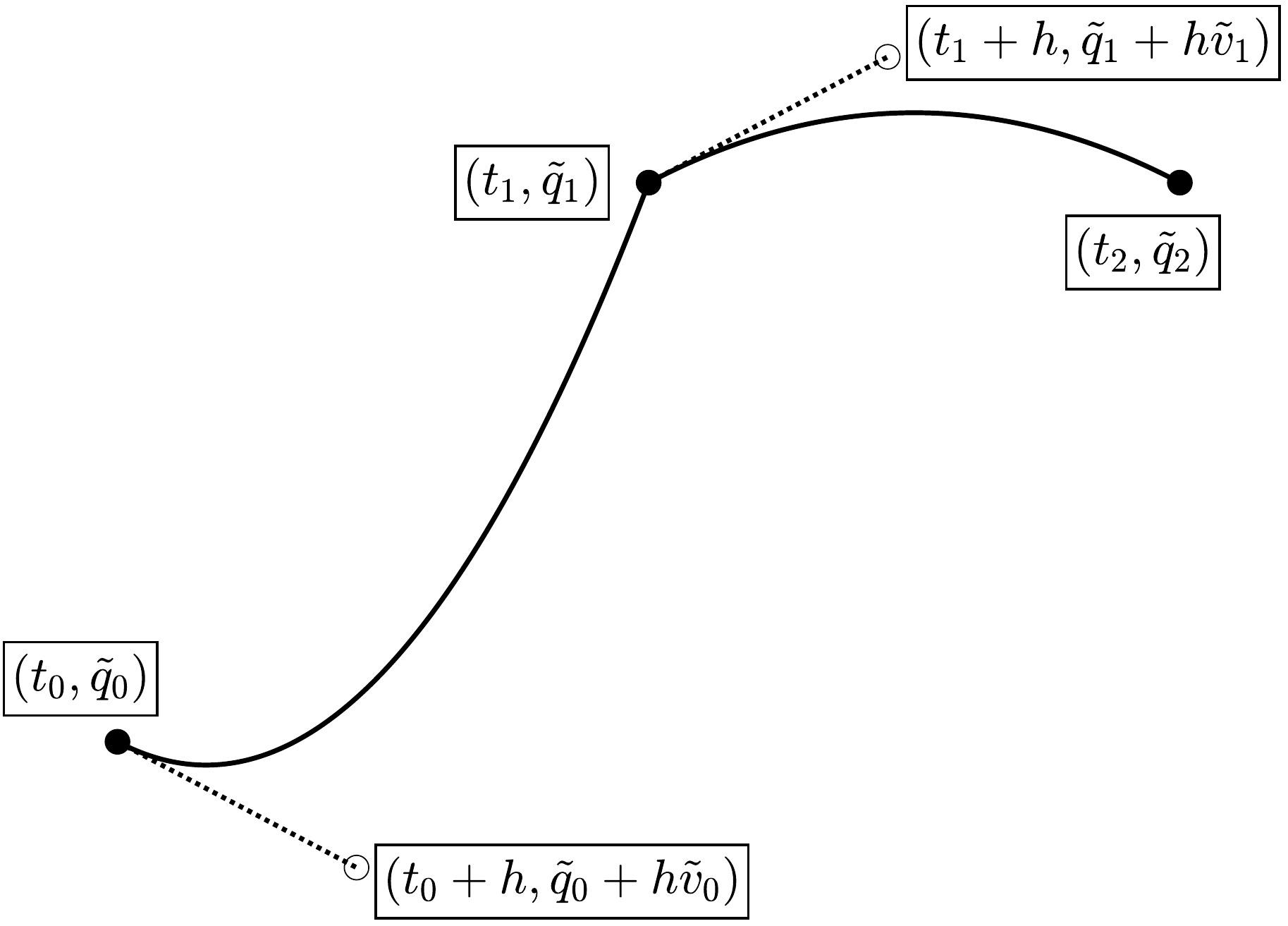} \\
{\em piecewise quadratic} \\ {\em interpolation of positions}
\end{tabular}
\end{center}
\end{minipage}
\end{minipage}

\smallskip

In the $n$-th step of the \emph{unadjusted Hamiltonian Monte Carlo algorithm with complete velocity refreshment and duration parameter $T\in (0,\infty )$}, the initial velocity $\xi_{n-1}$ is sampled independently of the previous development, and the current position and initial velocity are evolved by applying the Verlet approximation of the Hamiltonian flow over a 
time interval of length $T$.

\begin{defn}[\textbf{uHMC Markov chain}]
Given an initial state $x \in \mathbb{R}^d$, duration $T>0$, and time step size $h \ge 0$ with $T/h \in \mathbb{Z}$ for $h \ne 0$, define $\tilde{X}_0(x) := x$ and \[
\tilde{X}_n(x) := \tilde{q}_T( \tilde{X}_{n-1}(x), \xi_{n-1} ) \quad \text{for $n \in \mathbb{N}$} \;.
\] 
Let $\tilde{\pi}(x,A) = P[ \tilde{X}_1(x) \in A ]$ denote the corresponding one-step transition kernel.  
\end{defn}

For $h=0$, we recover the \emph{exact Hamiltonian Monte Carlo algorithm}. In this case, we drop all tildes in the notation, i.e., the $n$-th transition step is denoted by $X_n(x)$, and the corresponding transition kernel is denoted by $\pi$. The target measure $\mu$ is invariant under $\pi$, because the Hamiltonian flow \eqref{eq:Hamflow} preserves the probability measure on $\mathbb R^{2d}$ with density proportional to $\exp (-H(x,v))$, and $\mu$ is the first marginal of this measure. When $h>0$, the invariant probability measure for $\tilde{\pi}$ is denoted by $\tilde \mu$. In general, it does not agree with $\mu$, but when it exists, it typically approaches $\mu$ as $h\downarrow 0$.

\subsection{Assumptions}

Let $\mathbf{H} \equiv \differential^2 U$ denote the Hessian of $U$.
To prove our main results, we assume:

\begin{assumption} 
\label{A1234} The function $U: \mathbb{R}^d \to \mathbb{R}$ satisfies the following conditions:
\begin{itemize}
\item[(A1)] $U$ has a global minimum at $0$ and $U(0)=0$.
\item[(A2)] $U \in \mathcal{C}^2(\mathbb{R}^d)$ with bounded second derivative, i.e., $L:=\sup \mnorm{ \mathbf{H} } < \infty$.
\item[(A3)] $U \in \mathcal{C}^3(\mathbb{R}^d)$ with bounded third derivative, i.e., $L_H :=\sup \mnorm{ \differential\mathbf{H} } < \infty$.
\item[(A4)] $U \in \mathcal{C}^4(\mathbb{R}^d)$ with bounded fourth derivative, i.e., $L_I :=\sup \mnorm{ \differential^2\mathbf{H} } < \infty$.
\end{itemize}
Here $\mnorm{\cdot}$ denotes the operator norm of a multilinear form. 
\end{assumption}

\begin{rem}[Choice of norm and dimension dependence]
The dimension dependence of the subsequent results depends on the choice of norm used in Assumptions (A2)-(A4).
By the Cauchy-Schwarz inequality, note that the operator norm $\mnorm{\cdot}$ of a $k$-multilinear form $A = (a_{i_1 \cdots i_k})$ is always bounded by the Frobenius (or Hilbert-Schmidt) norm that is defined by $\mnorm{A}_F=(\sum_{i_1,\dots,i_k=1}^{d,\dots,d} a_{i_1 \cdots i_k}^2)^{1/2}$, i.e.,  $\mnorm{A} \le \mnorm{A}_F$, see \cite[Section 2]{friedland2018nuclear}. 
Choosing the Frobenius norm in Assumptions (A2)-(A4) would give an improved dimension dependence in the results below.
\end{rem}

We additionally assume that the transition kernel $\tilde{\pi}$ of uHMC satisfies the following $L^1$-Wasserstein convergence bound and discretization error bound.

\begin{assumption} 
\label{A56}
There exists an invariant probability measure $\tilde \mu$ of $\tilde{\pi}$ satisfying the following conditions.
\begin{itemize}
\item[(A5)] There exist constants $M_1,c\in (0,\infty )$ such that
for any $m \in \mathbb{N}$ and $\nu \in \mathcal{P}(\mathbb{R}^d)$,  $\W^1(\nu \tilde{\pi}^m, \tilde{\mu}) \, \le \, M_1 \, e^{-c m} \, \W^1(\nu, \tilde\mu )$. 
\item[(A6)]
There exists a constant $M_2\in (0,\infty )$ such that $\W^1(\mu, \tilde{\mu}) \, \le \, M_2 \, h^2$.
\end{itemize}
\end{assumption}
Note that by (A5), the invariant probability measure is unique.
Assumption (A6) often holds as a consequence of (A5), a triangle-inequality trick \cite[Remark 6.3]{mattingly2010convergence} and strong accuracy of the Verlet integrator; see~Section 3.2 of \cite{BoSc2020}. 
In general, the constants $M_1$, $M_2$ and $c$ appearing in (A5) and (A6) will depend on the dimension $d$, but usually they can be chosen independently of the discretization step size $h$. In Section \ref{sec:examples}, we will see 
several classes of examples where $c$ can be chosen independently of $d$, and the dimension
dependence of $M_1$ and $M_2$ is explicit.

\subsection{TV convergence bounds for uHMC}

We are now in position to state our main results.
For all of the following results, we assume that
Assumptions~\ref{A1234} and \ref{A56} are satisfied. Moreover, we fix a duration parameter $T>0$ and a time step size $h>0$ such that
$T/h\in\mathbb N$ and 
\begin{equation}
    \label{eq:Tconstraint}
    L(T^2 + T h) \ \le\ 1/6 \, .
\end{equation}

\begin{thm} \label{thm:tv_to_equilibrium}
For any $m \in \mathbb{N}_0$ and for any initial law
$\nu\in\mathcal P(\mathbb R^d)$,
\begin{equation} \label{eq:tv_to_equilibrium}
\TV{\tilde \mu}{\nu \tilde{\pi}^{m+1}} \ \le \  
(3/2) \, \left( T^{-2}  + 27 \, d \, L_H^2 \, T^4 \right)^{1/2} \, M_1 \, \W^1(\tilde \mu, \nu) \, e^{-c \, m}  \;.   
\end{equation}
\end{thm}

\begin{proof} For $m=0$, the statement is a 
special case of a $\W^1 $/TV regularization
result that is proven in
Lemma~\ref{lem:regularization} below. The general case follows by combining Lemma~\ref{lem:regularization} and Assumption~\ref{A56} (A5).
\end{proof}

Theorem~\ref{thm:tv_to_equilibrium} immediately implies an upper bound on the $\varepsilon$-mixing time 
$t_{\mathrm{mix}}(\varepsilon ,\nu ):=
\inf\left\{ m\ge 0:\TV{\tilde \mu}{\nu \tilde{\pi}^{m}} \le \varepsilon\right\}$
of the uHMC Markov chain.

\begin{cor}[Upper bound for mixing time]\label{cor:mixingtime}
For any $\varepsilon>0$ and any
$\nu\in\mathcal P(\mathbb R^d)$, 
\[
t_{\mathrm{mix}}(\varepsilon ,\nu )\ \le\  2 + \frac{1}{c}  \log\Big( \frac{3 \left( T^{-2}  + 27 \, d \, L_H^2 \, T^4 \right)^{1/2} \, M_1 \, \W^1(\tilde \mu, \nu)}{2 \varepsilon} \Big) \;.
\] 
\end{cor}
In particular, provided that $M_1$ and $\W^1(\tilde \mu, \nu)$ depend at most polynomially on $d$ and $c$ is independent of $d$, then the mixing time depends at most logarithmically on $d$ and $\varepsilon$. 
Since simulating one step of the uHMC chain
requires carrying out $T/h$ gradient evaluations,
an $\varepsilon$-approximation of $\tilde\mu$ 
in total variation distance can be achieved with 
$O\left(h^{-1}\log (d/\varepsilon )\right)$
gradient evaluations. 
However,
in order to control the approximation error with respect to the target distribution $\mu$, we also
have to take into account the systematic error
$\TV{\mu}{\tilde \mu}$.

\begin{thm} \label{thm:tv_IM_accuracy}
Under the assumptions made above,
\begin{align} \label{eq:tv_IM_accuracy} 
&\TV{\mu}{\tilde \mu} \ \le \  h^2 \, \Big[ (3/2)  \, \left( T^{-2}  + 27 \, d \, L_H^2 \, T^4 \right)^{1/2}  \, M_2  +  C \Big] ,\;\quad\text{where} \\
\begin{split} \label{eq:C}
& C \ := \ \Big[  d^3 \big( 4 L_I^2 T^4 + 14 L_H^2 L_I T^6 + 14 L_H^4 T^8 \big)  \\
& \qquad+ d^2 \big( 35 L_H^2 T^2 + 8 L_I^2 T^4 + 28 L_H^2 L_I T^6 + 28 L_H^4 T^8 \big)  \\
& \qquad+ d \big( 16 L^2 + 4 L_H^2 T^2) + \big(2 d L_H^2  + L^2 T^{-2}  \big) \int \norm{x}^2 \mu(dx)  \\ 
& \qquad + \big( d L_I^2 + d L_H^2 L_I T^2 + d L_H^4 T^4 + L_H^2 T^{-2} \big)  \int \norm{x}^4 \mu(dx)  \Big]^{1/2}  \;.
\end{split}
\end{align} 
\end{thm}

\begin{proof}
By the triangle inequality, \begin{equation} \label{tri::ineq}
\TV{\mu}{\tilde \mu} \ \le \ \TV{\mu \pi}{\mu \tilde \pi} +  \TV{\mu \tilde \pi}{\tilde \mu \tilde \pi} \;.
\end{equation}
Moreover,  by Lemma~\ref{lem:tv_strong_error} below, \begin{align}
& \TV{\mu \pi}{\mu \tilde \pi} \ \le \ C\,   h^2  \;. \label{mupi_mutpi}
\end{align}
Moreover, by Lemma~\ref{lem:regularization}, and by Assumption~\ref{A56} (A6), \begin{align}
\TV{\mu \tilde \pi}{\tilde \mu \tilde \pi} &\ \le\  (3/2) \, \left( T^{-2}  + 27 \, d \, L_H^2 \, T^4 \right)^{1/2}  \, \W^1(\mu, \tilde \mu) \nonumber \\
&\ \le\ (3/2)  \, \left( T^{-2}  + 27 \, d \, L_H^2 \, T^4 \right)^{1/2} \, M_2 \, h^2 \;. \label{mutpi_tmutpi}
\end{align}
Inserting \eqref{mutpi_tmutpi} and \eqref{mupi_mutpi} into \eqref{tri::ineq} gives \eqref{eq:tv_IM_accuracy}.  
\end{proof}

By the triangle inequality, note that \begin{equation}
\label{tri:::ineq}
\TV{\nu \tilde{\pi}^{m+1}}{\mu}\ \le\ \TV{\nu \tilde{\pi}^{m+1}}{\tilde \mu} + \TV{\tilde \mu}{\mu} \;.
\end{equation} Inserting \eqref{eq:tv_IM_accuracy} and \eqref{eq:tv_to_equilibrium} into \eqref{tri:::ineq} gives the following corollary.

\begin{cor} \label{cor:tv_to_target}
For any $m \in \mathbb{N}_0$ and any initial law $\nu\in \mathcal P(\mathbb R^d)$,
\begin{equation} \label{eq:tv_to_target}
\begin{aligned}
 \TV{\mu}{\nu \tilde{\pi}^{m+1}} \, &\le \,  
(3/2) \, \left( T^{-2}  + 27 \, d \, L_H^2 \, T^4 \right)^{1/2} \, M_1 \, \W^1(\tilde \mu, \nu) \, e^{-c \, m}  \\
& \quad + h^2 \, \Big[ (3/2) \, \left( T^{-2}  + 27 \, d \, L_H^2 \, T^4 \right)^{1/2}   \, M_2  + C\Big] .
\end{aligned}
\end{equation}
\end{cor}

The constant $C$ appearing in \eqref{eq:tv_to_target} is typically of order $O(d^{3/2})$, although in the Gaussian case the
dimension dependence improves further. Below, we will see examples where the constant $M_2$ appearing in \eqref{eq:tv_to_target} is 
of order $O(d)$. Then the TV accuracy
$\TV{\mu}{\tilde \mu}$ is at most of order
$O(h^2d^{3/2})$, and correspondingly, for an
$\varepsilon$-accurate approximation of $\mu$
by $\tilde\mu$, the step size $h$ has to be chosen 
of order $O(d^{-3/4}\varepsilon^{1/2})$. Thus if additionally,
$c$ is independent of the dimension, and $M_1$ depends at most polynomially on $d$, then an $\varepsilon$-accurate approximation of the target distribution $\mu$ 
in TV distance can be achieved 
by uHMC with 
$O\left(d^{3/4}\varepsilon^{-1/2}\log (d/\varepsilon )\right)$
gradient evaluations.

\subsection{Examples}\label{sec:examples}

  By coupling methods, (A5) has been verified in the following models.



\subsubsection{Asymptotically Strongly Logconcave Target}
\label{ex:asymp_cvx_A5}
Suppose that $U: \mathbb{R}^d \to \mathbb{R}$ satisfies (A1) and (A2), and $U$ is strongly convex outside a Euclidean ball, i.e., there exist constants 
$\mathcal{R} \in [0,\infty)$ and $K \in (0, \infty)$ such that
\[    (x-y) \cdot ( \nabla U(x)-\nabla U(y) )  \ \geq \   K \norm{x-y}^2  \quad \text{for all $x,y \in \mathbb{R}^d$ with  $\norm{x-y} \geq \mathcal{R}$.} \]
Using a synchronous coupling, and assuming that $T>0$ satisfies $L T^2 \le 1/4$, Chen and Vempala (2019) \cite{chen2019optimal} proved in the
globally strongly logconcave case (i.e., for $\mathcal R=0$) that 
for all initial distributions $\nu, \eta \in \mathcal{P}(\mathbb{R}^d)$, and for all $m\ge 0$, the transition kernel of exact HMC satisfies \begin{align*}
& \W^1 (\nu \pi^m ,\eta \pi^m )\ \le\  e^{- c  m } \, \W^1 (\nu ,\eta)\quad\text{where}  \quad c\ =\ K T^2 / 10.
\end{align*}  Mangoubi and Smith (2017) \cite{mangoubi2017rapid} obtained a similar result under $L T^2 \le \min(1/4, K/L)$.
Simple counterexamples demonstrate that a synchronous coupling is not contractive in the non-logconcave setting where $\mathcal R>0$.  Recently, a non-synchronous coupling tailored to HMC was introduced to obtain a corresponding result for non-convex potentials.
Suppose that $T >  0$ and $h_1 
\ge 0$ satisfy \[
L (T+h_1)^2 \le \min \left( \frac{3 K}{10 L},\frac{1}{4},\frac{3 K}{256 \cdot 5 \cdot 2^6\,L\mathcal R^2 (L+K)} \right) \quad \text{and} \quad 
h_1 \le \frac{K T}{525 L + 235 K} \;.
\]
Assuming that $h>0$ satisfies  $h \in [0, h_1]$,
Bou-Rabee, Eberle and Zimmer (2020) \cite{BoEbZi2020} essentially prove that
for all $m\ge 0$ and all $\nu, \eta \in \mathcal{P}(\mathbb{R}^d)$,
\begin{align*}
& \W^1 (\nu \tilde{\pi}^m ,\eta \tilde{\pi}^m )\ \le\ M_1 \, e^{- c m} \, \W^1 (\nu ,\eta) ,\qquad\text{where} \\
& M_1 \ = \ \exp\left( \frac{5}{2} ( 1 + \frac{4 \mathcal{R}}{T} \sqrt{\frac{L+K}{K}}) \right) \quad\text{and}\quad  c\ =\  \frac{K T^2}{156} \exp\left( -10 \frac{\mathcal{R}}{T} \sqrt{\frac{L+K}{K}} \right).
\end{align*}
Intuitively speaking, the factor $L \mathcal{R}^2$ measures the degree of non-convexity of $U$. 
The bounds show that (A5) is satisfied with
the explicit constants given above.  Now suppose additionally that
$U$ is in $C^3(\mathbb{R}^{d})$ and 
$L_H=\sup\|\differential^3U\|<\infty$. Then by Corollary 7  of \cite{BoSc2020} with $n=1$ and $\epsilon=0$, we have \[
\W^1(\mu, \tilde \mu)  \le
 \frac{1}{c}  \tilde{C}_2 M_1 \Big( d + \int   \norm{x} \mu(dx) + \int \norm{x}^2 \mu(dx) \Big) \; h^2 
\] where $\tilde{C}_2$ depends only on $K$, $L$, $L_H$ and T.  Therefore,  Assumption (A6) holds with \[
M_2 =  \frac{1}{c}  \tilde{C}_2 M_1 \Big( d + \int  \norm{x} \mu(dx) + \int  \norm{x}^2 \mu(dx) \Big) \;.\]
If the constants $\mathcal R$, $K$ and $L$ are fixed, then $c$ and $M_1$ do not depend on the
dimension. 
As a consequence, by Corollary \ref{cor:mixingtime}, the mixing time is of order $O(\log d)$, and by
Theorem \ref{thm:tv_IM_accuracy},
the TV accuracy $\TV{\mu}{\tilde \mu}$
is of order $O(d^{3/2}h^2)$. In particular, an $\varepsilon$-accurate approximation
can be achieved with a step size of order
$O(\varepsilon^{-1/2}h^{-3/4})$.
The latter bound is not sharp for strongly logconcave product models with i.i.d.\ factors where one
can prove by elementary methods that the correct order of $\TV{\mu}{\tilde \mu}$ is $\Theta (d^{1/2}h^2)$. In the general setup considered
above, however, we do not
expect a bound for TV accuracy
of a similar order to hold. In a complimentary work,
Durmus and Eberle \cite{DurmusEberle2021} show that
the $L^1$-Wasserstein accuracy $\W^1(\mu ,\tilde\mu )$
is of order
$O(d^{1/2}h^2)$ for ``nice'' models and 
of order $O(dh^2)$ for general models satisfying
the assumptions made above.
\smallskip

In general, we can not expect that
constants $\mathcal{R}$, $K$ and $L$ as above can be chosen independently
of the dimension $d$.
Then also $c$ and $M_1$
may depend implicitly on the dimension through these constants. For example, in mean-field models, which we consider next, $\mathcal{R}$ can increase with the number of particles; see Remark 1 of \cite{BoSc2020}.

\subsubsection{Mean-Field Model} \label{ex:mean_field}
Consider a mean-field model \cite{kac1956,KacProgram2013,BoSc2020} consisting of $n$ particles 
in dimension $k$ with potential energy $U: \mathbb{R}^{ nk} \to \mathbb{R}$ defined as \begin{equation}
U(x) \ = \ \sum_{i=1}^n \Big( V(x^i) + \frac{\epsilon}{n} \sum_{\ell=1, \ell \ne i}^n W(x^i - x^{\ell}) \Big) \;, \quad x \ = \ (x^1, \dots, x^n) \;, \quad x^i \in \mathbb{R}^k \;. \label{mf}
\end{equation} We assume that $V,W$ are functions in $C^2(\mathbb R^k)$ satisfying: 
\begin{itemize}
 	\item $V$ has a local minimum at $0$, and $V(0)=0$;
 	\item $L=\sup\|\differential^2V\|<\infty$ and $\tilde{L}=\sup\|\differential^2W\|<\infty$; and,
 \item there exist constants 
$\mathcal{R} \in [0,\infty)$ and $K \in (0, \infty)$ such that
\[    (x^1-y^1) \cdot ( \nabla V(x^1)-\nabla V(y^1) ) \  \geq \    K \norm{x^1-y^1}^2   \]
for all $x^1,y^1 \in \mathbb{R}^k$ with  $\norm{x^1-y^1} \geq \mathcal{R}$.
 \end{itemize} 
Suppose that $T >  0$, $\epsilon \ge 0$ and $h_1 \ge 0$ satisfy \begin{align*}
L (T + h_1)^2 \ &\le \ \frac{3}{5} \min \left( \frac{3 K}{10 L},\frac{1}{4},\frac{3 K}{256 \cdot 5 \cdot 2^6\,L\mathcal R^2 (L+K)}\right) \;, \\
|\epsilon| \tilde{L} \ &< \ \min \left( \frac{K}{6} , \frac{1}{2} \left( \frac{K}{36 \cdot 149} \right)^2 \left( T + 8 \mathcal{R} \sqrt{\frac{L+K}{K}} \right)^2 \exp\left(-40 \frac{\mathcal{R}}{T} \sqrt{\frac{L+K}{K}}\right) \right) \;, \\
h_1 \ &\le \ \frac{K T}{525 L + 235 K} \;.
\end{align*}  Then for all initial distributions $\nu, \eta \in \mathcal{P}(\mathbb{R}^{nk})$, for all $m\ge 0$ and for any $h \in [0,h_1]$, Bou-Rabee and Schuh (2020) 
\cite{BoSc2020} use a component-wise coupling method to prove that \begin{equation}
\label{eq:W1ell1_mf}
\begin{aligned}
& \W_{\ell_1}^1 (\nu \tilde{\pi}^m ,\eta \tilde{\pi}^m )\, \le\, M \, e^{- c m} \, \mathcal \W_{\ell_1}^1 (\nu ,\eta) ,~~\text{where} ~~ \ell_1(x,y) \, := \,  \sum_{i=1}^n |x^i - y^i|, \\
& \quad M \, = \, \exp\left( \frac{5}{2} ( 1 + \frac{4 \mathcal{R}}{T} \sqrt{\frac{L+K}{K}}) \right) ~~\text{and}~~  c\, =\, \frac{K T^2}{156} \exp\left( -10 \frac{\mathcal{R}}{T} \sqrt{\frac{L+K}{K}} \right).
\end{aligned}
\end{equation}
Since $\norm{x - y} \le \ell_1(x,y) \le \sqrt{n} \norm{x - y}$, we obtain \[ 
 \W^1 (\nu \tilde{\pi}^m ,\eta \tilde{\pi}^m )\ \le\ M \, e^{- c m} \,  \W_{\ell_1}^1 (\nu ,\eta) \ \le\ \sqrt{n} \, M \, e^{- c m} \, \mathcal \W^1 (\nu ,\eta) \;.
\]
Thus Assumption (A5) holds with $M_1 = \sqrt{n}  M$ and $c$ as above. In particular,
$c$ is independent of the number of particles $n$, and therefore, the mixing time depends only logarithmically on $n$.  This result also implies that there exists a unique invariant probability measure $\tilde \mu$ of $\tilde{\pi}$; see Corollary 5 of \cite{BoSc2020}.  In this mean-field context, we can also obtain slightly better dimension-dependence in the TV convergence bounds for uHMC.  To state these bounds, we additionally assume that $V,W$ are functions in $C^3(\mathbb{R}^{k})$ such that
$L_H=\sup\|\differential^3V\|<\infty$ and $\tilde{L}_H=\sup\|\differential^3W\|<\infty$.  Moreover, analogous to \eqref{eq:Tconstraint}, we assume that $T>0$ and the time step size $h>0$ are such that
$T/h\in\mathbb N$ and
\begin{equation}
    \label{eq:Tconstraint_mf}
    (L + 4 \epsilon \tilde{L}) (T^2 + T h) \ \le\ 1/6\, .
\end{equation}

\begin{thm} \label{thm:tv_to_equilibrium_mf}
For any $m \in \mathbb{N}_0$ and for any initial law
$\nu\in\mathcal P(\mathbb R^{n k})$,
\begin{equation} \label{eq:tv_to_equilibrium_mf}
\TV{\tilde \mu}{\nu \tilde{\pi}^{m+1}} \ \le \  
\frac{3}{2} \, \left( T^{-2}  + 34 \, k \, (L_H + 8 \epsilon \tilde{L}_H)^2 \, T^4 \right)^{1/2} \, \, M \, \W_{\ell_1}^1(\tilde \mu, \nu) \, e^{-c \, m}  \;. 
\end{equation}
\end{thm}

\begin{proof} For $m=0$, \eqref{eq:tv_to_equilibrium_mf} is a 
special case of a $\W_{\ell_1}^1 $/TV regularization
result proven in
Lemma~\ref{lem:regularization_mf} below. The general case follows by combining Lemma~\ref{lem:regularization_mf} and \eqref{eq:W1ell1_mf}.
\end{proof}

As before, Theorem~\ref{thm:tv_to_equilibrium_mf} immediately implies an upper bound on the $\varepsilon$-mixing time 
$t_{\mathrm{mix}}(\varepsilon ,\nu )$
of the uHMC Markov chain applied to mean-field $U$.  We stress that the following upper bound typically depends on the number of particles $n$ through $\W_{\ell_1}^1(\tilde \mu, \nu)$, but this dimension dependence is typically polynomial in $n$, and therefore, the upper bound typically depends logarithmically on $n$.

\begin{cor}[Upper bound for mixing time]\label{cor:mixingtime_mf}
For any $\varepsilon>0$ and any
$\nu\in\mathcal P(\mathbb R^{n k})$, 
\[
t_{\mathrm{mix}}(\varepsilon ,\nu )\ \le\ 2+ \frac{1}{c}  \log\bigg( \frac{3 \left( T^{-2}  + 34 \, k \, (L_H + 8 \epsilon \tilde{L}_H)^2 \, T^4 \right)^{1/2} \, M \, \W_{\ell_1}^1(\tilde \mu, \nu)}{2 \varepsilon} \bigg) \;. 
\] 
\end{cor}

Moreover, for mean-field $U$, we can obtain a better dimension-dependence in the upper bounds for the TV accuracy $\TV{\mu}{\tilde \mu}$.  To this end, we additionally assume that $V, W$ are functions in $C^4(\mathbb{R}^{k})$ such that 
$L_I=\sup\|\differential^4V\|<\infty$ and $\tilde{L}_I=\sup\|\differential^4W\|<\infty$.  By Corollary 7  of \cite{BoSc2020}, 
\begin{equation} \label{eq::W1ell1_mf}
    \begin{aligned}
& \W^1(\mu, \tilde \mu)\ \le\ \W_{\ell_1}^1(\mu, \tilde\mu ) \\
& \qquad \ \le\
 \frac{1}{c}  \tilde{C}_2 M \Big( n k + \sum_{\ell=1}^n \int \norm{x^{\ell}} \mu(d x) + \sum_{\ell=1}^n \int \norm{x^{\ell}}^2 \mu(d x)  \Big) \; h^2 
\end{aligned}
\end{equation} where $\tilde{C}_2$ depends only on $K$, $L$, $\mathcal{R}$, $\tilde{L}$, $L_H$, $\tilde{L}_H$,  
and T.  Therefore, Assumption (A6) holds with \[
M_2 \ =\  \frac{1}{c}  \tilde{C}_2 M \Big( n k + \sum_{\ell=1}^n \int \norm{x^{\ell}} \mu(d x) + \sum_{\ell=1}^n \int \norm{x^{\ell}}^2 \mu(d x)  \Big) \;.
\]
Note that $M_2$ depends linearly on the number of particles and on the sum of the first and second moments of the $n$ components of the target distribution.  


\begin{thm} \label{thm:tv_IM_accuracy_mf}
Under the assumptions made above,
\begin{align} \label{eq:tv_IM_accuracy_mf} 
&\TV{\mu}{\tilde \mu} \ \le \  h^2 \, \Big[ (3/2)  \, \left( T^{-2}  + 34 \, k \, (L_H + 8 \epsilon \tilde{L}_H)^2 \, T^4 \right)^{1/2}  \, M_2  +  C \Big] ,\;\quad\text{where} \\
\begin{split} \label{eq:C_mf} 
& C \ := \ \Big[ 
 n^2 k \bigg(  17 (L+4 \epsilon \tilde{L})^2 +
 28 (L_H +8 \epsilon \tilde{L}_H)^2 T^2
 + 
 104 k (L_H +8 \epsilon \tilde{L}_H)^2 T^2 
 \\  & 
 \qquad  + 180 (2 k  
 + k^2 ) T^4 \big( (L_I + 14 \epsilon \tilde{L}_I)  +  (L_H +8 \epsilon \tilde{L}_H)^2 T^2 \big)^2  \bigg) 
  \\ 
 & \qquad  + n \big( 10 k  (L_H +8 \epsilon \tilde{L}_H)^2 +  7 (L+4 \epsilon \tilde{L})^2 T^{-2} \big) \sum_{\ell=1}^n \int \norm{x^{\ell}}^2 \mu(d x) \\
 & \qquad + n \bigg( 40 k \big((L_I + 14 \epsilon \tilde{L}_I)^2 + (L_H +8 \epsilon \tilde{L}_H)^2 T^2 \big)^2  
 \\
&\qquad  +  7 (L_H +8 \epsilon \tilde{L}_H)^2 T^{-2}  \bigg) \sum_{\ell=1}^n \int \norm{x^{\ell}}^4 \mu(d x)  \Big]^{1/2} .
\end{split}
\end{align} 
\end{thm}

For mean-field $U$, note that the constant $C$ depends linearly on the number of particles $n$ provided that $n^{-1} \sum_{\ell=1}^n \int \norm{x^{\ell}}^4 \mu(d x)$ does not depend on $n$.  

\begin{proof}
By Lemma~\ref{lem:tv_strong_error_mf} below, \begin{align}
& \TV{\mu \pi}{\mu \tilde \pi} \ \le \ C\,   h^2  \;. \label{mupi_mutpi_mf}
\end{align}
Moreover, by Lemma~\ref{lem:regularization_mf}, and by \eqref{eq::W1ell1_mf}, \begin{align}
\TV{\mu \tilde \pi}{\tilde \mu \tilde \pi} &\ \le\  (3/2) \, \left( T^{-2}  + 34 \, k \, (L_H + 8 \epsilon \tilde{L}_H)^2 \, T^4 \right)^{1/2}  \, \W_{\ell_1}^1(\mu, \tilde \mu) \nonumber \\
&\ \le\ (3/2)  \, \left( T^{-2}  + 34 \, k \, (L_H + 8 \epsilon \tilde{L}_H)^2 \, T^4 \right)^{1/2} \, M_2 \, h^2 \;. \label{mutpi_tmutpi_mf}
\end{align}
Inserting \eqref{mutpi_tmutpi_mf} and \eqref{mupi_mutpi_mf} into \eqref{tri::ineq} gives \eqref{eq:tv_IM_accuracy_mf}.  
\end{proof}

 Inserting \eqref{eq:tv_IM_accuracy_mf} and \eqref{eq:tv_to_equilibrium_mf} into \eqref{tri:::ineq} gives the following corollary.

\begin{cor} \label{cor:tv_to_target_mf}
For any $m \in \mathbb{N}_0$ and any initial law $\nu\in \mathcal P(\mathbb R^{n k})$,
\begin{equation} \label{eq:tv_to_target_mf}
\begin{aligned}
 \TV{\mu}{\nu \tilde{\pi}^{m+1}} \ &\le \  
\frac{3}{2} \, \left( T^{-2}  + 34 \, k \, (L_H + 8 \epsilon \tilde{L}_H)^2 \, T^4 \right)^{1/2} \, M \, \W_{\ell_1}^1(\tilde \mu, \nu) \, e^{-c \, m}  \\
& \quad + h^2 \, \bigg[ \frac{3}{2} \, \left( T^{-2}  + 34 \, k \, (L_H + 8 \epsilon \tilde{L}_H)^2  \, T^4 \right)^{1/2}   \, M_2  + C\bigg] .
\end{aligned}
\end{equation}
\end{cor}

\section{Key ingredients of the proofs}

\subsection{Velocity Verlet as a Variational Integrator, Revisited}
\label{sec:vi}

 Here we recall a well-known variational characterization of the velocity Verlet integrator.  In particular, velocity Verlet can be derived from a discrete-time variational principle, and hence, is a variational integrator \cite{MaWe2001}.  A key tool in our analysis is a sufficient condition for convexity of the corresponding discrete action sum.  

\medskip

To this end, we introduce the \emph{discrete Lagrangian} $L_h: \mathbb{R}^d \times \mathbb{R}^d \to \mathbb{R}$ corresponding to velocity Verlet  \[
L_h(x, y) \ = \  \frac{h}{2} \left(  \frac{| y - x |^2}{h^2}  -  U(y ) -  U(x) \right) \;.  
\]
Given $\mathsf{a}, \mathsf{b} \in \mathbb{R}^d$, define the \emph{discrete path space} \[
\mathcal{C}_{h} = \left\{ \tilde{q}: \{t_k \}_{k=0}^{N} \to \mathbb{R}^d ~~\text{with endpoint conditions}~ \tilde{q}_{0} = \mathsf{a} ~\text{and}~ \tilde{q}_{T} = \mathsf{b} \right\} \;,
\]
and the \emph{discrete action sum} $S_h: \mathcal{C}_{h} \to \mathbb{R}$ by \[
S_h(\tilde{q}) = \sum_{k=0}^{N-1}  L_h(  \tilde{q}_{t_k}, \tilde{q}_{t_{k+1}} )  \;.  
\] Note that $t_0 =0$ and $t_N = T$.
Computing the directional derivative of $S_h$ in the direction of $u: \{t_k \}_{k=0}^{N} \to \mathbb{R}^d $ with $u_{0} = u_{T} = 0$ gives the first variation of $S_h$ \begin{align*}
\partial_u S_h(\tilde{q})  \ &= \  \left. \frac{\partial}{\partial \epsilon} S_h(\tilde{q} + \, \epsilon \, u) \right|_{\epsilon = 0} \\
\ &= \ \sum_{k=0}^{N-1} \left[ D_1 L_h(\tilde{q}_{t_k}, \tilde{q}_{t_{k+1}}) \cdot u_{t_k} + D_2 L_h(\tilde{q}_{t_k}, \tilde{q}_{t_{k+1}}) \cdot u_{t_{k+1}} \right]   \\
\ &= \ \sum_{k=1}^{N-1} \left( D_1 L_h(\tilde{q}_{t_k}, \tilde{q}_{t_{k+1}}) + D_2 L_h(\tilde{q}_{t_{k-1}}, \tilde{q}_{t_{k}}) \right)   \cdot u_{t_k}  
\end{align*} 
where in the last step we used summation by parts and  $u_{0} = u_{T} = 0$.  Stationarity of this action sum gives the discrete Euler-Lagrange equations \begin{equation}
    \label{eq:del}
     D_1 L_h(\tilde{q}_{t_k}, \tilde{q}_{t_{k+1}}) + D_2 L_h(\tilde{q}_{t_{k-1}}, \tilde{q}_{t_{k}}) = 0 \;,  \quad k \in \{1, \dots , N-1 \} \;.
\end{equation}  Introducing discrete velocities and simplifying yields the velocity Verlet integrator \begin{align*}
\begin{rcases*}
\tilde{v}_{t_{k}}  \ = \ - D_1 L_h (\tilde{q}_{t_{k}} , \tilde{q}_{t_{k+1}})  \\ 
\tilde{v}_{t_{k+1}}  \ = \  D_2 L_h (\tilde{q}_{t_{k}}, \tilde{q}_{t_{k+1}})  
\end{rcases*} \implies 
\begin{cases}
\tilde{q}_{t_{k+1}} \ = \ \tilde{q}_{t_{k}}+ h \tilde{v}_{t_{k}}  - ( h^2 / 2 ) \nabla U(\tilde{q}_{t_{k}})   \;,  \\
 \tilde{v}_{t_{k+1}} \ = \ \tilde{v}_{t_{k}} - (h/2) [ \nabla U(\tilde{q}_{t_k}) +   \nabla U(\tilde{q}_{t_{k+1}}) ] \;. \end{cases}  
\end{align*}
The following lemma indicates that for sufficiently short time intervals the discrete action sum $S_h$ corresponding to velocity Verlet is strongly convex.

\begin{lem}[M\"{u}ller \& Ortiz 2004 \cite{MueOr2004}]\label{lem:MuOr2004}
Suppose that (A2) holds.  For any $\mathsf{a},\mathsf{b} \in \mathbb{R}^d$, for any $h>0$ and $T>0$ satisfying $LT^2 \le (2/5) \pi^2$ and $T/h \in \mathbb{N}$, the discrete action sum $S_h$ is strongly convex.  
\end{lem}

\begin{proof}
The proof shows that the second derivative of the action sum is positive definite.  
The second derivative of the action sum is given by
\begin{align*}
& D^2 S_h(\tilde{q}) (u, u) \ = \  \left. \frac{\partial^2}{\partial \epsilon^2}  S_h(\tilde{q} + \epsilon \, u) \right|_{\epsilon= 0 } \ = \ I_1 + I_2 \;,
\end{align*}
where $I_1$ and $I_2$ are defined and bounded as follows. Using summation by parts, $u_0 = u_T = 0$, and $\| D^2 U \| < L$ (by assumption (A2)) yields
\begin{align*}
I_1 \ &:= \ 
\frac{-h}{2} \sum_{k=0}^{N-1} \left( D^2 U(\tilde{q}_{t_k})(u_{t_k},u_{t_k}) + D^2 U(\tilde{q}_{t_{k+1}})(u_{t_{k+1}},u_{t_{k+1}}) \right)  \\
\ &= \ 
-h \sum_{k=1}^{N-1}  D^2 U(\tilde{q}_{t_k})(u_{t_k},u_{t_k})    \ \ge \  - L \sum_{k=1}^{N-1} h |u_{t_k}|^2 \;.
\end{align*}
Applying a Poincar\'{e} inequality yields \[
I_2 \ := \   h \sum_{k=0}^{N-1}  \frac{|u_{t_{k+1}} - u_{t_k}|^2}{h^2} \ \ge \  \frac{\pi^2}{T^2} \left( \frac{\sin(\pi/ (2 N))}{\pi / (2 N)} \right)^2 \sum_{k=1}^{N-1}  h |u_{t_k}|^2 \ \ge \ \frac{2}{5} \frac{\pi^2}{T^2} \sum_{k=1}^{N-1}  h |u_{t_k}|^2 \;.
\]  Thus, $D^2 S_h( \tilde{q}) ( u, u)  \ge ( (2/5) (\pi^2 / T^2) - L)  \sum_{k=1}^{N-1}  h |u_{t_k}|^2 $.
\end{proof}

\medskip
As an immediate corollary to Lemma~\ref{lem:MuOr2004} we obtain.

\begin{cor} \label{cor:MuOr2004}
For any $\mathsf{a},\mathsf{b} \in \mathbb{R}^d$, and for any $h>0$ and $T>0$ satisfying $LT^2 \le (2/5) \pi^2$ and $T/h \in \mathbb{N}$, there exists a unique solution $\tilde{q}^{\star}: \{t_k \}_{k=0}^{N} \to \mathbb{R}^d$ of the discrete Euler-Lagrange equations \eqref{eq:del} with endpoint conditions $\tilde{q}_0^{\star} = \mathsf{a}$ and $\tilde{q}_T^{\star} = \mathsf{b}$.
\end{cor}

Note that $q^{\star}$ in Cor.~\ref{cor:MuOr2004} is the minimum of the corresponding action sum $S_h$.

\subsection{Overlap between Reference and Perturbed Gaussian Measure}

Let $\xi \sim \mathcal{N}(0,1)^d$ and $\Phi: \mathbb{R}^d \to \mathbb{R}^d$ be a differentiable, near identity map.    Here we provide a general upper bound for the TV distance between the \emph{reference} Gaussian measure $\law(\xi)$ and the \emph{perturbed} Gaussian measure $\law(\Phi(\xi))$.  The couple $(\xi, \Phi(\xi))$ is an example of a deterministic coupling \cite[Definition 1.2]{villani2008optimal}.  In conjunction with Corollary~\ref{cor:MuOr2004}, the general bound given below is crucial to proving that two copies of uHMC can meet in one step.

\begin{lem} \label{lem:overlap}
Let $\xi \sim \mathcal{N}(0,1)^d$ and suppose that $\Phi: \mathbb{R}^d \to \mathbb{R}^d$ is an invertible and differentiable map satisfying $\mnorm{D \Phi(v) - I_d} \le 1/2$ for all $v \in \mathbb{R}^d$. Then
\begin{equation} \label{tvbound}
\TV{\law(\xi)}{\law(\Phi(\xi))} \ \le \ 
\sqrt{ E[ \norm{\Phi(\xi)-\xi}^2 + 2  \mnorm{D \Phi(\xi) - I_d}_F^2 ] } \;.
\end{equation}
\end{lem}

Let $\varphi$ denote the  probability density function of the standard $d$-dimensional normal distribution satisfying \[
\law(\xi)(dv) \ = \ \mathcal{N}(0,I_d)(dv) \ = \ \varphi(v) dv \ =\ (2 \pi)^{-d/2} \exp(-\norm{x}^2 / 2) dx\;.
\] By change of variables, note that  \begin{equation} \label{eq:LawOfEta}
\law(\Phi(\xi))(dv) \ = \ |\det( D \Phi^{-1}(v) ) | \varphi( \Phi^{-1}(v)) dv \;. 
\end{equation}

To prove Lemma~\ref{lem:overlap}, we recall Pinsker's inequality.  For probability measures $\nu_1, \nu_2 \in \mathcal{P}(\mathbb{R}^d)$  with probability densities $p_{1}$ and $p_{2}$ with respect to a common reference measure $\lambda$, define  the Kullback-Leibler divergence by\[
\KL{\nu_1}{\nu_2} \ := \ \int_{\mathbb{R}^d}  p_2(v) G\Big( \frac{p_1(v)}{p_2(v)}\Big) \, \lambda(d v)  ~~\text{where~~$G(x) = \begin{cases} x \log(x) & \text{if $x \ne 0$}, \\
0 & \text{else}.  \end{cases}$}  
\] Although non-negative, the Kullback-Leibler divergence is not a metric on $\mathcal{P}(\mathbb{R}^d)$ because it is not symmetric and does not satisfy the triangle inequality.  In this situation to bound the TV distance between $\nu_1$ and $\nu_2$, it is convenient to use Pinsker's inequality \begin{equation} \label{pinsker}
\TV{\nu_1}{\nu_2} \ \le \  \sqrt{2 \ \KL{\nu_1}{\nu_2} } \;.
\end{equation}

\begin{proof}
Since $\TV{\law(\xi)}{\law(\Phi(\xi)} =\TV{\law(\xi)}{\law(\Phi^{-1}(\xi)}$,
by Pinsker's inequality \eqref{pinsker}, it suffices to bound \begin{align}
&    \KL{\law(\xi)}{\law(\Phi^{-1}(\xi)} \ = \  \int_{\mathbb{R}^d} \frac{e^{- \frac{\norm{v}^2}{2}}}{ (2 \pi)^{d /2}}  \left[ \frac{\norm{\Phi(v)}^2}{2} -\frac{\norm{v}^2}{2} - \log|\det D\Phi(v)| \right] dv \nonumber \\
& \quad = \int_{\mathbb{R}^d} \frac{e^{- \frac{\norm{v}^2}{2}}}{ (2 \pi)^{d /2}}  \left[ \frac{1}{2} \norm{\Phi(v)-v}^2 + ( \Phi(v)-v) \cdot v - \log|\det D\Phi(v)| \right] dv    \nonumber \\
& \quad = \int_{\mathbb{R}^d} \frac{e^{- \frac{\norm{v}^2}{2}}}{ (2 \pi)^{d /2}}  \left[ \frac{1}{2} \norm{\Phi(v)-v}^2 + \tr( D\Phi(v)-I_d) - \log|\det D\Phi(v)| \right] dv  \label{eq:KL}  
\end{align}
where in the last step we used the following integration by parts identity \[  \int_{\mathbb{R}^d} e^{- \frac{\norm{v}^2}{2}} (\Phi(v) - v) \cdot v dv \ = \  \int_{\mathbb{R}^d} e^{- \frac{\norm{v}^2}{2}} \tr(D \Phi(v) - I_d) dv \;. \]
Since $\mnorm{D \Phi(v) - I_d} \le 1/2$, the spectral radius of $D \Phi(v) - I_d$ does not exceed $1/2$.   Therefore, we can invoke Theorem 1.1 of \cite{rump2018estimates}, to obtain  \begin{equation} \label{eq:pertId} 
\tr( D\Phi(v)-I_d) - \log|\det D\Phi(v)| \le \frac{\mnorm{D\Phi(v)-I_d}_F^2 / 2}{1-\mnorm{D\Phi(v)-I_d}} \le  \mnorm{D\Phi(v)-I_d}_F^2 .
\end{equation} 
Inserting \eqref{eq:pertId} into \eqref{eq:KL} yields \[
 \KL{\law(\xi)}{\law(\Phi^{-1}(\xi)} \ \le \ E\left[ (1/2) \norm{\Phi(\xi)-\xi}^2 + \mnorm{D \Phi(\xi) - I_d}_F^2 \right] \;.  
\] Applying Pinsker's inequality \eqref{pinsker} gives \eqref{tvbound}.
\end{proof}

\subsection{TV Bounds and Regularization by One-Shot Couplings} 

By using the one-shot coupling illustrated in Figure~\ref{fig:coupling} (a), here we prove that the transition kernel of uHMC has a regularizing effect.  


\begin{lem} \label{lem:overlap1}
Suppose Assumption~\ref{A1234} (A1)-(A3) hold and that the duration $T>0$ and time step size $h \ge 0$ satisfy $L(T^2 + T h) \le 1/6$ and $T/h \in \mathbb{N}$.  For any $x,y,v \in \mathbb{R}^d$, let $\Phi: \mathbb{R}^d \to \mathbb{R}^d$ be the map defined by $\tilde{q}_T(x,v) = \tilde{q}_T(y,\Phi(v))$.  Then 
\begin{equation} \label{tvbound1}
\TV{\delta_x \tilde \pi}{\delta_y \tilde \pi} \, \le \, 
\TV{\law(\xi)}{\law(\Phi(\xi))} \, \le \, \frac{3}{2}  \left( T^{-2}  + 27  d L_H^2  T^4 \right)^{1/2}  |x-y|.
\end{equation}
\end{lem}

\begin{proof}
Since $\tilde{q}_T: \mathbb{R}^{2 d} \to \mathbb{R}^d$ is deterministic and measurable \cite[Lemma 3]{madras2010quantitative}, \begin{align}
\TV{\delta_x \tilde \pi}{\delta_y \tilde \pi} &= \TV{\law(\tilde{q}_T(y,\xi))}{\law(\tilde{q}_T(y,\Phi(\xi)))} \nonumber \\
&\le  \TV{\law(\xi)}{\law(\Phi(\xi))} \;,
\end{align} 
which gives the first inequality in \eqref{tvbound1}.
Moreover, by Lemmas~\ref{lem:overlap},~\ref{lem:oneshot:1:a} and~\ref{lem:oneshot:1:b}, we have \[
 \TV{\law(\xi)}{\law(\Phi(\xi)}^2 \ \le \
\left( (9/4) \, T^{-2}  + (121/2) \, d \, L_H^2 \, T^4 \right)  \, |x-y|^2
\] 
where we used $\mnorm{D\Phi(v)-I_d}_F^2 \le d \mnorm{D\Phi(v)-I_d}^2$.  Taking square roots and inserting $ 27> (4/9)(121/2)$ gives the second inequality in \eqref{tvbound1}.
\end{proof}

 Lemma~\ref{lem:overlap1} implies that the transition kernel of uHMC has a regularizing effect in the following sense (cf. Theorem 12 (a) of \cite{madras2010quantitative}).  

\begin{lem} \label{lem:regularization}
Suppose Assumption~\ref{A1234} (A1)-(A3) hold and that the duration $T>0$ and time step size $h \ge 0$ satisfy $L(T^2 + T h) \le 1/6$ and $T/h \in \mathbb{N}$.  For any  $\nu,\eta \in \mathcal{P}(\mathbb{R}^d)$,
\begin{equation} \label{eq:regularization}
\TV{\eta \tilde \pi}{\nu \tilde \pi} \ \le \ 
 \frac{3}{2} \, \left( T^{-2}  + 27  d L_H^2  T^4 \right)^{1/2} \W^1(\eta, \nu)  \;.
\end{equation}
\end{lem}

\begin{proof}
Let $\omega$ be an arbitrary coupling of $\nu, \eta$.
By the coupling characterization of the TV distance in \eqref{eq:tv_coupling}, 
\begin{align*}
\TV{\eta \tilde \pi}{\nu \tilde \pi} \le E[ \TV{\delta_X \tilde \pi}{\delta_Y \tilde \pi} ] \le  \frac{3}{2} \, \left( T^{-2}  + 27  d L_H^2  T^4 \right)^{1/2} E[ \norm{X - Y} ] 
\end{align*}
where $\law(X,Y) = \omega$ and in the last step we inserted \eqref{tvbound1} in Lemma~\ref{lem:overlap1}.  Since $\omega$ is arbitrary, we can take the infimum over all $\omega \in \J(\nu,\eta)$ to obtain \eqref{eq:regularization}.  
\end{proof}

To bound the bias in the invariant measure of uHMC, we bound the TV distance between $\mu \tilde{\pi}$ and $\mu \pi$ by using the one-shot coupling illustrated in Figure~\ref{fig:coupling}(b).


\begin{lem} \label{lem:overlap2}
Suppose Assumption~\ref{A1234} holds and that the duration $T>0$ and time step size $h \ge 0$ satisfy $L(T^2 + T h) \le 1/6$ and $T/h \in \mathbb{N}$. 
For any $x,v \in \mathbb{R}^d$, let $\Phi: \mathbb{R}^d \to \mathbb{R}^d$ be the map defined by $\tilde{q}_T(x,v) = q_T(x,\Phi(v))$.  Then
\begin{equation} \label{tvbound2}
\begin{aligned}
& \TV{\delta_x \pi}{\delta_x \tilde \pi} \ \le \ \TV{\law(\xi)}{\law(\Phi(\xi))} \\ 
& \ \le \
h^2 \, \Big[  d^3 \big( 4 L_I^2 T^4 + 14 L_H^2 L_I T^6 + 14 L_H^4 T^8 \big) \\
& \qquad+ d^2 \big( 35 L_H^2 T^2 + 8 L_I^2 T^4 + 28 L_H^2 L_I T^6 + 28 L_H^4 T^8 \big) \\
& \qquad+ d \big( 16 L^2 + 4 L_H^2 T^2) + \big(2 d L_H^2  + L^2 T^{-2}  \big) \norm{x}^2 \\ 
& \qquad + \big( d L_I^2 + d L_H^2 L_I T^2 + d L_H^4 T^4 + L_H^2 T^{-2} \big) \norm{x}^4   \Big]^{1/2} \;.
\end{aligned}
\end{equation}
\end{lem}

\begin{proof}
Since $q_T: \mathbb{R}^{2 d} \to \mathbb{R}^d$ is deterministic and measurable \cite[Lemma 3]{madras2010quantitative}, \begin{align}
\TV{\delta_x  \pi}{\delta_x \tilde \pi} &= \TV{\law(q_T(x,\xi))}{\law(q_T(x,\Phi(\xi)))} \nonumber \\
&\le  \TV{\law(\xi)}{\law(\Phi(\xi))} \;,
\end{align} 
which gives the first inequality in \eqref{tvbound2}.
 By Lemmas~\ref{lem:overlap},~\ref{lem:oneshot:2:a} and~\ref{lem:oneshot:2:b}, $\mnorm{D\Phi(v)-I_d}_F^2 \le d \mnorm{D\Phi(v)-I_d}^2$, and the Cauchy-Schwarz inequality we have \begin{align*}
& \TV{\law(\xi)}{\law(\Phi(\xi)}^2
\ \le \ E\left[  \norm{\Phi(\xi)-\xi}^2 + 2 d \mnorm{D \Phi(\xi) - I_d}^2 \right] \\
& \quad \le \  h^4 \, E\Big[  10 d L^2 
 +  \big( 2 d L_H^2 + L^2 T^{-2} \big) \norm{x}^2  + \big( 6 L^2 + 33 d L_H^2 T^2  \big) \norm{\xi}^2 \\
& \quad \qquad + \big( d L_I^2 + L_H^2 T^{-2} + d L_H^2 L_I T^2 + d L_H^4 T^4 \big) \norm{x}^4  \\
&\quad \qquad + \big( L_H^2 T^2 + 2 d L_I^2 T^4 + 7 d L_H^2 L_I T^6 + 7 d L_H^4 T^8 \big) \norm{\xi}^4 \Big] \\ 
& \quad \le \  h^4 \, \Big[  d^3 \big( 4 L_I^2 T^4 + 14 L_H^2 L_I T^6 + 14 L_H^4 T^8 \big) \\
& \quad \qquad+ d^2 \big( 35 L_H^2 T^2 + 8 L_I^2 T^4 + 28 L_H^2 L_I T^6 + 28 L_H^4 T^8 \big) \\
& \quad \qquad+ d \big( 16 L^2 + 4 L_H^2 T^2) + \big(2 d L_H^2  + L^2 T^{-2}  \big) \norm{x}^2 \\ 
& \quad \qquad + \big( d L_I^2 + d L_H^2 L_I T^2 + d L_H^4 T^4 + L_H^2 T^{-2} \big) \norm{x}^4   \Big]
\end{align*} 
where in the last step we used 
$E[ \norm{\xi}^2 ] = d$ and
$E[ \norm{\xi}^4 ] = 2 d (d+2)$.
Taking square roots gives the second inequality in \eqref{tvbound2}.
\end{proof}

 Lemma~\ref{lem:overlap2} implies the following bound on the TV distance between $\mu \pi$ and $\mu \tilde \pi$.  

\begin{lem} \label{lem:tv_strong_error}
Suppose Assumption~\ref{A1234} holds and that the duration $T>0$ and time step size $h \ge 0$ satisfy $L(T^2 + T h) \le 1/6$ and $T/h \in \mathbb{N}$. 
 Then
$\ \TV{\mu \pi}{\mu \tilde \pi}  \le C  h^2$,
where $C$ is defined in \eqref{eq:C}.
\end{lem}

The proof of this result is similar to the proof of Lemma~\ref{lem:regularization} and therefore omitted.

\begin{figure}[t]
\centering
\begin{minipage}{\textwidth} 
      \begin{minipage}[b]{0.50\textwidth}
\begin{tikzpicture}[scale=1.]
\begin{scope}[very thick] 
       \tikzmath{
       \t=1.0; \x1 = 0.0; \x2 =-4.0; \y1=0.5; \y2=-4.5;  \u1=4.; \u2=0.0; \om=1.;
                      \v1=\u1+\om*cos(deg(\om*\t))/sin(deg(\om*\t))*(\x1-\y1); 
                      \v2=\u2+\om*cos(deg(\om*\t))/sin(deg(\om*\t))*(\x2-\y2); 
                      \z1=cos(deg(\om*\t))*\x1+sin(deg(\om*\t))*\u1/\om;
                      \z2=cos(deg(\om*\t))*\x2+sin(deg(\om*\t))*\u2/\om;
                      } 
                       \draw  [color=black,dashed,smooth,domain=0:\t,samples=10,variable=\xx]  plot ({cos(deg(\om*\xx))*\x1+sin(deg(\om*\xx))*\u1/\om},{cos(deg(\om*\xx))*\x2+sin(deg(\om*\xx))*\u2/\om});          
                       \draw  [color=red,dashed,smooth,domain=0:\t,samples=10,variable=\x]  plot ({cos(deg(\om*\x))*\y1+sin(deg(\om*\x))*\v1/\om},{cos(deg(\om*\x))*\y2+sin(deg(\om*\x))*\v2/\om});
\filldraw[color=black,fill=black] (\x1,\x2) circle (0.1) node [left,black, scale=1.5]  {$x$};
\filldraw[color=black,fill=black] (\y1,\y2) circle (0.1) node [left,black, scale=1.5]  {$y$};
\filldraw[color=black,fill=black] (\z1,\z2) circle (0.1) node [above right,black, scale=1.5] {$\tilde q_T(x,\xi)$};
\draw[->,-{Latex[length=3mm]}](\x1,\x2)--(\x1+0.25*\u1,\x2+0.25*\u2) node [above,black, scale=1.5] {$\xi$};
\draw[->,-{Latex[length=3mm]}](\y1,\y2)--(\y1+0.25*\v1,\y2+0.25*\v2) node [below right,black, scale=1.5] {$\Phi(\xi)$};
 \end{scope}
\end{tikzpicture}
\end{minipage}
      \begin{minipage}[b]{0.50\textwidth}
\begin{tikzpicture}[scale=1.0]
\begin{scope}[very thick] 
       \tikzmath{
       \t=1.0; \x1 = 0.0; \x2 =-4;  \u1=4; \u2=0.0; \y1=\x1; \y2=\x2; \om=1.5;
                      \v1=sin(deg(\t))/sin(deg(\om*\t))*\u1+cos(deg(\t))/sin(deg(\om*\t))*\x1-cos(deg(\om*\t))/sin(deg(\om*\t))*\y1); 
                      \v2=sin(deg(\t))/sin(deg(\om*\t))*\u2+cos(deg(\t))/sin(deg(\om*\t))*\x2-cos(deg(\om*\t))/sin(deg(\om*\t))*\y2; 
                      \z1=cos(deg(\t))*\x1+sin(deg(\t))*\u1;
                      \z2=cos(deg(\t))*\x2+sin(deg(\t))*\u2;
                      } 
                                  \draw  [color=black,dashed,smooth,domain=0:\t,samples=10,variable=\xx]  plot ({cos(deg(\xx))*\x1+sin(deg(\xx))*\u1},{cos(deg(\xx))*\x2+sin(deg(\xx))*\u2});            
                       \draw  [color=red,dashed,smooth,domain=0:\t,samples=10,variable=\x]  plot ({cos(deg(\om*\x))*\y1+sin(deg(\om*\x))*\v1},{cos(deg(\om*\x))*\y2+sin(deg(\om*\x))*\v2});
\filldraw[color=black,fill=black] (\x1,\x2) circle (0.1) node [left,black, scale=1.5]  {$x$};
\filldraw[color=black,fill=black] (\z1,\z2) circle (0.1) node [above right,black, scale=1.5] {$\tilde q_T(x,\xi)$};
\draw[->,-{Latex[length=3mm]}](\x1,\x2)--(\x1+0.25*\u1,\x2+0.25*\u2) node [above,black, scale=1.5] {$\xi$};
\draw[->,-{Latex[length=3mm]}](\y1,\y2)--(\y1+0.25*\v1,\y2+0.25*\v2) node [below right,black, scale=1.5] {$\Phi(\xi)$};
 \end{scope}
\end{tikzpicture}
\end{minipage}
\end{minipage}
\begin{minipage}{\textwidth}
      \begin{minipage}[b]{0.48\textwidth}
      \centering
{\normalsize (a) $\tilde q_T(x,\xi) = \tilde q_T(y, \Phi(\xi))$}
\end{minipage}
\begin{minipage}[b]{0.48\textwidth}
\centering
{\normalsize (b) $\tilde q_T(x,\xi) =  q_T(x, \Phi(\xi))$}
\end{minipage}
\end{minipage}
\caption{\small {\bf  One-shot couplings.} (a) To obtain TV convergence bounds for uHMC, the initial velocities are coupled such that $\tilde q_T(x,\xi) = \tilde q_T(y, \Phi(\xi))$ with maximal possible probability.  (b) For TV accuracy of the invariant measure of uHMC, the initial velocities are coupled such that $\tilde q_T(x,\xi) =  q_T(x, \Phi(\xi))$ with maximal possible probability. }
  \label{fig:coupling}
\end{figure}
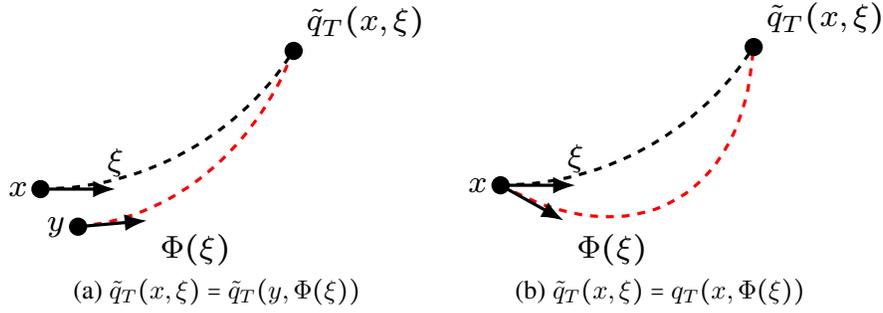

\subsection{Successful Coupling for uHMC}\label{sec:main_results:couplings}

The TV bound in Theorem~\ref{thm:tv_to_equilibrium} can be realized by a successful coupling for uHMC which combines a coupling that brings the two copies arbitrarily close together and the following one-shot coupling construction.

\medskip

Let $\mathcal U \sim \text{Unif}(0,1)$ and $\xi \sim \mathcal{N}(0,I_d)$ be independent random variables. 
For notational convenience, set \begin{equation} 
\alpha(v) = 	\min\Big( \dfrac{\varphi(\Phi(v)) |\det(D \Phi(v))|}{\varphi(v)} , 1 \Big)  \;,
\end{equation}
where for any $x,y,v \in \mathbb{R}^d$ we define $\Phi$ via \begin{equation}
\tilde q_T(x,v) =\tilde q_T(y,\Phi(v)) \;.
\end{equation} 

The transition step of the one-shot coupling is given by \begin{equation}\label{**} \begin{cases} 		\tilde{X}(x,y) \ = \ \tilde{q}_T(x,\xi) , \\		\tilde{Y}(x,y) \ = \ \tilde{q}_T(y,\eta), \end{cases} \quad \text{where} \quad \eta \:= \ \begin{cases} 		\Phi(\xi)  & \text{if } \ \mathcal U \ \leq \ 		\alpha(\xi),		\\ 		\tilde \eta  & \text{otherwise},		\end{cases} \end{equation} where $\tilde \eta$ is independent of $\xi$ and $\mathcal{U}$ and satisfies  \begin{equation}
\law( 
\tilde{\eta})(dv)  =  \frac{(\varphi(v) -  \varphi(\Phi^{-1}(v)) |\det(D\Phi^{-1}(v))|)^+}{1-\gamma}  dv 
\end{equation}
where $\gamma = E[\alpha(\xi)]$.
By definition of $\Phi$, $\tilde{q}_T(x,\xi) = \tilde{q}_T(x,\Phi(\xi))$.

In order to verify that $(\tilde{X}(x,y),\tilde{Y}(x,y))$ is indeed a coupling of the transition probabilities $\tilde{\pi} (x,\cdot )$ and $\tilde{\pi} (y,\cdot ) $, we remark that the distribution of $\eta$ is  $\mathcal{N}(0,I_d)$ since, by definition of $\eta$ in \eqref{**} and a change of variables, \begin{eqnarray*} P[ \eta \in B ] &  =& E \left[ I_B( \Phi(\xi))\,  \alpha(\xi) \right] \, + \, E \left[ I_B( \tilde{\eta} ) 
\right] \, E \left[  1 - \alpha(\xi)  \right] \\ &=& \int I_B( v )\,   \min(\varphi(\Phi^{-1}(v)) |\det(D\Phi^{-1}(v))|, \varphi(v)) \,dv\\ &&\quad +\int I_B( v )  (\varphi(v) - \varphi(\Phi^{-1}(v)) |\det(D\Phi^{-1}(v))|)^+\, dv\\ &=&\int I_B(v)\,\varphi(v)\, dv\   =\ P[\xi \in B]  \end{eqnarray*} for any measurable set $B$.
As a byproduct of this calculation, note also that \begin{align*} P[ \eta \ne \Phi(\xi) ] \ &= \ \int \left( \varphi( v)- \varphi(\Phi^{-1}(v)) |\det(D\Phi^{-1}(v))| \right)^+ dv \\ 
\ &= \ \TV{\law(\xi)}{\law(\Phi(\xi))} \;. \end{align*} 
This calculation shows that the one-shot coupling in \eqref{**} ensures that $\eta=\Phi(\xi)$ with maximal possible probability, and with remaining probability, $
\eta = \tilde{\eta}$ with $\tilde{\eta} \perp \xi$.  Thus, the coupling coalesces in one step whenever $\eta=\Phi(\xi)$, as illustrated in Figure~\ref{fig:coupling}(a).

\section{A priori estimates}

Here we state several bounds for the dynamics \eqref{velVerlet} that are crucial in the proof of our main results.  Throughout this section, we fix a duration parameter $T>0$ and a time step size $h>0$ such that
$T/h\in\mathbb N$ and 
\begin{equation}
    \label{eq::Tconstraint}
    L(T^2 + T h) \ \le\ 1/6\, .
\end{equation}

\subsection{A priori estimates for the dynamics}

\begin{lem} \label{lem:apriori}
Suppose that Assumption~\ref{A1234} (A1)-(A2) hold. For any $x,y,u,v\in\R^d$,
	\begin{align}
	\label{apriori:1}
& \max_{s\leq T}\norm{\tilde{q}_s(x,v)} \, \leq \, (1+L (T^2+T h)) \max(\norm{x},\norm{x+T v}),\\
\label{apriori:2}
& \max_{s\leq T}\norm{\tilde{v}_s(x,v)} \, \leq \,
\norm{v} + L T (1+L (T^2+T h))  \max(\norm{x},\norm{x+T v}), \\ 
\begin{split}
& \max_{s\leq T} \norm{\tilde{q}_s(x,u)-\tilde{q}_s(y,v)} \label{apriori:3} \\
& \qquad \leq \, (1+ L (T^2+T h))\max\left(\norm{x-y},\norm{x-y+T (u-v)}\right) .
\end{split}
\end{align}
\end{lem}

Lemma~\ref{lem:apriori} is contained in Lemmas~3.1 and 3.2 of \cite{BoEbZi2020} and hence a proof is omitted. 

\subsection{A priori estimates for velocity derivative flow}

The next lemma provides similar bounds for $D_2 \tilde{q}_T(x,u) := \partial \tilde{q}_T(x,u) / \partial u$.

\begin{lem} \label{lem:apriori2}
Suppose that Assumption~\ref{A1234} (A1)-(A3) hold. For any $x,y,u,v\in\R^d$,
	\begin{align}
	\label{apriori:5}
&\max_{s\leq T}\mnorm{ D_2 \tilde{q}_s(x,v) } \, \leq \, (6/5) \, T \, , \\
	\label{apriori:6}
&\max_{s\leq T}\mnorm{ D_2 \tilde{v}_s(x,v) } \, \leq \, (6/5) \, , \quad\mbox{and} \\ 
	\begin{split} \label{apriori:7} 	
&\lefteqn{\max_{s\leq T} \mnorm{D_2 \tilde{q}_s(x,u)-D_2 \tilde{q}_s(y,v)}} \\
& \qquad \leq \,  L_H (6/5)^2 T^3 (1+L(T^2+T h))\max\left(\norm{x-y},\norm{(x-y)+T(u-v)}\right) . 
\end{split}
\end{align}
\end{lem}

\begin{proof}
From \eqref{velVerlet}, note that $D_2 \tilde{q}_t(x,u)$ and $D_2 \tilde{v}_t(x,u)$ satisfy: \begin{align} 
\begin{split}
\label{dervelVerlet}
D_2 \tilde{q}_T(x,u)  \ &= \ T I_d - \frac{1}{2} \int_0^T (T - s) \left[ \mathbf{H}(\tilde{q}_{\lb{s}}) D_2 \tilde{q}_{\lb{s}}  + \mathbf{H}(\tilde{q}_{\ub{s}}) D_2 \tilde{q}_{\ub{s}} \right] ds  \\
    & \quad + \frac{1}{2} \int_0^T (s - \lb{s}) \left[ \mathbf{H}(\tilde{q}_{\ub{s}}) D_2 \tilde{q}_{\ub{s}} - \mathbf{H}(\tilde{q}_{\lb{s}}) D_2 \tilde{q}_{\lb{s}} \right] ds \;,  
    \end{split} \\
  D_2 \tilde{v}_T(x,u)  \ &= \ I_d - \frac{1}{2} \int_0^T \left[ \mathbf{H}(\tilde{q}_{\lb{s}}) D_2 \tilde{q}_{\lb{s}}  + \mathbf{H}(\tilde{q}_{\ub{s}}) D_2 \tilde{q}_{\ub{s}} \right] ds  \;. \nonumber
\end{align} Applying Assumption~\ref{A1234} (A2) and inserting $L T^2 \le 1/6$, we obtain \begin{align*}
&\max_{s \le T}  \mnorm{D_2 \tilde{q}_s(x,u)} \ \le \ T \mnorm{I_d} + T^2 \max_{s \le T} \mnorm{\mathbf{H}(\tilde{q}_s) D_2 \tilde{q}_s(x,u)}    \\
 & \qquad \le \ T + L T^2 \max_{s \le T} \mnorm{D_2 \tilde{q}_s(x,u)}  \ \le \ T + (1/6) \max_{s \le T} \mnorm{D_2 \tilde{q}_s(x,u)} \ \le \ (6/5) T \;, \\
&\max_{s \le T}  \mnorm{D_2 \tilde{v}_s(x,u)} \ \le \ 
1 + L T \max_{s \le T} \mnorm{D_2 \tilde{q}_s(x,u)} \ \le \ 1 + L T^2 (6/5) \le 6/5 \;, 
\end{align*} 
which gives \eqref{apriori:5} and 
\eqref{apriori:6}.   For \eqref{apriori:7}, it is notationally convenient to introduce the shorthand $\tilde{q}_s^{(1)} : = \tilde{q}_s(x,u)$ and $\tilde{q}_s^{(2)} : = \tilde{q}_s(y,v)$ for any $s \in [0,T]$.  Then from \eqref{dervelVerlet}  \begin{align*}
& D_2 \tilde{q}_T^{(1)} - D_2 \tilde{q}_T^{(2)}  \ = \ - \frac{1}{2} \int_0^T (T - s) \left[ \mathbf{H}(\tilde{q}_{\lb{s}}^{(1)}) D_2 \tilde{q}_{\lb{s}}^{(1)}  - \mathbf{H}(\tilde{q}_{\lb{s}}^{(2)}) D_2 \tilde{q}_{\lb{s}}^{(2)} \right] ds  \\
&\qquad\qquad\qquad\qquad - \frac{1}{2} \int_0^T (T - s) \left[  \mathbf{H}(\tilde{q}_{\ub{s}}^{(1)}) D_2 \tilde{q}_{\ub{s}}^{(1)} -  \mathbf{H}(\tilde{q}_{\ub{s}}^{(2)}) D_2 \tilde{q}_{\ub{s}}^{(2)} \right] ds  \\
    &\qquad\qquad\qquad\qquad + \frac{1}{2} \int_0^T (s - \lb{s}) \left[ \mathbf{H}(\tilde{q}_{\ub{s}}^{(1)}) D_2 \tilde{q}_{\ub{s}}^{(1)} - \mathbf{H}(\tilde{q}_{\ub{s}}^{(2)}) D_2 \tilde{q}_{\ub{s}}^{(2)}  \right] ds \\
        &\qquad\qquad\qquad\qquad + \frac{1}{2} \int_0^T (s - \lb{s}) \left[ \mathbf{H}(\tilde{q}_{\lb{s}}^{(2)}) D_2 \tilde{q}_{\lb{s}}^{(2)}  - \mathbf{H}(\tilde{q}_{\lb{s}}^{(1)}) D_2 \tilde{q}_{\lb{s}}^{(1)} \right] ds \\
& =  \int_0^T \frac{T - s}{2} \left[ \mathbf{H}(\tilde{q}_{\lb{s}}^{(2)}) ( D_2 \tilde{q}_{\lb{s}}^{(2)} - D_2 \tilde{q}_{\lb{s}}^{(1)} )   - 
( \mathbf{H}(\tilde{q}_{\lb{s}}^{(1)}) -  \mathbf{H}(\tilde{q}_{\lb{s}}^{(2)})) D_2 \tilde{q}_{\lb{s}}^{(1)} \right] ds  \\
& + \int_0^T \frac{T - s}{2} \left[ \mathbf{H}(\tilde{q}_{\ub{s}}^{(2)}) ( D_2 \tilde{q}_{\ub{s}}^{(2)} - D_2 \tilde{q}_{\ub{s}}^{(1)}) - 
( \mathbf{H}(\tilde{q}_{\ub{s}}^{(1)}) - \mathbf{H}(\tilde{q}_{\ub{s}}^{(2)}) )  D_2 \tilde{q}_{\ub{s}}^{(1)}
 \right] ds  \\
    & -  \int_0^T \frac{s - \lb{s}}{2} \left[  \mathbf{H}(\tilde{q}_{\ub{s}}^{(2)}) ( D_2 \tilde{q}_{\ub{s}}^{(2)}  -D_2 \tilde{q}_{\ub{s}}^{(1)})  -  ( \mathbf{H}(\tilde{q}_{\ub{s}}^{(1)}) - \mathbf{H}(\tilde{q}_{\ub{s}}^{(2)}) )  D_2 \tilde{q}_{\ub{s}}^{(1)} \right] ds \\
        & -  \int_0^T \frac{s - \lb{s}}{2} \left[ \mathbf{H}(\tilde{q}_{\lb{s}}^{(1)}) ( D_2 \tilde{q}_{\lb{s}}^{(1)} - D_2 \tilde{q}_{\lb{s}}^{(2)}) - ( \mathbf{H}(\tilde{q}_{\lb{s}}^{(2)}) - \mathbf{H}(\tilde{q}_{\lb{s}}^{(1)}) ) D_2 \tilde{q}_{\lb{s}}^{(2)}    \right] ds \;.
\end{align*}
By using Assumption~\ref{A1234} (A2),   Assumption~\ref{A1234} (A3), and \eqref{apriori:5}, we then get 
\begin{align*}
& \max_{s \le T} \mnorm{D_2 \tilde{q}_s^{(1)} - D_2 \tilde{q}_s^{(2)}} \\
& \qquad \le \ L T^2 \max_{s \le T} \mnorm{ D^2 \tilde{q}_s^{(1)} -  D^2 \tilde{q}_s^{(2)} } + L_H (6/5) T^3 \max_{s \le T} \norm{\tilde{q}_s^{(1)} - \tilde{q}_s^{(2)}} \;. 
\end{align*}
Inserting $LT^2 \le 1/6$ and \eqref{apriori:3}, and simplifying yields \eqref{apriori:7}.  
\end{proof}

\subsection{Discretization error bounds for Verlet}

In this part we gather standard estimates of the error between the exact solution $q_T(x,v)$ and Verlet $\tilde q_T(x,v)$.   These estimates are a key ingredient in quantifying the invariant measure accuracy in TV distance of uHMC.

\medskip

To this end, the following error estimate for a variant of the trapezoidal rule is useful because \eqref{velVerlet} involves this particular trapezoidal rule approximation.

\begin{lem} \label{trapz}
Fix $T>0$.  Let $(\mathbf{V},\mnorm{\cdot})$ be a normed space.  Let $f: [0,T] \to \mathbf{V}$ be a twice differentiable function such that $\max_{s \in [0,T]} \mnorm{f'(s)} \le B_1$ and $\max_{s \in [0,T]} \mnorm{f''(s)} \le B_2$.  Then for any $h > 0$ such that $T/h \in \mathbb{N}$ 
\begin{equation} \label{eq:trapz_2}
\begin{aligned}
& \mnorm{ \int_0^T (T-s) f(s) ds - \frac{1}{2} \int_0^T (T-s) \left[ f(\lb{s}) + f(\ub{s}) \right] ds } \\
& \qquad \, \le \,  \frac{h^2}{12}  ( B_2  T^2 + B_1  T) \, .
\end{aligned}
\end{equation}
\end{lem}

\begin{proof}
The error over $[t_k, t_{k+1}]$ is given by
\[
\epsilon_k \ := \  \mnorm{ \int_{t_k}^{t_{k+1}} (T-s) f(s) ds - \frac{1}{2} \int_{t_k}^{t_{k+1}} (T-s) \left[ f(\lb{s}) + f(\ub{s}) \right] ds } \;.
\] By integrating by parts,  \[
\epsilon_k \ = \ \mnorm{ \int_{t_k}^{t_{k+1}} \left[ \frac{s^2}{2}  - T s - \alpha_k \right] f'(s) ds }
\] where $\alpha_k = (1/4) (t_{k+1}^2 - 2 T t_{k+1} - 2 T t_k + t_k^2)$.  A second integration by parts gives \begin{align*}
\epsilon_k =  \mnorm{ \int_{t_k}^{t_{k+1}} \left[ \frac{s^3}{6} - \frac{T s^2}{2}  - \alpha_k s + \beta_k \right] f''(s) ds  - \frac{1}{12} f'(t_{k}) h^3 } \le  \frac{1}{12} ( B_2 T + B_1 ) h^3 
\end{align*}
where $\beta_k = (1/12) (-6 T t_k + 3 t_k^2  + t_{k+1}^2) t_{k+1}$ and in the last step we used \[
\int_{t_k}^{t_{k+1}} \left| \frac{s^3}{6} - \frac{T s^2}{2}  - \alpha_k s + \beta_k \right| ds = \frac{1}{12} (T-t_k) h^3 
\le \frac{1}{12} \, T \, h^3  \;.  
\]
Summing these errors over the $T/h$ subintervals gives the upper bound in \eqref{eq:trapz_2}.  
\end{proof}

\medskip

\begin{lem} \label{lem:verlet_error:1}
Suppose Assumption~\ref{A1234} (A1)-(A3) hold. For any $x,v\in\R^d$,
\begin{equation} \label{verlet_error:1}
	\begin{aligned}
&\max_{s\leq T}\norm{ q_s(x,v) -  \tilde{q}_s(x,v)}   \\
& \qquad  \leq  h^2 \big(\frac{7}{45} L  \norm{x}   + \frac{1547}{1800} L T \norm{v}
 + \frac{1}{120} L_H \norm{x}^2 + \frac{3}{10} L_H T^2 \norm{v}^2  \big)  \;.
\end{aligned}
\end{equation}
\end{lem}

\begin{proof}
By \eqref{apriori:1} and \eqref{apriori:2} with $h=0$, and since $LT^2 \le 1/6$, note that \begin{align}
    \label{max:q} \max_{s \le T} \norm{q_s} \ &\le \ (7/6) \, (\norm{x} + T \norm{v}) \;,  \\
    \label{max:v} \max_{s \le T} \norm{v_s} \ &\le \ \norm{v} + (7/6) \, L T  \, (\norm{x} + T \norm{v}) \ \le \ (7/6) \, L T \, \norm{x}  +  (6/5) \, \norm{v} \;,  \\
        \label{max:v2} \max_{s \le T} \norm{v_s}^2 \ &\le \ 2 (7/6)^2 L^2 T^2 \norm{x}^2 + 2 (6/5)^2 \norm{v}^2 \ \le \ (1/2) L \norm{x}^2 + 3 \norm{v}^2 \;.
\end{align}
Introduce the shorthand $q_T:= q_T(x,v)$ and $\tilde q_T := \tilde q_T(x,v)$.  Using \eqref{velVerlet}, note that
 \begin{align}
 q_T &  - \tilde q_T \ = \ \rn{1} + \rn{2} + \rn{3} \quad \text{where} \label{diff:q_tq} \\
\rn{1} \ &:= \ \frac{1}{2} \int_0^T (T-s) [ \nabla U(\tilde q_{\lb{s}} ) - \nabla U(q_{\lb{s}} ) + \nabla U(\tilde q_{\ub{s}}) - \nabla U(q_{\ub{s}}) ] ds  \nonumber \\
& \quad - \frac{1}{2} \int_0^T ( s - \lb{s}) [ \nabla U(\tilde q_{\ub{s}}) - \nabla U(q_{\ub{s}}) - (\nabla U(\tilde q_{\lb{s}}) - \nabla U(q_{\lb{s}}) ) ] ds \;, \nonumber \\
\rn{2} \ &:= \ -\frac{1}{2} \int_0^T ( s - \lb{s}) [ \nabla U(q_{\ub{s}}) - \nabla U(q_{\lb{s}}) ] ds \;,  \quad \text{and} \nonumber \\
\rn{3} \ &:= \ - \int_0^T (T-s) \nabla U(q_s) ds + \frac{1}{2} \int_0^T (T-s) [ \nabla U(q_{\lb{s}} ) + \nabla U(q_{\ub{s}}) ] ds   \;. \nonumber
\end{align} 
By Assumption~\ref{A1234} (A2) and the condition $LT^2 \le 1/6$,  \begin{align} \label{qerr:1}
    \norm{\rn{1}} \ &\le \ L T^2 \, \max_{s \le T} \norm{q_s - \tilde q_s} \ \le \ (1/6) \, \max_{s \le T} \norm{q_s - \tilde q_s} \;, \\
  \norm{\rn{2}} \ &\le \ (1/2) h L \int_0^T \norm{q_{\ub{s}}-q_{\lb{s}}} ds \ = \  (1/2) h L \int_0^T \norm{\int_{\lb{s}}^{\ub{s}} v_r dr} ds \;, \nonumber \\
\  & \le \  (1/2) L h^2 T \max_{s \le T} \norm{v_s} \ \le \  L h^2 \, ( (7/72) \, \norm{x} + (3/5) \, T \, \norm{v} ) \label{qerr:2} \;,
\end{align}
where for \eqref{qerr:2} we  used \eqref{max:v}. Next, we apply Lemma~\ref{trapz} with 
$f(s) = - \nabla U(q_s)$ and 
\begin{align*}
&\norm{f'(s)} = \norm{\mathbf{H}(q_s) v_s}  \le L \norm{v_s} \le L ( L T (7/6) |x| + (6/5) |v|) =: B_1 \;,    \\
&\norm{f''(s)} = \norm{-\differential^3 U(q_s)(v_s,v_s) + \mathbf{H}(q_s) \nabla U(q_s)} \;, \\
&\quad \le L_H \norm{v_s}^2 + L^2 \norm{q_s} \le L_H ( L (1/2) \norm{x}^2 + 3 \norm{v}^2 ) + L^2 (7/6) (\norm{x} + T \norm{v}) =: B_2   \;, 
\end{align*} where we used Assumption~\ref{A1234} (A2), (A3) and $L T^2 \le 1/6$. Thus, \begin{align}
  \norm{\rn{3}} \le \frac{h^2}{12} \left(  \frac{7}{18} L  \norm{x} +  \frac{251}{180} L T \norm{v} + \frac{1}{12} L_H \norm{x}^2 + 3 L_H T^2 \norm{v}^2 ) \right)  
    \label{qerr:3} \;.
\end{align}
Inserting \eqref{qerr:1}, \eqref{qerr:2} and \eqref{qerr:3}  into \eqref{diff:q_tq} and simplifying gives \eqref{verlet_error:1}.
\end{proof}

\begin{lem} \label{lem:verlet_error:2}
Suppose Assumption~\ref{A1234} (A1)-(A4) hold. For any $x,v\in\R^d$,
\begin{equation} \label{verlet_error:2}
	\begin{aligned}
&\max_{s\leq T} \mnorm{D_2 q_s(x,v)-D_2 \tilde{q}_s(x,v)}    \leq   h^2 \Big(  \frac{43}{50} L T  +  \frac{1183}{4500} L_H T \norm{x}    +  \frac{1847}{1250} L_H T^2 \norm{v} \\
&\qquad 
+
( \frac{3}{250}  L_H^2 T^3 + \frac{1}{100} L_I T ) \norm{x}^2  + ( \frac{54}{125}  L_H^2 T^5 + \frac{9}{25}  L_I T^3 )  \norm{v}^2 \Big) \;.  
\end{aligned}
\end{equation}
\end{lem}

\begin{proof}
Using \eqref{velVerlet}, write the difference as
 \begin{align}
 D_2 q_T & (x,v)  - D_2 \tilde q_T(x,v) \ = \ \rn{1} + \rn{2} + \rn{3} + \rn{4} \quad \text{where} \label{diff:D2q_D2tq} \\
 \rn{1} \ &:= \ \frac{1}{2} \int_0^T (T-s)   \mathbf{H}(\tilde q_{\lb{s}} ) (D_2 \tilde q_{\lb{s}} - D_2 q_{\lb{s}})   ds \nonumber  \\
& \quad + \frac{1}{2} \int_0^T (T-s)  \mathbf{H}(\tilde q_{\ub{s}}) ( D_2 \tilde q_{\ub{s}} - D_2  q_{\ub{s}})  ds  \nonumber \\
& \quad - \frac{1}{2} \int_0^T ( s - \lb{s})  \mathbf{H}(\tilde q_{\ub{s}}) ( D_2 \tilde q_{\ub{s}} - D_2  q_{\ub{s}} ) ds \nonumber \\
& \quad + \frac{1}{2} \int_0^T ( s - \lb{s}) \mathbf{H}(\tilde q_{\lb{s}}) (D_2 \tilde q_{\lb{s}} - D_2  q_{\lb{s}})   ds \nonumber \\
\rn{2} \ &:= \ \frac{1}{2} \int_0^T (T-s)  ( \mathbf{H}(\tilde q_{\lb{s}} ) - \mathbf{H}(q_{\lb{s}} ) ) D_2  q_{\lb{s}}  ds \nonumber  \\
& \quad + \frac{1}{2} \int_0^T (T-s)  ( \mathbf{H}(\tilde q_{\ub{s}}) - \mathbf{H}(q_{\ub{s}}) ) D_2  q_{\ub{s}}  ds  \nonumber \\
& \quad - \frac{1}{2} \int_0^T ( s - \lb{s})  (\mathbf{H}(\tilde q_{\ub{s}}) - \mathbf{H}(q_{\ub{s}}) ) D_2  q_{\ub{s}} ds \nonumber \\
& \quad + \frac{1}{2} \int_0^T ( s - \lb{s}) (\mathbf{H}(\tilde q_{\lb{s}}) - \mathbf{H}(q_{\lb{s}}) )  D_2 q_{\lb{s}} ds \nonumber \\
\rn{3} \ &:= \ 
  -\frac{1}{2} \int_0^T ( s - \lb{s}) [ (\mathbf{H}(q_{\ub{s}}) - \mathbf{H}(q_{\lb{s}})) D_2 q_{\ub{s}}] \nonumber \\
 & \quad - \frac{1}{2} \int_0^T ( s - \lb{s}) [ \mathbf{H}(q_{\lb{s}}) (D_2 q_{\ub{s}} - D_2 q_{\lb{s}} ) ] ds 
  \;,  \quad \text{and} \nonumber \\
\rn{4} \ &:= \ - \int_0^T (T-s) \mathbf{H}(q_s) D_2 q_s ds \nonumber \\ 
& \quad + \frac{1}{2} \int_0^T (T-s) [ \mathbf{H}(q_{\lb{s}} ) D_2 q_{\lb{s}} + \mathbf{H}(q_{\ub{s}}) D_2 q_{\ub{s}} ] ds   \;. \nonumber
\end{align} 
By Assumptions~\ref{A1234} (A2) and (A3), and the condition $LT^2 \le 1/6$,  \begin{align} \label{D2qerr:1}
    \mnorm{\rn{1}} \ &\le \ L T^2 \, \max_{s \le T} \mnorm{D_2 q_s - D_2 \tilde q_s} \ \le \ (1/6) \, \max_{s \le T}\mnorm{D_2 q_s - D_2 \tilde q_s} \;, \\
  \mnorm{\rn{2}} \ &\le \   L_H T^2 \max_{s \le T} \norm{q_s - \tilde q_s} \, \max_{s \le T} \mnorm{D_2 q_s} \ \le \ (6/5) L_H T^3 \, \max_{s \le T} \norm{q_s - \tilde q_s}  \label{D2qerr:2} \;, \\
  \mnorm{\rn{3}} 
  \ &\le \   \frac{h L_H}{2}  \max_{s \le T} \mnorm{D_2 q_s}  \int_0^T  \norm{ \int_{\lb{s}}^{\ub{s}} v_r dr} ds  + \frac{h L}{2}  \int_0^T \mnorm{ \int_{\lb{s}}^{\ub{s}} D_2 v_r dr } ds \;, \nonumber \\
  \ &\le \ \frac{h^2 L_H T^2}{2} \frac{6}{5}  \max_{s \le T} \norm{v_s} +  \frac{h^2 L T}{2} \max_{s \le T}  \mnorm{D_2 v_s} \;, \nonumber \\
  \ &\le \ \frac{3}{5} h^2 L T + h^2 L_H ( \frac{7}{60} T \norm{x} + 
  \frac{18}{25} T^2 \norm{v} )     \label{D2qerr:3} \;,
\end{align}
where for \eqref{D2qerr:2} we used \eqref{apriori:5} and for \eqref{D2qerr:3} we used \eqref{max:v} and \eqref{apriori:6}.  Next, we apply Lemma~\ref{trapz} with 
$f(s) = -\mathbf{H}(q_s) D_2 q_s$ and 
\begin{align*}
&\mnorm{f'(s)} \ = \ \mnorm{(\differential^3 U(q_s) v_s) D_2 q_s + \mathbf{H}(q_s) D_2 v_s }  \ \le \ L_H \norm{v_s} \mnorm{D_2 q_s} + L \mnorm{D_2 v_s}  \;,    \\
&\quad  \ \le \ \frac{6}{5} ( L_H  ( \frac{7}{36} \norm{x} +\frac{6}{5} T \norm{v} )   + L ) \ =: \ B_1 \;, \\
&\mnorm{f''(s)} \ = \ \mnorm{(\differential^4 U(v_s,v_s) - \differential^3 U \nabla U - \mathbf{H}^2) D_2 q_s +
2 (\differential^3 U v_s) D_2 v_s
} \;, \\
&\quad \ \le \ (L_I \norm{v_s}^2 + L_H L \norm{q_s} + L^2) \mnorm{D_2 q_s} + 2 L_H \norm{v_s} \mnorm{D_2 v_s} \;, \\
&\quad \ \le \ [L_I ( \frac{L}{2} \norm{x}^2 + 3 \norm{v}^2 ) + L_H L \frac{7}{6} (\norm{x} + T \norm{v}) + L^2 ] \frac{6}{5} T +  L_H  [\frac{7}{3} L T \norm{x} + \frac{12}{5} \norm{v}] =: B_2   \;,
\end{align*}
where we used $L T^2 \le 1/6$,
Assumption~\ref{A1234},
\eqref{max:v}, \eqref{max:v2}, \eqref{apriori:5} and \eqref{apriori:6}.  Thus,
  \begin{align}
  \mnorm{\rn{4}} \le \frac{h^2}{12} \left( \frac{7}{5} L T + \frac{77}{90} L_H T \norm{x} + \frac{611}{150}  L_H T^2  \norm{v}  +  \frac{1}{10} L_I T \norm{x}^2   + \frac{18}{5}  L_I T^3  \norm{v}^2   \right)  
    \label{D2qerr:4} .
\end{align}
Inserting \eqref{D2qerr:1}, \eqref{D2qerr:2}, \eqref{D2qerr:3}, 
\eqref{D2qerr:4} and
Lemma~\ref{lem:verlet_error:1}
into \eqref{diff:D2q_D2tq} and simplifying gives \eqref{verlet_error:2}.
\end{proof}

\section{One-shot Coupling Bounds}
\label{sec:one_shot_bounds}

Here we prove bounds related to the one-shot coupling.  Throughout this section, we fix a duration parameter $T>0$ and a time step size $h>0$ such that
$T/h\in\mathbb N$ and 
\begin{equation}
    L(T^2 + T h) \ \le\ 1/6\, .
\end{equation}

\subsection{One-shot for $\tilde q_T(x,\xi) = \tilde q_T(y, \Phi(\xi))$}

By Corollary~\ref{cor:MuOr2004}, for any $x,y,v \in \mathbb{R}^d$, there exists a unique minimum $(\tilde{q}_t^{\star})_{0 \le t \le T}$ of the action sum satisfying the endpoint conditions $\tilde{q}_0^{\star} = y$ and $\tilde{q}_T^{\star} = \tilde{q}_T(x, v)$.  In terms of this minimum, we introduce a function $\Phi: \mathbb{R}^{d}  \to \mathbb{R}^d$ defined as $\Phi(v) \ := \ - D_1 L_h(\tilde{q}_0^{\star}, \tilde{q}_{h}^{\star})$, which by definition satisfies $\tilde q_T(x, v) = \tilde q_T^{\star}(y, \Phi(v))$. Here we develop some bounds for $\Phi(v)$ and $D \Phi(v)$.  

\begin{lem}\label{lem:oneshot:1:a}
Suppose that Assumption~\ref{A1234} (A1)-(A2) hold.  Then for any $x,y,v \in \mathbb{R}^d$ such that $\tilde q_T(x,v) = \tilde q_T(y,\Phi(v))$, we have \begin{equation} \label{Phiv}
T \norm{\Phi(v)-v} \ \le \ (3/2) \ \norm{x - y} \;.
\end{equation}
\end{lem}

\begin{proof}
Let $u=\Phi(v)$.  Integrating \eqref{velVerlet} yields \begin{align*}
   \tilde{q}_T(x,u) & = x + T u - \frac{1}{2} \int_0^T (T - s) \left[ \nabla U(\tilde{q}_{\lb{s}}) + \nabla U(\tilde{q}_{\ub{s}}) \right] ds  \\
    & + \frac{1}{2} \int_0^T (s - \lb{s}) \left[ \nabla U(\tilde{q}_{\ub{s}}) - \nabla U(\tilde{q}_{\lb{s}}) \right] ds 
\end{align*}
and since $\tilde{q}_T(x,u) =\tilde{q}_T(y,v)$, we obtain \begin{align*}
    T \norm{u - v} &\le \norm{x - y}   +  L T^2 \max_{s \le T} \norm{\tilde{q}_s(x,u)-\tilde{q}_s(y,v)} \\
    &\le \norm{x - y}   +  L T^2  (1+ L (T^2+T h))\max\left(\norm{x-y},\norm{(x-y)+T(u-v)}\right) \\
    &\le \norm{x - y}   +  (7/36) (\norm{x-y}+T\norm{u-v} ) \;,
\end{align*}
where in the second to last step we used \eqref{apriori:3} from Lemma~\ref{lem:apriori} and in the last step we used $L (T^2 + T h) \le 1/6$.  Simplifying and using $(1+7/36)/(1-7/36) < 3/2$ gives  \eqref{Phiv}.
\end{proof}

\begin{lem}\label{lem:oneshot:1:b}
Suppose that Assumption~\ref{A1234} (A1)-(A3) hold. Then for any $x,y,v \in \mathbb{R}^d$ such that $\tilde q_T(x,v) = \tilde q_T(y,\Phi(v))$, we have \begin{equation} \label{DPhiv}
\mnorm{D \Phi(v)- I_d} \ \le \ (1/2) \min(1, 11 \, L_H  T^2 \, \norm{x - y}) \;.   
\end{equation}
\end{lem}

\begin{proof}
 Introduce the shorthand $\tilde{q}_s^{(1)} : = \tilde{q}_s(x,v)$ and $\tilde{q}_s^{(2)} : = \tilde{q}_s(y,\Phi(v))$ for any $s \in [0,T]$.
Differentiating both sides of $\tilde{q}_T(x,v) = \tilde{q}_T(y,\Phi(v))$ with respect to $v$ yields \begin{equation} \label{eq:DPhimId}
\begin{aligned}
&T (D \Phi(v) - I_d) = \\
& + \frac{1}{2} \int_0^T (T - s) \left[ \mathbf{H}(\tilde{q}_{\lb{s}}^{(2)})  D_2 \tilde{q}_{\lb{s}}^{(2)}   - 
 \mathbf{H}(\tilde{q}_{\lb{s}}^{(1)})  D_2 \tilde{q}_{\lb{s}}^{(1)} \right] ds  \\
& + \frac{1}{2} \int_0^T (T - s) \left[ \mathbf{H}(\tilde{q}_{\ub{s}}^{(2)})  D_2 \tilde{q}_{\ub{s}}^{(2)}  - 
 \mathbf{H}(\tilde{q}_{\ub{s}}^{(1)})   D_2 \tilde{q}_{\ub{s}}^{(1)}
 \right] ds  \\
    & - \frac{1}{2} \int_0^T (s - \lb{s}) \left[  \mathbf{H}(\tilde{q}_{\ub{s}}^{(2)})  D_2 \tilde{q}_{\ub{s}}^{(2)}    -   \mathbf{H}(\tilde{q}_{\ub{s}}^{(1)})   D_2 \tilde{q}_{\ub{s}}^{(1)} \right] ds \\
        & - \frac{1}{2} \int_0^T (s - \lb{s}) \left[ \mathbf{H}(\tilde{q}_{\lb{s}}^{(1)})  D_2 \tilde{q}_{\lb{s}}^{(1)}  -  \mathbf{H}(\tilde{q}_{\lb{s}}^{(2)}) D_2 \tilde{q}_{\lb{s}}^{(2)}    \right] ds \\
& + \frac{1}{2} \int_0^T (T-s) \left[\mathbf{H}(\tilde{q}_{\lb{s}}^{(2)}) D_2 \tilde{q}_{\lb{s}}^{(2)} + \mathbf{H}(\tilde{q}_{\ub{s}}^{(2)}) D_2 \tilde{q}_{\ub{s}}^{(2)} \right] (D \Phi(v) - I_d) ds  \\
 & - \frac{1}{2} \int_0^T (s-\lb{s}) \left[\mathbf{H}(\tilde{q}_{\ub{s}}^{(2)}) D_2 \tilde{q}_{\ub{s}}^{(2)} - \mathbf{H}(\tilde{q}_{\lb{s}}^{(2)}) D_2 \tilde{q}_{\lb{s}}^{(2)} \right] (D \Phi(v) - I_d) ds \;.
 \end{aligned}
 \end{equation}
 By using Assumption~\ref{A1234} (A2),  \eqref{apriori:5},  and $L T^2 \le 1/6$, note that \begin{equation} \label{DPhiv:1}
   \mnorm{D \Phi(v) - I_d} \le (6/5) 2 L T^2 / (1- (6/5) L T^2) \le 1/2 \;.
\end{equation}   We can also rewrite \eqref{eq:DPhimId} as 
 \begin{align*}
 &T (D \Phi(v) - I_d) = \\
&+\frac{1}{2} \int_0^T (T - s) \left[ \mathbf{H}(\tilde{q}_{\lb{s}}^{(2)}) ( D_2 \tilde{q}_{\lb{s}}^{(2)} - D_2 \tilde{q}_{\lb{s}}^{(1)} )   - 
( \mathbf{H}(\tilde{q}_{\lb{s}}^{(1)}) -  \mathbf{H}(\tilde{q}_{\lb{s}}^{(2)})) D_2 \tilde{q}_{\lb{s}}^{(1)} \right] ds  \\
& + \frac{1}{2} \int_0^T (T - s) \left[ \mathbf{H}(\tilde{q}_{\ub{s}}^{(2)}) ( D_2 \tilde{q}_{\ub{s}}^{(2)} - D_2 \tilde{q}_{\ub{s}}^{(1)}) - 
( \mathbf{H}(\tilde{q}_{\ub{s}}^{(1)}) - \mathbf{H}(\tilde{q}_{\ub{s}}^{(2)}) )  D_2 \tilde{q}_{\ub{s}}^{(1)}
 \right] ds  \\
    & - \frac{1}{2} \int_0^T (s - \lb{s}) \left[  \mathbf{H}(\tilde{q}_{\ub{s}}^{(2)}) ( D_2 \tilde{q}_{\ub{s}}^{(2)}  -D_2 \tilde{q}_{\ub{s}}^{(1)})  -  ( \mathbf{H}(\tilde{q}_{\ub{s}}^{(1)}) - \mathbf{H}(\tilde{q}_{\ub{s}}^{(2)}) )  D_2 \tilde{q}_{\ub{s}}^{(1)} \right] ds \\
        & + \frac{1}{2} \int_0^T (s - \lb{s}) \left[ \mathbf{H}(\tilde{q}_{\lb{s}}^{(2)}) ( D_2 \tilde{q}_{\lb{s}}^{(2)} - D_2 \tilde{q}_{\lb{s}}^{(1)}) - ( \mathbf{H}(\tilde{q}_{\lb{s}}^{(1)}) - \mathbf{H}(\tilde{q}_{\lb{s}}^{(2)}) ) D_2 \tilde{q}_{\lb{s}}^{(1)}    \right] ds \\
& + \frac{1}{2} \int_0^T (T-s) \left[\mathbf{H}(\tilde{q}_{\lb{s}}^{(2)}) D_2 \tilde{q}_{\lb{s}}^{(2)} + \mathbf{H}(\tilde{q}_{\ub{s}}^{(2)}) D_2 \tilde{q}_{\ub{s}}^{(2)} \right] (D \Phi(v) - I_d) ds  \\
 & - \frac{1}{2} \int_0^T (s-\lb{s}) \left[\mathbf{H}(\tilde{q}_{\ub{s}}^{(2)}) D_2 \tilde{q}_{\ub{s}}^{(2)} - \mathbf{H}(\tilde{q}_{\lb{s}}^{(2)}) D_2 \tilde{q}_{\lb{s}}^{(2)} \right] (D \Phi(v) - I_d) ds \;.
\end{align*}
By using Assumption~\ref{A1234} (A2),  \eqref{apriori:5}, Assumption~\ref{A1234} (A3), \eqref{apriori:7}, and $L T^2 \le 1/6$, we get \begin{align*}
    (4/5) \, T \,  \mnorm{D \Phi(v) - I_d}   
    & \le 
     ( 1 - L T^2 (6/5) ) \, T \, \mnorm{D \Phi(v) - I_d}  \\
    & \le   L T^2 \max_{s \le T} \mnorm{ D^2 \tilde{q}_s^{(1)} -  D^2 \tilde{q}_s^{(2)} } + L_H (6/5) T^3 \max_{s \le T} \norm{\tilde{q}_s^{(1)} - \tilde{q}_s^{(2)}} \\
    & \le  (42/25) \, L_H T^3 \, ( \norm{x-y} + T \norm{\Phi(v) - v} )  \;.
\end{align*} Inserting \eqref{Phiv}, using $21/4 < 11/2$, simplifying and inserting \eqref{DPhiv:1} gives \eqref{DPhiv}.  
\end{proof}

\subsection{One-shot for $\tilde q_T(x,\xi) = q_T(x, \Phi(\xi))$}


By Corollary~\ref{cor:MuOr2004}, for any $x,v \in \mathbb{R}^d$, there exists a unique minimum $(q_t^{\star})_{0 \le t \le T}$ of the action integral satisfying the endpoint conditions $q_0^{\star} = x$ and $q_T^{\star} = \tilde{q}_T(x, v)$.  In terms of this minimum, we introduce a function $\Phi: \mathbb{R}^{d}  \to \mathbb{R}^d$ defined as $\Phi(v) \ := \ v_0^{\star}$, which by definition satisfies $\tilde q_T(x, v) =  q_T^{\star}(x, \Phi(v))$. Here we develop some bounds for $\Phi(v)$ and $D \Phi(v)$.

\begin{lem}\label{lem:oneshot:2:a}
Suppose that Assumption~\ref{A1234} (A1)-(A3) hold. Then for any $x,v \in \mathbb{R}^d$ such that $\tilde q_T(x,v) = q_T(x,\Phi(v))$, we have \begin{align} 
& T \norm{\Phi(v)-v}  \ \le \ \frac{7}{6}  h^2 \big(\frac{7}{45} L  \norm{x}   + \frac{1547}{1800} L T \norm{v}
 + \frac{1}{120} L_H \norm{x}^2 + \frac{3}{10} L_H T^2 \norm{v}^2  \big) \;.   \label{Phiv2}
\end{align}
\end{lem}

\begin{proof}
Introduce the shorthand $\tilde q_T = \tilde q_T(x,v)$, $q_T^{(1)} = q_T(x,v)$ and $q_T^{(2)} = q_T(x,\Phi(v))$.  Since  $\tilde q_T = q_T^{(2)}$  we have $q_T^{(2)}  - q_T^{(1)} = \tilde q_T - q_T^{(1)}$ which implies that \begin{align}
 T \, \norm{ \Phi(v)  - v }  \ &= \ \norm{ \int_0^T ( T - s) [ \nabla U(q_{s}^{(2)}) - \nabla U(q_{s}^{(1)}) ] ds + \tilde{q}_T - q_T^{(1)} } \nonumber \\
 & \le \
 \frac{L T^2}{2} \max_{s \le T}  \norm{q_s^{(2)} - q_s^{(1)} }
 + \norm{\tilde{q}_T - q_T^{(1)} } \;.
 \label{diff:Phiv_v}
\end{align}  
Moreover, by \eqref{apriori:3} with $h=0$, we obtain $\max_{s \le T}  \norm{q_s^{(2)} - q_s^{(1)} } \le (7/6) T \norm{\Phi(v) - v}$.  Inserting this latter inequality into \eqref{diff:Phiv_v} and simplifying gives \[
T \, \norm{ \Phi(v)  - v } \ \le \ (7/6) \max_{s \le T} \norm{ \tilde  q_s - q_s^{(1)} }   \ = \   (7/6) \max_{s \le T} \norm{ \tilde  q_s(x,v) - q_s(x,v) } \;. 
\] Inserting \eqref{verlet_error:1} gives the required result.  
\end{proof}

\begin{lem}\label{lem:oneshot:2:b}
Suppose that Assumption~\ref{A1234} (A1)-(A4) hold.  Then for any $x,v \in \mathbb{R}^d$ such that $\tilde q_T(x,v) = q_T(x,\Phi(v))$, we have \begin{equation} \label{DPhiv2}
\begin{aligned} 
&  \mnorm{D \Phi(v)- I_d} \ \le \   \min\Bigg(\frac{1}{2}, h^2 \Big(  \frac{43}{45} L  + \frac{1946}{6075}   L_H   \norm{x}   + \frac{873707}{486000}  L_H T \norm{v} \\
&\qquad +  ( \frac{121}{5400}  L_H^2 T^2 + \frac{1}{90} L_I  ) \norm{x}^2  
 + ( \frac{121}{150}  L_H^2 T^4 + \frac{2}{5}  L_I T^2 )  \norm{v}^2 \Big) \Bigg) \;.
\end{aligned} 
\end{equation}
\end{lem}

\begin{proof}
Introduce the shorthand $\tilde q_T = \tilde q_T(x,v)$, $q_T^{(1)} = q_T(x,v)$ and $q_T^{(2)} = q_T(x,\Phi(v))$. The derivative of $q_T^{(2)}  - q_T^{(1)} = \tilde q_T - q_T^{(1)}$ with respect to $v$ yields \begin{align}
 & T \, ( D \Phi(v)  - I_d)  \ = \  
 D_2 \tilde q_T - D_2 q_T^{(1)} 
 + \int_0^T ( T - s) [ \mathbf{H}(q_{s}^{(2)}) - \mathbf{H} (q_{s}^{(1)}) ] D_2 q_{s}^{(1)} ds \nonumber \\ 
 & \qquad + \int_0^T ( T - s) \mathbf{H}(q_{s}^{(2)}) [ D_2 q_s^{(2)} - D_2 q_s^{(1)} + D_2 q_s^{(2)} ( D \Phi(v) - I_d)]  ds \;. \nonumber 
\end{align}  
By Assumption~\ref{A1234} (A2)-(A3) and \eqref{apriori:5}, \begin{align}
     & T \, \mnorm{ D \Phi(v)  - I_d } \ \le\ \mnorm{D_2 \tilde q_T - D_2 q_T^{(1)} } +
     \frac{L_H T^2}{2}  \frac{6 T}{5}  \max_{s \le T} \norm{ q_s^{(2)} - q_s^{(1)}} \nonumber \\
     &\qquad + \frac{L T^2}{2} \max_{s \le T} \mnorm{ D_2 q_s^{(2)} - D_2 q_s^{(1)} } + \frac{L T^2}{2}  \frac{6 T}{5} \mnorm{ D \Phi(v)  - I_d } \;. \nonumber
\end{align} Since $L T^2 \le 1/6$, and inserting  \eqref{apriori:3} and \eqref{apriori:7} with $h=0$,  we obtain  
 \begin{align}
     & T \, \mnorm{ D \Phi(v)  - I_d }  
      \le \frac{10}{9} \mnorm{D_2 \tilde q_T - D_2 q_T^{(1)} } + L_H T^3 \frac{14}{15} \left( T \norm{\Phi(v) - v} \right)  \;. \nonumber
\end{align}
Inserting \eqref{Phiv2} and \eqref{verlet_error:2}, simplifying, and combining with a bound similar to \eqref{DPhiv:1} (and therefore omitted), gives \eqref{DPhiv2}. 
\end{proof}

\section{Proofs of Results for Mean-Field $U$}

In this section, we provide the remaining ingredients needed to prove Theorem~\ref{thm:tv_to_equilibrium_mf} and Theorem~\ref{thm:tv_IM_accuracy_mf} from Section~\ref{ex:mean_field}.

\subsection{Preliminaries}

Consider the mean-field model with potential energy function given in \eqref{mf}.  Throughout this section, we assume  that $V,W$ are functions in $C^4(\mathbb{R}^k)$ with   \begin{align*}
L &=\sup\mnorm{\differential^2V},~~
L_H=\sup\mnorm{\differential^3V},~~ 
~~L_I=\sup\mnorm{\differential^4V}, \\
\tilde{L} &=\sup\mnorm{\differential^2W},~~ \tilde{L}_H=\sup\mnorm{\differential^3W},
~~\tilde{L}_I=\sup\mnorm{\differential^4 W},
\end{align*}
which we assume are all bounded.  This assumption is not needed in every statement given in this section, but for simplicity, we assume it throughout.   Moreover,  we assume that $T>0$ and the time step size $h>0$ are such that
$T/h\in\mathbb N$ and
\begin{equation}
    \label{eq::Tconstraint_mf}
    (L + 4 \epsilon \tilde{L}) (T^2 + T h) \ \le\ 1/6\, .
\end{equation}
As we discuss next, the constant $(L + 4 \epsilon \tilde{L})$ represents an effective Lipschitz constant for the gradient of $U$.

Define $\mathbf{H}_{ij}(x) = \nabla_{ij} U(x) = \partial^2 U(x) / \partial x^i \partial x^j$ for $i,j \in \{1, \dots, n \}$.
From \eqref{mf}, note that for all $x, y \in \mathbb{R}^{n k}$ and $i,j \in \{1, \dots, n \}$,  \begin{align}
\norm{\nabla_i U(x)}  \, &\le \,  (L+2 \epsilon \tilde{L}) \norm{x^i} + \frac{2 \epsilon \tilde{L}}{n} \sum_{\ell \ne i} \norm{x^{\ell}} ,  \\
\label{eq:dG_i_mf}
\norm{\nabla_i U(x) - \nabla_i U(y)}  \, &\le \,  (L + 2 \epsilon \tilde{L}) \norm{x^i - y^i} + \frac{2 \epsilon \tilde{L} }{n} \sum_{\ell \ne i}  \norm{x^{\ell} - y^{\ell}}  ,  \\
\label{eq:L_mf}
\mnorm{\mathbf{H}_{ij}(x)} \ &\le \ \begin{cases}
L + 2 \epsilon \tilde{L}  & \text{if $i=j$},  \\
2 \epsilon \tilde L / n & \text{else},
\end{cases} \\
\label{eq:dH_ij_mf}
\mnorm{\mathbf{H}_{ij}(x) - \mathbf{H}_{ij}(y)} \, &\le \,  \begin{cases}
( L_H + 2 \epsilon \tilde{L}_H ) \norm{x^i - y^i} + \frac{2 \epsilon \tilde{L}_H}{n} \sum_{\ell \ne i}  \norm{x^{\ell} - y^{\ell}} & \text{if $i=j$},  \\
\frac{2 \epsilon \tilde L_H}{n} ( \norm{x^i - y^i} + \norm{x^j - y^j})   & \text{else},
\end{cases} 
\end{align}
where we used $\sum_{\ell \ne i} = \sum_{\ell=1, \ell \ne i}^n$.   Additionally, for any $i_1, i_2, i_3, i_4 \in \{ 1, \dots, n \}$, \begin{align}
\label{eq:LH_mf}
    \mnorm{\frac{\partial^3 U(x)}{\partial x^{i_1} \partial x^{i_2} \partial x^{i_3}}}  \, &\le \,  \begin{cases} L_H + 2 \epsilon \tilde{L}_H 
    & \text{if all indices are equal},
    \\
    2 \epsilon \tilde{L}_H/n & \text{if exactly two indices are equal}, \\
    0 & \text{else},
    \end{cases}  \\
\label{eq:LI_mf}
    \mnorm{\frac{\partial^4 U(x)}{\partial x^{i_1} \partial x^{i_2} \partial x^{i_3} \partial x^{i_4}}}  \, &\le \,  \begin{cases} L_I + 2 \epsilon \tilde{L}_I & \text{if all indices are equal}, \\
    2 \epsilon \tilde{L}_I/n & \text{if exactly three indices are equal}, \\
    0 & \text{else}.
    \end{cases} 
\end{align}

As in Section~\ref{sec:one_shot_bounds} for general $U$,  since we assume \eqref{eq::Tconstraint_mf}, we can invoke Corollary~\ref{cor:MuOr2004},  to obtain the existence/uniqueness of a function $\Phi: \mathbb{R}^{n k}
 \to \mathbb{R}^{nk}$ that satisfies either
$\tilde{q}_T(x,v) = \tilde{q}_T(y,\Phi(v))$ for any $x,y,v \in \mathbb{R}^{n k}$ or $\tilde{q}_T(x,v) = q_T(x,\Phi(v))$ for any $x,v \in \mathbb{R}^{n k}$.

\subsection{TV bounds and regularization by one-shot couplings}

Here we prove that the transition kernel of uHMC for mean-field $U$ has a regularizing effect with a better dimension dependence than for general $U$.

\begin{lem} \label{lem:overlap1_mf}
 For any $x,y,v \in \mathbb{R}^{n k}$, let $\Phi: \mathbb{R}^{n k} \to \mathbb{R}^{n k}$ be the map defined by $\tilde{q}_T(x,v) = \tilde{q}_T(y,\Phi(v))$.  Then 
\begin{equation} \label{tvbound1_mf}
\begin{aligned}
& \TV{\delta_x \tilde \pi}{\delta_y \tilde \pi}  \ \le \ 
\TV{\law(\xi)}{\law(\Phi(\xi))}  \\
& \qquad \ \le \ \frac{3}{2} \, \left( T^{-2}  + 34  k (L_H+8 \epsilon \tilde{L}_H)^2  T^4 \right)^{1/2} \, \sum_{\ell =1}^n \norm{x^{\ell} - y^{\ell} } .
\end{aligned}
\end{equation}
\end{lem}

Note that the prefactor in the upper bound of \eqref{tvbound1_mf} does not depend on the number $n$ of particles.  In this sense \eqref{tvbound1_mf} is an improvement over \eqref{tvbound1}.  

\begin{proof}
This proof is similar to the proof of Lemma~\ref{lem:overlap1}, except here we directly insert the results in Lemma~\ref{lem:oneshot:1:a_mf} and~\ref{lem:oneshot:1:b_mf} into Lemma~\ref{lem:overlap} to obtain \begin{align*}
& \TV{\law(\xi)}{\law(\Phi(\xi)}^2  \ \le \
\left( \frac{9}{4} \, T^{-2}  + \frac{2401}{32} \, k \, (L_H+8 \epsilon \tilde{L}_H)^2 \, T^4 \right)  \, \left( \sum_{\ell =1}^n \norm{x^{\ell} - y^{\ell} } \right)^2 .
\end{align*}
 Taking square roots and inserting $ 34> (4/9)(2401/32)$ gives \eqref{tvbound1}.
\end{proof}

Analogous to Lemmas~\ref{lem:overlap1} and~\ref{lem:regularization}, Lemma~\ref{lem:overlap1_mf} similarly implies that the transition kernel of uHMC has a regularizing effect in the following sense.

\begin{lem} \label{lem:regularization_mf}
 For any  $\nu,\eta \in \mathcal{P}(\mathbb{R}^{nk})$,
\begin{equation} \label{eq:regularization_mf}
\TV{\eta \tilde \pi}{\nu \tilde \pi} \ \le \ 
 \frac{3}{2} \, \left( T^{-2}  + 34  k (L_H+8 \epsilon \tilde{L}_H)^2  T^4 \right)^{1/2} \W_{\ell_1}^1(\eta, \nu)  \;.
\end{equation}
\end{lem}

The proof of Lemma~\ref{lem:regularization_mf} is very similar to the proof of Lemma~\ref{lem:regularization} except that it involves the $\W_{\ell_1}^1$ distance rather than the $\W^1$ distance, and therefore, omitted.

\begin{lem} \label{lem:overlap2_mf}
For any $x,v \in \mathbb{R}^{n k}$, let $\Phi: \mathbb{R}^{n k} \to \mathbb{R}^{n k}$ be the map defined by $\tilde{q}_T(x,v) = q_T(x,\Phi(v))$.  Then
\begin{equation} \label{tvbound2_mf}
\begin{aligned}
& \TV{\delta_x \pi}{\delta_x \tilde \pi} \ \le \ \TV{\law(\xi)}{\law(\Phi(\xi))} \\ 
& \quad\le \
h^2 \, \Big[ 
 n^2 k \Big(  17 (L+4 \epsilon \tilde{L})^2 +
 28 (L_H +8 \epsilon \tilde{L}_H)^2 T^2
 + 
 104 k (L_H +8 \epsilon \tilde{L}_H)^2 T^2 
 \\  &  \quad +
 180 (2 k +k^2) \big( (L_I + 14 \epsilon \tilde{L}_I)^2 T^2 + (L_H +8 \epsilon \tilde{L}_H)^2  T^4 \big)^2  \Big)  \\ 
 & \quad  + n \big( 10 k  (L_H +8 \epsilon \tilde{L}_H)^2 +  7 (L+4 \epsilon \tilde{L})^2 T^{-2} \big) \sum_{\ell=1}^n \norm{x^{\ell}}^2 \\
 & \quad +  n \Big(40 k \big( (L_I + 14 \epsilon \tilde{L}_I) + (L_H +8 \epsilon \tilde{L}_H)^2  T^2 \big)^2 \\ 
 & \quad + 7 (L_H +8 \epsilon \tilde{L}_H)^2 T^{-2}
 \Big) \sum_{\ell=1}^n \norm{x^{\ell}}^4  \Big]^{1/2} .
\end{aligned}
\end{equation}
\end{lem}

\begin{proof}
The proof is similar to the proof of Lemma~\ref{lem:overlap2} except here we directly insert the results of Lemmas \ref{lem:oneshot:2:a_mf} and~\ref{lem:oneshot:2:b_mf} into Lemma~\ref{lem:overlap}, and then use the Cauchy-Schwarz inequality, to obtain \begin{align*}
& \TV{\law(\xi)}{\law(\Phi(\xi)}^2
\ \le \ E\left[  \norm{\Phi(\xi)-\xi}^2 + 2 \mnorm{D \Phi(\xi) - I_d}_F^2 \right] \\
& \le \  h^4 \, E\Big[ 
10 n^2 k  (L+4 \epsilon \tilde{L})^2 + 
  n \big( 10 k  (L_H +8 \epsilon \tilde{L}_H)^2 +  7 (L+4 \epsilon \tilde{L})^2 T^{-2} \big) \sum_{\ell=1}^n \norm{x^{\ell}}^2  \\
& \quad  + 
  n \big( 90 k  (L_H +8 \epsilon \tilde{L}_H)^2 T^2 +  7 (L+4 \epsilon \tilde{L})^2  \big) \sum_{\ell=1}^n \norm{v^{\ell}}^2 \\
 & \quad + n \Big( 10 k (L_I + 14 \epsilon \tilde{L}_I)^2 + 
 7 (L_H +8 \epsilon \tilde{L}_H)^2 T^{-2} \\
& \qquad + 40 k (L_H +8 \epsilon \tilde{L}_H)^2 (L_I + 14 \epsilon \tilde{L}_I) T^2 + 40 k (L_H +8 \epsilon \tilde{L}_H)^4 T^4
 \big) \sum_{\ell=1}^n \norm{x^{\ell}}^4 \\ 
& \quad +  n \Big( 10 k (L_I + 14 \epsilon \tilde{L}_I)^2 T^4 + 
 7 (L_H +8 \epsilon \tilde{L}_H)^2 T^{2} 
\\
& \qquad + 60 k (L_H +8 \epsilon \tilde{L}_H)^2 (L_I + 14 \epsilon \tilde{L}_I) T^6 + 90 k (L_H +8 \epsilon \tilde{L}_H)^4 T^8
 \Big) \sum_{\ell=1}^n \norm{v^{\ell}}^4 \Big] \;. 
 \end{align*}
We then use 
$E[ \norm{\xi^{\ell}}^2 ] = k$ and
$E[ \norm{\xi^{\ell}}^4 ] = 2 k (k+2)$ to obtain
 \begin{align*}
 & \TV{\law(\xi)}{\law(\Phi(\xi)}^2 \\
 & \le \  h^4 \, \Big[ 
 n^2 k \Big(  17 (L+4 \epsilon \tilde{L})^2 +
 28 (L_H +8 \epsilon \tilde{L}_H)^2 T^2
 + 
 104 k (L_H +8 \epsilon \tilde{L}_H)^2 T^2 
 \\  &  \qquad +
 (40 k  +
 20 k^2) (L_I + 14 \epsilon \tilde{L}_I)^2 T^4 + (240 k + 120 k^2) (L_H +8 \epsilon \tilde{L}_H)^2 (L_I + 14 \epsilon \tilde{L}_I) T^6  \\  &  \qquad  + (360 k  
 + 180 k^2 )  (L_H +8 \epsilon \tilde{L}_H)^4 T^8 \Big) 
  \\ 
 & \quad  + n \big( 10 k  (L_H +8 \epsilon \tilde{L}_H)^2 +  7 (L+4 \epsilon \tilde{L})^2 T^{-2} \big) \sum_{\ell=1}^n \norm{x^{\ell}}^2 \\
 & \quad + n \Big( 10 k (L_I + 14 \epsilon \tilde{L}_I)^2 + 
 7 (L_H +8 \epsilon \tilde{L}_H)^2 T^{-2} \\
& \qquad + 40 k (L_H +8 \epsilon \tilde{L}_H)^2 (L_I + 14 \epsilon \tilde{L}_I) T^2 + 40 k (L_H +8 \epsilon \tilde{L}_H)^4 T^4
 \big) \sum_{\ell=1}^n \norm{x^{\ell}}^4
 \Big] \;.
\end{align*} 
Taking square roots and simplifying gives \eqref{tvbound2_mf}.
\end{proof}

 Lemma~\ref{lem:overlap2_mf} implies the following bound on the TV distance between $\mu \pi$ and $\mu \tilde \pi$.  

\begin{lem} \label{lem:tv_strong_error_mf}
We have 
$\TV{\mu \pi}{\mu \tilde \pi}  \le C  h^2$,
where $C$ is defined in \eqref{eq:C_mf}.
\end{lem}

The proof of this result is similar to the proof of Lemma~\ref{lem:regularization} and therefore omitted.

\subsection{A priori estimates for the dynamics}

Here we develop estimates on the dynamics \eqref{velVerlet} for the special case of the mean-field model. 
Let $\tilde q_t(x,v) = (\tilde q_t^1(x,v), \dots, \tilde q_t^n(x,v))$ and similarly for $\tilde v_t(x,v)$.

\begin{lem} \label{lem:apriori_mf}
For any $x,y,u,v\in\R^{n k}$ and $i \in \{1, \dots, n\}$,
	\begin{align}
	\begin{split} \label{apriori:0a_mf}
& \max_{s \le T} \norm{\tilde{q}_s^i(x,v)}
\, \le \, (1+(L+2 \epsilon \tilde{L}) (T^2+T h)) \max( \norm{x^i}, \norm{x^i + T v^i})  \\ 	
& \qquad + \frac{2 \epsilon \tilde{L} (T^2 + T h)}{n} \sum_{\ell \ne i} \max_{s \le T} \norm{\tilde{q}_s^i(x,v)},
\end{split} \\
\begin{split} \label{apriori:0b_mf} 
& \max_{s \le T} \norm{\tilde{v}_s^i(x,v)} 
\, \le \, \norm{v^i} + \frac{2 \epsilon \tilde{L} T}{n} (1+ (L+2 \epsilon \tilde{L}) (T^2 + T h)) \sum_{\ell \ne i} \max_{s \le T} \norm{\tilde{q}_s^i(x,v)} 	 \\	
 	& \qquad 
 	+ (L+2 \epsilon \tilde{L}) T (1+ (L+2 \epsilon \tilde{L}) (T^2 + T h)) \max( \norm{x^i}, \norm{x^i + T v^i}),
 	\end{split} 
 	\end{align}

\begin{align}
 	\begin{split} \label{apriori:1_mf}
&\sum_{\ell=1}^n \max_{s\leq T}\norm{\tilde{q}^{\ell}_s(x,v)} \, \leq \, (1+(L+4 \epsilon \tilde{L}) (T^2+T h)) \sum_{\ell=1}^n \max(\norm{x^{\ell}},\norm{x^{\ell}+T v^{\ell}}),\\
\end{split} \\
\begin{split} \label{apriori:2_mf}
&\sum_{\ell=1}^n \max_{s\leq T}\norm{\tilde{v}^{\ell}_s(x,v)} \, \leq \,
\sum_{\ell=1}^n \norm{v^{\ell}}  \\
& \qquad + (L + 4 \epsilon \tilde{L}) T (1+(L + 4 \epsilon \tilde{L}) (T^2+T h)) \sum_{\ell=1}^n
\max(\norm{x^{\ell}},\norm{x^{\ell}+T v^{\ell}}), \\
\end{split} \\
\begin{split} \label{apriori:3_mf} 
&\sum_{\ell=1}^n \max_{s\leq T} \norm{\tilde{q}^{\ell}_s(x,u)-\tilde{q}^{\ell}_s(y,v)} 
\\
& \qquad \leq \, (1+ (L + 4 \epsilon \tilde{L}) (T^2+T h))\sum_{\ell=1}^n \max\left(\norm{x^{\ell}-y^{\ell}},\norm{x^{\ell}-y^{\ell}+T (u^{\ell}-v^{\ell})}\right) .
\end{split}
\end{align}
\end{lem}

Lemma~\ref{lem:apriori_mf} is the mean-field analog of Lemma~\ref{lem:apriori} and is contained in Lemmas~12 and 13 of \cite{BoSc2020}.  Hence, a proof is omitted. 

\subsection{A priori estimates for the velocity derivative flow}

Here we develop estimates on the velocity derivative flows $\partial_{v^j} \tilde q_s^i = \partial \tilde q_s^i/\partial v^j$.

\begin{lem} \label{lem:apriori2_mf}
For any $x,y,u,v\in\R^{n k}$ and $j \in \{1, \dots, n \}$,  
	\begin{align}
	\label{apriori:5_mf}
& \sum_{\ell=1}^n \max_{s\leq T}\mnorm{ \partial_{v^j} \tilde q_s^{\ell}(x,v) } \, \leq \, (6/5) \, T \, , \\
	\label{apriori:5b_mf}
& \sum_{\ell=1}^n \max_{s\leq T}\mnorm{ \partial_{v^{\ell}} \tilde q_s^j(x,v) } \, \leq \, (7/5) \, T   \, , \\
	\label{apriori:6_mf}
&\sum_{\ell=1}^n \max_{s\leq T}\mnorm{ \partial_{v^j} \tilde{v}_s^{\ell}(x,v) } \, \leq \, (6/5) \, , \quad\mbox{and} \\ 
\begin{split} \label{apriori:7_mf} 
&\sum_{i=1}^n  \sum_{\ell=1}^n \max_{s\leq T} \mnorm{\partial_{v^i}  \tilde{q}^{\ell}_s(x,u)-\partial_{v^i}  \tilde{q}^{\ell}_s(y,v)} \, \leq \,   (42/25) (L_H + 8 \epsilon \tilde{L}_H)  T^3  \\
& \qquad \times (1+(L + 4 \epsilon \tilde{L}) (T^2+T h)) \sum_{\ell=1}^n \max\left(\norm{x^{\ell}-y^{\ell}},\norm{x^{\ell}-y^{\ell}+T (u^{\ell}-v^{\ell})}\right)  . 
\end{split}
\end{align}
\end{lem}

Lemma~\ref{lem:apriori2_mf} is the mean-field analog of Lemma~\ref{lem:apriori2}.

\begin{proof}
From \eqref{velVerlet}, note that for any $\ell, j \in \{1, \dots, n\}$,  \begin{align} 
\begin{split} \label{dervelVerlet_mf}
& \partial_{v^j} \tilde{q}_T^{\ell}(x,u)   =  T I_k \delta_{\ell j} \\
& \qquad - \frac{1}{2} \int_0^T (T - s) \sum_{i=1}^n \left[  \mathbf{H}_{\ell i}(\tilde{q}_{\lb{s}}) \partial_{v^j}  \tilde{q}^i_{\lb{s}}  + \mathbf{H}_{\ell i}(\tilde{q}_{\ub{s}}) \partial_{v^j}  \tilde{q}^i_{\ub{s}} \right] ds  \\
    & \qquad + \frac{1}{2} \int_0^T (s - \lb{s}) \sum_{i=1}^n \left[ \mathbf{H}_{\ell i}(\tilde{q}_{\ub{s}}) \partial_{v^j} \tilde{q}^i_{\ub{s}} - \mathbf{H}_{\ell i}(\tilde{q}_{\lb{s}}) \partial_{v^j} \tilde{q}^i_{\lb{s}} \right] ds \;,  
    \end{split} \\
& \partial_{v^j} \tilde{v}_T^{\ell}(x,u)   = I_k \delta_{\ell j} - \frac{1}{2} \int_0^T \sum_{i=1}^n \left[ \mathbf{H}_{\ell i}(\tilde{q}_{\lb{s}}) \partial_{v^j} \tilde{q}^i_{\lb{s}}  + \mathbf{H}_{\ell i}(\tilde{q}_{\ub{s}}) \partial_{v^j} \tilde{q}^i_{\ub{s}} \right] ds  . \nonumber 
\end{align} Using \eqref{eq:L_mf} and inserting $(L + 4 \epsilon \tilde L) T^2 \le 1/6$, we obtain \begin{align*}
&\sum_{\ell=1}^n \max_{s \le T}  \mnorm{\partial_{v^j} \tilde{q}_T^{\ell}(x,u)} \ \le \ T \mnorm{I_k} + T^2   \sum_{\ell=1}^n \sum_{i=1}^n \max_{s \le T} \mnorm{\mathbf{H}_{\ell i}(\tilde{q}_s) \partial_{v^j} \tilde{q}^i_s(x,u)}    \\
\ &\qquad \le \ T + (L + 4 \epsilon \tilde{L}) T^2 \sum_{i=1}^n \max_{s \le T}  \mnorm{\partial_{v^j} \tilde{q}^i_s(x,u)}  \ \le \ (6/5) T \;, \\
&\sum_{\ell=1}^n \max_{s \le T}  \mnorm{\partial_{v^j} \tilde{v}_s^{\ell}(x,u)} \ \le \ 
1 + (L+4 \epsilon \tilde L) T \sum_{\ell=1}^n \max_{s \le T}  \mnorm{\partial_{v^j} \tilde{q}_T^{\ell}(x,u)} \le 6/5 \;, 
\end{align*} 
which gives \eqref{apriori:5_mf} and 
\eqref{apriori:6_mf}.  For \eqref{apriori:5b_mf}, by \eqref{eq:L_mf} and $(L + 4 \epsilon \tilde L) T^2 \le 1/6$ imply that,
\begin{align*}
&\sum_{\ell=1}^n \max_{s \le T}  \mnorm{\partial_{v^{\ell}} \tilde{q}_T^{j}(x,u)} \ \le \ T + T^2   \sum_{\ell=1}^n \sum_{i=1}^n \max_{s \le T} \mnorm{\mathbf{H}_{j i}(\tilde{q}_s) \partial_{v^{\ell}} \tilde{q}^i_s(x,u)}    \\
\  &\qquad\le \ T + (L + 2 \epsilon \tilde{L}) T^2 \sum_{\ell=1}^n \max_{s \le T}  \mnorm{\partial_{v^{\ell}} \tilde{q}^{j}_s(x,u)} +   T^2 \frac{2 \epsilon \tilde{L}}{n} \sum_{\ell=1}^n  \sum_{i=1}^n \max_{s \le T}  \mnorm{\partial_{v^{\ell}} \tilde{q}^{i}_s(x,u)}  \\
\ &\qquad\le \ (6/5) T +2 (6/5)^2 T^3  \epsilon \tilde{L} \ \le \ (7/5) T  \;, \end{align*}
where in the second to last step we used \eqref{apriori:5_mf}, and in the last step, we used $\epsilon \tilde L T^2 \le 1/24$ and $33/25 \le 7/5$.  For \eqref{apriori:7_mf}, it is notationally convenient to introduce $\tilde{q}_s^{(1),\ell} : = \tilde{q}_s^{\ell}(x,u)$ and $\tilde{q}_s^{(2),\ell} : = \tilde{q}_s^{\ell}(y,v)$ for any $s \in [0,T]$ and $\ell \in \{1, \dots, n\}$.  Then from \eqref{dervelVerlet_mf}  \begin{align*}
& \partial_{v^j}\tilde{q}_T^{(1),\ell} - \partial_{v^j} \tilde{q}_T^{(2),\ell}   =  -  \int\displaylimits_0^T \frac{T - s}{2} \sum_{i=1}^n \left[ \mathbf{H}_{\ell i}(\tilde{q}^{(1)}_{\lb{s}}) \partial_{v^j}  \tilde{q}^{(1),i}_{\lb{s}}  - \mathbf{H}_{\ell i}(\tilde{q}^{(2)}_{\lb{s}}) \partial_{v^j}  \tilde{q}^{(2),i}_{\lb{s}}\right] ds  \\
&\qquad\qquad\qquad -  \int\displaylimits_0^T \frac{T - s}{2} \sum_{i=1}^n \left[  \mathbf{H}_{\ell i}(\tilde{q}_{\ub{s}}^{(1)}) \partial_{v^j} \tilde{q}_{\ub{s}}^{(1),i} -  \mathbf{H}_{\ell i}(\tilde{q}_{\ub{s}}^{(2)}) \partial_{v^j} \tilde{q}_{\ub{s}}^{(2),i} \right] ds  \\
    &\qquad\qquad\qquad +  \int\displaylimits_0^T \frac{s - \lb{s}}{2} \sum_{i=1}^n \left[ \mathbf{H}_{\ell i}(\tilde{q}_{\ub{s}}^{(1)}) \partial_{v^j} \tilde{q}_{\ub{s}}^{(1),i} - \mathbf{H}_{\ell i}(\tilde{q}_{\ub{s}}^{(2)}) \partial_{v^j} \tilde{q}_{\ub{s}}^{(2),i}  \right] ds \\
        &\qquad\qquad\qquad +  \int\displaylimits_0^T \frac{s - \lb{s}}{2} \sum_{i=1}^n \left[ \mathbf{H}_{\ell i}(\tilde{q}_{\lb{s}}^{(2)}) \partial_{v^j}  \tilde{q}_{\lb{s}}^{(2),i}  - \mathbf{H}_{\ell i}(\tilde{q}_{\lb{s}}^{(1)}) \partial_{v^j}  \tilde{q}_{\lb{s}}^{(1),i} \right] ds \\
& =  \int\displaylimits_0^T \frac{T - s}{2} \sum_{i=1}^n  \left[ \mathbf{H}_{\ell i}(\tilde{q}_{\lb{s}}^{(2)}) ( \partial_{v^j} \tilde{q}_{\lb{s}}^{(2),i} - \partial_{v^j} \tilde{q}_{\lb{s}}^{(1),i} )   - 
( \mathbf{H}_{\ell i}(\tilde{q}_{\lb{s}}^{(1)}) -  \mathbf{H}_{\ell i}(\tilde{q}_{\lb{s}}^{(2)})) \partial_{v^j} \tilde{q}_{\lb{s}}^{(1)} \right] ds  \\
& +  \int\displaylimits_0^T \frac{T - s}{2} \sum_{i=1}^n  \left[ \mathbf{H}_{\ell i}(\tilde{q}_{\ub{s}}^{(2)}) ( \partial_{v^j} \tilde{q}_{\ub{s}}^{(2),i} - \partial_{v^j} \tilde{q}_{\ub{s}}^{(1),i}) - 
( \mathbf{H}_{\ell i}(\tilde{q}_{\ub{s}}^{(1)}) - \mathbf{H}_{\ell i}(\tilde{q}_{\ub{s}}^{(2)}) )  \partial_{v^j} \tilde{q}_{\ub{s}}^{(1),i}
 \right] ds  \\
    & -  \int\displaylimits_0^T \frac{s - \lb{s}}{2} \sum_{i=1}^n  \left[  \mathbf{H}_{\ell i}(\tilde{q}_{\ub{s}}^{(2)}) ( \partial_{v^j} \tilde{q}_{\ub{s}}^{(2),i}  -\partial_{v^j} \tilde{q}_{\ub{s}}^{(1),i})  -  ( \mathbf{H}_{\ell i}(\tilde{q}_{\ub{s}}^{(1)}) - \mathbf{H}_{\ell i}(\tilde{q}_{\ub{s}}^{(2)}) )  \partial_{v^j} \tilde{q}_{\ub{s}}^{(1),i} \right] ds \\
        & - \int\displaylimits_0^T \frac{s - \lb{s}}{2} \sum_{i=1}^n  \left[ \mathbf{H}_{\ell i}(\tilde{q}_{\lb{s}}^{(1)}) ( \partial_{v^j} \tilde{q}_{\lb{s}}^{(1),i} - \partial_{v^j} \tilde{q}_{\lb{s}}^{(2),i}) - ( \mathbf{H}_{\ell i}(\tilde{q}_{\lb{s}}^{(2)}) - \mathbf{H}_{\ell i}(\tilde{q}_{\lb{s}}^{(1)}) ) \partial_{v^j} \tilde{q}_{\lb{s}}^{(2),i}    \right] ds.
\end{align*}
By using \eqref{eq:L_mf}, \eqref{eq:dH_ij_mf}, and \eqref{apriori:5b_mf},   we then get 
\begin{align*}
& \sum_{j=1}^n \sum_{\ell=1}^n \max_{s \le T} \mnorm{\partial_{v^j}\tilde{q}_T^{(1),\ell} - \partial_{v^j} \tilde{q}_T^{(2),\ell} } \le  (L + 4 \epsilon \tilde{L}) T^2 
\sum_{j=1}^n \sum_{\ell=1}^n  
\max_{s \le T} \mnorm{ \partial_{v^j}\tilde{q}_T^{(1),\ell} - \partial_{v^j} \tilde{q}_T^{(2),\ell} } \\
& \qquad + (L_H + 8 \epsilon \tilde{L}_H)  (7/5) T^3 \sum_{\ell=1}^n \max_{s \le T} \norm{\tilde{q}_s^{(1),\ell} - \tilde{q}_s^{(2),\ell}} \;. 
\end{align*}
Inserting $(L + 4 \epsilon \tilde{L}) T^2 \le 1/6$ and \eqref{apriori:3_mf}, and simplifying yields \eqref{apriori:7_mf}.  
\end{proof}

\subsection{Discretization error bounds}

The following lemma is the mean-field analog of Lemma~\ref{lem:verlet_error:1}.
\begin{lem} \label{lem:verlet_error:1_mf}
For any $x,v\in\R^{nk}$,
\begin{equation} \label{verlet_error:1_mf}
	\begin{aligned}
&\sum_{\ell=1}^n \max_{s\leq T}\norm{ q_s^{\ell}(x,v) -  \tilde{q}^{\ell}_s(x,v)}    \ \leq \  h^2 \Big(    (L+4 \epsilon \tilde{L}) \sum_{\ell=1}^n \norm{x^{\ell}}  
 + (L+4 \epsilon \tilde{L}) T \sum_{\ell=1}^n 
  \norm{v^{\ell}}
\\
&\qquad\qquad 
 +  (L_H +8 \epsilon \tilde{L}_H)  \sum_{\ell=1}^n \norm{x^{\ell}}^2 
 +  (L_H +8 \epsilon \tilde{L}_H) T \sum_{\ell=1}^n \norm{v^{\ell}}^2  \Big) \;.
\end{aligned}
\end{equation}
\end{lem}

\begin{proof}
By \eqref{apriori:1_mf} and \eqref{apriori:2_mf} with $h=0$, and since $(L+4 \epsilon \tilde{L}) T^2 \le 1/6$, note that \begin{align}
    \label{max:q_mf} \sum_{\ell=1}^n \max_{s \le T} \norm{q_s^{\ell}} \ &\le \ (7/6) \, \sum_{\ell=1}^n  \norm{x^{\ell}} + (7/6) T \sum_{\ell=1}^n \norm{v^{\ell}} \;,  \\
    \label{max:v_mf} \sum_{\ell=1}^n \max_{s \le T} \norm{v_s^{\ell}} \ &\le \ (7/6) \, (L + 4 \epsilon \tilde{L}) T \sum_{\ell=1}^n \norm{x^{\ell}}  +(6/5) \sum_{\ell=1}^n \norm{v^{\ell}}    \;,  \\
            \label{max:q2_mf} \sum_{\ell=1}^n \max_{s \le T} \norm{q_s^{\ell}}^2 \ &\le \   (9/2) \sum_{\ell=1}^n \norm{x^{\ell}}^2 +  (9/2) T^2 \sum_{\ell=1}^n \norm{v^{\ell}}^2 \;, \\
        \label{max:v2_mf} \sum_{\ell=1}^n \max_{s \le T} \norm{v_s^{\ell}}^2 \ &\le \  (9/2) (L+4 \epsilon \tilde{L}) \sum_{\ell=1}^n \norm{x^{\ell}}^2 +  (9/2) \sum_{\ell=1}^n \norm{v^{\ell}}^2 \;.
\end{align} Introduce the shorthand $q_T:= q_T(x,v)$ and $\tilde q_T := \tilde q_T(x,v)$.  Using \eqref{velVerlet}, note that
 \begin{align}
 q_T^{\ell} &  - \tilde q_T^{\ell} \ = \ \rn{1}^{\ell} + \rn{2}^{\ell} + \rn{3}^{\ell} \quad \text{where} \label{diff:q_tq_mf} \\
\rn{1}^{\ell} \ &:= \ \frac{1}{2} \int_0^T (T-s) [ \nabla_{\ell} U(\tilde q_{\lb{s}} ) - \nabla_{\ell} U(q_{\lb{s}} ) + \nabla_{\ell} U(\tilde q_{\ub{s}}) - \nabla_{\ell} U(q_{\ub{s}}) ] ds  \nonumber \\
& - \frac{1}{2} \int_0^T ( s - \lb{s}) [ \nabla_{\ell} U(\tilde q_{\ub{s}}) - \nabla_{\ell} U(q_{\ub{s}}) - (\nabla_{\ell} U(\tilde q_{\lb{s}}) - \nabla_{\ell} U(q_{\lb{s}}) ) ] ds \;, \nonumber \\
\rn{2}^{\ell} \ &:= \ -\frac{1}{2} \int_0^T ( s - \lb{s}) [ \nabla_{\ell} U(q_{\ub{s}}) - \nabla_{\ell} U(q_{\lb{s}}) ] ds \;,  \quad \text{and} \nonumber \\
\rn{3}^{\ell} \ &:= \ - \int_0^T (T-s) \nabla_{\ell} U(q_s) ds + \frac{1}{2} \int_0^T (T-s) [ \nabla_{\ell} U(q_{\lb{s}} ) + \nabla_{\ell} U(q_{\ub{s}}) ] ds   \;. \nonumber
\end{align} 
By \eqref{eq:dG_i_mf} and since $(L + 4 \epsilon \tilde{L}) T^2 \le 1/6$,  \begin{align} \label{qerr:1_mf}
&   \sum_{\ell=1}^n \norm{\rn{1}^{\ell}}  \le  (L + 4 \epsilon \tilde{L})  T^2 \, \sum_{\ell=1}^n \max_{s \le T} \norm{q_s^{\ell} - \tilde q_s^{\ell}} \ \le \ (1/6) \, \sum_{\ell=1}^n  \max_{s \le T} \norm{q_s^{\ell} - \tilde q_s^{\ell}}, \\
& \sum_{\ell=1}^n  \norm{\rn{2}^{\ell}}  \le  \frac{h (L + 4 \epsilon \tilde{L})}{2}  \sum_{\ell=1}^n \int_0^T \norm{q_{\ub{s}}^{\ell}-q_{\lb{s}}^{\ell}} ds   =  \frac{h (L + 4 \epsilon \tilde{L})}{2}  \sum_{\ell=1}^n \int_0^T \norm{\int_{\lb{s}}^{\ub{s}} v_r^{\ell} dr} ds, \nonumber \\
  &\quad \le   \frac{h^2  (L + 4 \epsilon \tilde{L}) T}{2}  \sum_{\ell=1}^n \max_{s \le T} \norm{v_s^{\ell}} \ \le \  (L+4 \epsilon \tilde{L}) h^2 \, \Big( \frac{7}{72} \sum_{\ell=1}^n \norm{x^{\ell}} + \frac{3}{5} \, T \,\sum_{\ell=1}^n \norm{v^{\ell}} \Big) \label{qerr:2_mf} \;,
\end{align}
where for \eqref{qerr:2_mf} we  used \eqref{apriori:2_mf}. Applying Lemma~\ref{trapz} with 
$f(s) = - \nabla_{\ell} U(q_s)$ we obtain \begin{equation}   \label{qerr:3_mf}
\begin{aligned}
& \sum_{\ell=1}^n \norm{\rn{3}^{\ell}} \le \frac{h^2}{12} \Big( 
( L+ 4 \epsilon \tilde{L}) T \sum_{\ell=1}^n \max_{s \le T} \norm{v_s^{\ell}} +
( L+ 4 \epsilon \tilde{L})^2 T^2 \sum_{\ell=1}^n \max_{s \le T} \norm{q_s^{\ell}} 
\\
& \qquad + (L_H + 8 \epsilon \tilde{L}_H) T^2 \sum_{\ell=1}^n \max_{s \le T} \norm{ v_s^{\ell}}^2  \Big)   \;.
\end{aligned}
\end{equation}
Inserting \eqref{qerr:1_mf}, \eqref{qerr:2_mf}  and \eqref{qerr:3_mf} into the norm of \eqref{diff:q_tq_mf} summed over $\ell$; then inserting \eqref{max:q_mf}, \eqref{max:v_mf}, and \eqref{max:v2_mf}; and then simplifying gives \eqref{verlet_error:1_mf}.
\end{proof}

\begin{lem} \label{lem:verlet_error:2_mf}
For any $x,v\in\R^{nk}$,
\begin{equation} \label{verlet_error:2_mf}
	\begin{aligned}
&\sum_{i=1}^n \sum_{\ell=1}^n \max_{s\leq T} \mnorm{\partial_{v^i} q^{\ell}_s(x,v)-\partial_{v^i}  \tilde{q}^{\ell}_s(x,v)}  \\
&\quad\leq \  h^2 \Big( 
(L+4 \epsilon \tilde{L}) T n  + 
(L_H + 8 \epsilon \tilde{L}_H) T \sum_{\ell=1}^n \norm{x^{\ell}}   + 2 (L_H +8 \epsilon \tilde{L}_H) T^2  \sum_{\ell=1}^n \norm{v^{\ell}}
\\
&\quad \qquad
 + \big( 2 (L_H +8 \epsilon \tilde{L}_H)^2 T^3  + (L_I + 14 \epsilon \tilde{L}_I) T  \big)  \sum_{\ell=1}^n \norm{x^{\ell}}^2 \\
&\quad\qquad
 + \big( 2 (L_H +8 \epsilon \tilde{L}_H)^2 T^5  + (L_I + 14 \epsilon \tilde{L}_I) T^3  \big) \sum_{\ell=1}^n \norm{v^{\ell}}^2  \Big) \;.
\end{aligned}
\end{equation}
\end{lem}

\begin{proof}

Using \eqref{velVerlet}, write the difference as
 \begin{align}
\partial_{v^i} q_T^{\ell} & (x,v)  - \partial_{v^i} \tilde q_T^{\ell} (x,v) \ = \ \rn{1}^{\ell}_i + \rn{2}^{\ell}_i + \rn{3}^{\ell}_i + \rn{4}^{\ell}_i \quad \text{where} \label{diff:D2q_D2tq_mf} \\
 \rn{1}^{\ell}_i \ &:= \ \frac{1}{2} \int_0^T (T-s)  \sum_{m=1}^n \mathbf{H}_{\ell m}(\tilde q_{\lb{s}} ) (\partial_{v^i} \tilde q_{\lb{s}}^m - \partial_{v^i} q_{\lb{s}}^m)   ds \nonumber  \\
& \quad + \frac{1}{2} \int_0^T (T-s) \sum_{m=1}^n \mathbf{H}_{\ell m}(\tilde q_{\ub{s}}) ( \partial_{v^i} \tilde q_{\ub{s}}^m - \partial_{v^i}  q_{\ub{s}}^m)  ds  \nonumber \\
& \quad - \frac{1}{2} \int_0^T ( s - \lb{s}) \sum_{m=1}^n \mathbf{H}_{\ell m}(\tilde q_{\ub{s}}) ( \partial_{v^i} \tilde q_{\ub{s}}^m - \partial_{v^i}  q_{\ub{s}}^m ) ds \nonumber \\
& \quad + \frac{1}{2} \int_0^T ( s - \lb{s}) \sum_{m=1}^n \mathbf{H}_{\ell m}(\tilde q_{\lb{s}}) (\partial_{v^i} \tilde q_{\lb{s}}^m - \partial_{v^i}  q_{\lb{s}}^m)   ds \nonumber \\
\rn{2}^{\ell}_i \ &:= \ \frac{1}{2} \int_0^T (T-s) \sum_{m=1}^n ( \mathbf{H}_{\ell m}(\tilde q_{\lb{s}} ) - \mathbf{H}_{\ell m}(q_{\lb{s}} ) ) \partial_{v^i}  q_{\lb{s}}^m  ds \nonumber  \\
& \quad + \frac{1}{2} \int_0^T (T-s) \sum_{m=1}^n ( \mathbf{H}_{\ell m}(\tilde q_{\ub{s}}) - \mathbf{H}_{\ell m}(q_{\ub{s}}) ) \partial_{v^i}  q_{\ub{s}}^m  ds  \nonumber \\
& \quad - \frac{1}{2} \int_0^T ( s - \lb{s}) \sum_{m=1}^n (\mathbf{H}_{\ell m}(\tilde q_{\ub{s}}) - \mathbf{H}_{\ell m}(q_{\ub{s}}) ) \partial_{v^i}  q_{\ub{s}}^m ds \nonumber \\
& \quad + \frac{1}{2} \int_0^T ( s - \lb{s}) \sum_{m=1}^n (\mathbf{H}_{\ell m}(\tilde q_{\lb{s}}) - \mathbf{H}_{\ell m}(q_{\lb{s}}) )  \partial_{v^i} q_{\lb{s}}^m ds \nonumber \\
\rn{3}^{\ell}_i \ &:= \ 
  -\frac{1}{2} \int_0^T ( s - \lb{s}) \sum_{m=1}^n [ (\mathbf{H}_{\ell m}(q_{\ub{s}}) - \mathbf{H}_{\ell m}(q_{\lb{s}})) \partial_{v^i} q_{\ub{s}}^m] \nonumber \\
 & \quad - \frac{1}{2} \int_0^T ( s - \lb{s}) \sum_{m=1}^n [ \mathbf{H}_{\ell m}(q_{\lb{s}}) (\partial_{v^i} q_{\ub{s}}^m - \partial_{v^i} q_{\lb{s}}^m ) ] ds 
  \;,  \quad \text{and} \nonumber \\
\rn{4}^{\ell}_i \ &:= \ - \int_0^T (T-s) \sum_{m=1}^n \mathbf{H}_{\ell m}(q_s) \partial_{v^i} q_s^m ds \nonumber \\ 
& \quad + \frac{1}{2} \int_0^T (T-s) \sum_{m=1}^n [ \mathbf{H}_{\ell m}(q_{\lb{s}} ) \partial_{v^i} q_{\lb{s}}^m + \mathbf{H}_{\ell m}(q_{\ub{s}}) \partial_{v^i} q_{\ub{s}}^m ] ds   \;. \nonumber
\end{align} 
To obtain the following bounds, 
we use \eqref{eq:L_mf} and  $(L + 4 \epsilon \tilde{L}) T^2 \le 1/6$ for \eqref{D2qerr:1_mf};
  \eqref{eq:dH_ij_mf} and \eqref{apriori:5b_mf} for \eqref{D2qerr:2_mf}; and \eqref{eq:L_mf}, \eqref{eq:dH_ij_mf}, \eqref{apriori:5b_mf}, and for \eqref{D2qerr:3_mf}. 
\begin{align} \label{D2qerr:1_mf}
& \sum_{i=1}^n \sum_{\ell=1}^n   \mnorm{\rn{1}^{\ell}_i} \ \le \ (1/6)  \, \max_{s \le T} \sum_{i=1}^n \sum_{\ell=1}^n \mnorm{\partial_{v^i} q_s^{\ell} - \partial_{v^i} \tilde q_s^{\ell}} \;, \\
& \sum_{i=1}^n \sum_{\ell=1}^n   \mnorm{\rn{2}^{\ell}_i} \ \le \  (7/5) (L_H + 8 \epsilon \tilde{L}_H) T^3 \sum_{\ell=1}^m 
  \max_{s \le T} \norm{q_s^{\ell} - \tilde q_s^{\ell}}  \label{D2qerr:2_mf} \;, \\
 & \sum_{i=1}^n \sum_{\ell=1}^n   \mnorm{\rn{3}^{\ell}_i} 
  \ \le \  (7/5) \frac{h (L_H+8 \epsilon \tilde{L}_H) T}{2}    \sum_{\ell=1}^n \int_0^T  \norm{ \int_{\lb{s}}^{\ub{s}} v_r^{\ell} dr} ds \nonumber \\
  & \qquad \qquad \qquad  + \frac{h (L+4 \epsilon \tilde{L}) }{2}  \sum_{i=1}^n \sum_{\ell=1}^n \int_0^T \mnorm{ \int_{\lb{s}}^{\ub{s}} \partial_{v^i} v_r^{\ell} dr } ds \nonumber \\
& \quad \ \le \ 
  \frac{7 h^2  (L_H+8 \epsilon \tilde{L}_H) T^2}{10}   \sum_{\ell=1}^n \max_{s \le T} \norm{v_s^{\ell}}
 + \frac{h^2 (L+4 \epsilon \tilde{L}) T}{2}  \sum_{i=1}^n \sum_{\ell=1}^n \max_{s \le T} \mnorm{  \partial_{v^i} v_s^{\ell}  }
 \label{D2qerr:3_mf} .
\end{align}
 Applying Lemma~\ref{trapz} with 
$f(s) = \mathbf{H}_{\ell m}(q_s) \partial_{v^i} q_s^m$ and using $(L+4 \epsilon \tilde{L}) T^2 \le 1/6$, \eqref{eq:L_mf}, \eqref{eq:LH_mf}, \eqref{eq:LI_mf}, \eqref{apriori:5b_mf} and \eqref{apriori:6_mf}, we obtain 
  \begin{align}
&   \sum_{i=1}^n \sum_{\ell=1}^n 
  \mnorm{\rn{4}^{\ell}_i} \nonumber \\
   \begin{split}
& \quad  \ \le \ \frac{h^2}{12} \Big( \frac{43}{30} (L+4 \epsilon \tilde{L}) T n 
+  4 (L_H + 8 \epsilon \tilde{L}_H) T^2 \sum_{\ell=1}^n \max_{s \le T}\norm{v_s^{\ell}}     \label{D2qerr:4_mf} \\
&  \qquad   + \frac{7}{30} (L_H + 8 \epsilon \tilde{L}_H) T \sum_{\ell=1}^n \max_{s \le T}\norm{q_s^{\ell}}  +  \frac{7}{5} (L_I + 14 \epsilon \tilde{L}_I) T^3  \sum_{\ell=1}^n \max_{s \le T} \norm{v_s^{\ell}}^2
  \Big)  \;.
    \end{split}
\end{align}
Insert \eqref{D2qerr:1_mf}, \eqref{D2qerr:2_mf}, \eqref{D2qerr:3_mf}, 
\eqref{D2qerr:4_mf} and
Lemma~\ref{lem:verlet_error:1_mf}
into norm of the double sum of \eqref{diff:D2q_D2tq_mf} over $i$ and $\ell$; use \eqref{max:q_mf}, \eqref{max:v_mf}, \eqref{max:v2_mf}, and \eqref{apriori:6_mf}; and simplify to obtain \eqref{verlet_error:2_mf}.
\end{proof}

\subsection{One-shot coupling bounds for $\tilde q_T(x,\xi) = \tilde q_T(y, \Phi(\xi))$}

The following lemmas are the mean-field analogs of Lemmas~\ref{lem:oneshot:1:a} and~\ref{lem:oneshot:1:b}. 

\begin{lem}\label{lem:oneshot:1:a_mf}
For any $x,y,v \in \mathbb{R}^{n k}$ such that $\tilde q_T(x,v) = \tilde q_T(y,\Phi(v))$, we have \begin{equation} \label{Phiv_mf}
T \norm{\Phi(v)-v} \ \le \ T \sum_{\ell =1}^n \norm{\Phi^{\ell}(v)-v^{\ell}} \ \le  \ (3/2) \ \sum_{\ell =1}^n \norm{x^{\ell} - y^{\ell} }.
\end{equation}
\end{lem}

\begin{proof}
Let $u=\Phi(v)$.
From  $\tilde q_T(x,v) = \tilde q_T(y,u)$, \begin{align*}
& T \norm{u^{\ell} - v^{\ell}} \ \le \ \norm{x^{\ell} - y^{\ell}} + (L+2 \epsilon \tilde{L}) T^2 \max_{s \le T} \norm{\tilde{q}_s^i(x,v) - \tilde{q}_s^i(y,u)} \\
& \qquad + \frac{2 \epsilon \tilde{L} T^2}{n}  \sum^n_{i =1, i \ne \ell} \norm{\tilde{q}_s^i(x,v) - \tilde{q}_s^i(y,u)}. 
\end{align*}
Summing over $\ell$ and using \eqref{apriori:3_mf} and $(L+4 \epsilon \tilde{L}) (T^2 + T h) \le 1/6$ gives  \begin{align*}
 T \sum_{\ell=1}^n \norm{u^{\ell} - v^{\ell}} \ &\le \ \sum_{\ell=1}^n \norm{x^{\ell} - y^{\ell}} + (L+4 \epsilon \tilde{L}) T^2 \sum_{\ell=1}^n \max_{s \le T} \norm{\tilde{q}_s^{\ell}(x,v) - \tilde{q}_s^{\ell}(y,u)} \\
 \ &\le \ \sum_{\ell=1}^n \norm{x^{\ell} - y^{\ell}} + (7/36)  \sum_{\ell=1}^n \left( \norm{x^{\ell} - y^{\ell}} + T \norm{u^{\ell} - v^{\ell}} \right) \le (3/2) \sum_{\ell=1}^n \norm{x^{\ell} - y^{\ell}}
\end{align*}
where in the last step we used $(1+7/36)/(1-7/36) < 3/2$.

\end{proof}

\begin{lem}\label{lem:oneshot:1:b_mf}
For any $x,y,v \in \mathbb{R}^{n k}$ such that $\tilde q_T(x,v) = \tilde q_T(y,\Phi(v))$, we have that
$\mnorm{D \Phi(v)- I_d} \ \le \ 1/2 \ $ and \begin{equation} \label{DPhiv_mf}
\mnorm{D \Phi(v)- I_d}_F \ \le \  \frac{49}{8} \, \sqrt{k} \, (L_H + 8 \epsilon \tilde{L}_H) \,  T^2 \, \sum_{\ell=1}^n \norm{x^{\ell} - y^{\ell}} \;.
\end{equation}
\end{lem}

\begin{proof}
First, note that \[
\mnorm{D \Phi(v)- I_d}_F 
\, = \, \left( \sum_{\ell=1}^n \sum_{i=1}^{n}  \mnorm{\partial_{v^{\ell}} \Phi^i(v) - \delta_{i \ell} I_k}_F^2 \right)^{1/2} 
\, \le \,  \sqrt{k} \sum_{\ell=1}^n \sum_{i=1}^{n}  \mnorm{\partial_{v^{\ell}} \Phi^i(v) - \delta_{i \ell} I_k} .
\]  Next, and in turn, we upper bound $\mnorm{D \Phi(v)- I_d}$ and $\sum_{\ell=1}^n \sum_{i=1}^{n}  \mnorm{\partial_{v^{\ell}} \Phi^i(v) - \delta_{i \ell} I_k} $.   Introduce the shorthand $\tilde{q}_s^{(1)} : = \tilde{q}_s(x,v)$ and $\tilde{q}_s^{(2)} : = \tilde{q}_s(y,\Phi(v))$ for any $s \in [0,T]$.
Differentiating both sides of $\tilde{q}_T^{\ell}(x,v) = \tilde{q}_T^{\ell}(y,\Phi(v))$ with respect to $v^j$ yields \begin{equation} \label{eq:DPhimId_mf}
\begin{aligned}
&T (\partial_{v^{\ell}} \Phi^i(v) - \delta_{i \ell} I_k) = \\
& +  \int\displaylimits_0^T \frac{T - s}{2} \sum_{j=1}^n \left[ \mathbf{H}_{ij}(\tilde{q}_{\lb{s}}^{(2)})  \partial_{v^{\ell}} \tilde{q}_{\lb{s}}^{(2),j}   - 
 \mathbf{H}_{ij}(\tilde{q}_{\lb{s}}^{(1)}) \partial_{v^{\ell}} \tilde{q}_{\lb{s}}^{(1),j} \right] ds  \\
& + \int\displaylimits_0^T \frac{T - s}{2} \sum_{j=1}^n \left[ \mathbf{H}_{ij}(\tilde{q}_{\ub{s}}^{(2)}) \partial_{v^{\ell}} \tilde{q}_{\ub{s}}^{(2),j}  - 
 \mathbf{H}_{ij}(\tilde{q}_{\ub{s}}^{(1)})   \partial_{v^{\ell}} \tilde{q}_{\ub{s}}^{(1),j}
 \right] ds  \\
    & -  \int\displaylimits_0^T \frac{s - \lb{s}}{2} \sum_{j=1}^n \left[  \mathbf{H}_{ij}(\tilde{q}_{\ub{s}}^{(2)})  \partial_{v^{\ell}} \tilde{q}_{\ub{s}}^{(2),j}    -   \mathbf{H}_{ij}(\tilde{q}_{\ub{s}}^{(1)})   \partial_{v^{\ell}} \tilde{q}_{\ub{s}}^{(1),j} \right] ds \\
        & -  \int\displaylimits_0^T \frac{s - \lb{s}}{2} \sum_{j=1}^n \left[ \mathbf{H}_{ij}(\tilde{q}_{\lb{s}}^{(1)})  \partial_{v^{\ell}} \tilde{q}_{\lb{s}}^{(1),j}  -  \mathbf{H}_{ij}(\tilde{q}_{\lb{s}}^{(2)}) \partial_{v^{\ell}} \tilde{q}_{\lb{s}}^{(2),j}    \right] ds \\
& + \int\displaylimits_0^T \frac{T - s}{2} \sum_{j,m}  \left[\mathbf{H}_{ij}(\tilde{q}_{\lb{s}}^{(2)}) \partial_{v^{m}} \tilde{q}_{\lb{s}}^{(2),j} + \mathbf{H}_{ij}(\tilde{q}_{\ub{s}}^{(2)}) \partial_{v^{m}} \tilde{q}_{\ub{s}}^{(2),j} \right] (\partial_{v^{\ell}} \Phi^m(v) - \delta_{m \ell} I_k) ds  \\
 & -  \int\displaylimits_0^T \frac{s - \lb{s}}{2} \sum_{j,m}  \left[\mathbf{H}_{ij}(\tilde{q}_{\ub{s}}^{(2)}) \partial_{v^{m}} \tilde{q}_{\ub{s}}^{(2),j} - \mathbf{H}_{ij}(\tilde{q}_{\lb{s}}^{(2)}) \partial_{v^{m}} \tilde{q}_{\lb{s}}^{(2),j} \right] (\partial_{v^{\ell}} \Phi^m(v) - \delta_{m \ell} I_k) ds.
 \end{aligned}
 \end{equation}
 By using \eqref{eq:L_mf}, \eqref{apriori:5_mf},  and $(L + 4 \epsilon \tilde{L}) T^2 \le 1/6$, for fixed $\ell \in \{1, \dots, n\}$ note that \[ 
  \sum_{i=1}^n  \mnorm{\partial_{v^{\ell}} \Phi^i(v) - \delta_{i \ell} I_k} \le \frac{ (6/5) 2 (L + 4 \epsilon \tilde{L}) T^2}{ (1- (6/5) (L + 4 \epsilon \tilde{L})T^2)} = 1/2 \;,
\] and therefore, for any $z=(z^1, \dots, z^n) \in \mathbb{R}^{n k}$, we have \[
\norm{(D \Phi(v) - I_d) z}^2 
\le \sum_{\ell=1}^n \sum_{i=1}^n \norm{z^{\ell}}^2  \mnorm{\partial_{v^{\ell}} \Phi^i(v) - \delta_{i \ell}^2 I_k}^2 \le \frac{1}{4} \sum_{\ell=1}^n \norm{z^{\ell}}^2 .
\] Thus, $\mnorm{D \Phi(v) - I_d} \le 1/2$.
We can also rewrite \eqref{eq:DPhimId_mf} as 
\begin{equation*} 
\begin{aligned}
&T (\partial_{v^{\ell}} \Phi^i(v) - \delta_{i \ell} I_k) = \\
& +  \int\displaylimits_0^T \frac{T - s}{2} \sum_{j=1}^n \left[ \mathbf{H}_{ij}(\tilde{q}_{\lb{s}}^{(2)}) ( \partial_{v^{\ell}} \tilde{q}_{\lb{s}}^{(2),j} - \partial_{v^{\ell}} \tilde{q}_{\lb{s}}^{(1),j})   - 
( \mathbf{H}_{ij}(\tilde{q}_{\lb{s}}^{(1)})- \mathbf{H}_{ij}(\tilde{q}_{\lb{s}}^{(2)})) \partial_{v^{\ell}} \tilde{q}_{\lb{s}}^{(1),j} \right] ds  \\
& + \int\displaylimits_0^T \frac{T - s}{2} \sum_{j=1}^n \left[ \mathbf{H}_{ij}(\tilde{q}_{\ub{s}}^{(2)}) (\partial_{v^{\ell}} \tilde{q}_{\ub{s}}^{(2),j} - \partial_{v^{\ell}} \tilde{q}_{\ub{s}}^{(1),j}) - 
( \mathbf{H}_{ij}(\tilde{q}_{\ub{s}}^{(1)})    - \mathbf{H}_{ij}(\tilde{q}_{\ub{s}}^{(2)})) \partial_{v^{\ell}} \tilde{q}_{\ub{s}}^{(1),j}
 \right] ds  \\
    & -  \int\displaylimits_0^T \frac{s - \lb{s}}{2} \sum_{j=1}^n \left[  \mathbf{H}_{ij}(\tilde{q}_{\ub{s}}^{(2)})  (\partial_{v^{\ell}} \tilde{q}_{\ub{s}}^{(2),j} - \partial_{v^{\ell}} \tilde{q}_{\ub{s}}^{(1),j} )     -   (\mathbf{H}_{ij}(\tilde{q}_{\ub{s}}^{(1)})  - \mathbf{H}_{ij}(\tilde{q}_{\ub{s}}^{(2)}))  \partial_{v^{\ell}} \tilde{q}_{\ub{s}}^{(1),j} \right] ds \\
        & +  \int\displaylimits_0^T \frac{s - \lb{s}}{2} \sum_{j=1}^n \left[ \mathbf{H}_{ij}(\tilde{q}_{\lb{s}}^{(2)})  ( \partial_{v^{\ell}} \tilde{q}_{\lb{s}}^{(2),j} - \partial_{v^{\ell}} \tilde{q}_{\lb{s}}^{(1),j}   )  -  (\mathbf{H}_{ij}(\tilde{q}_{\lb{s}}^{(1)}) - \mathbf{H}_{ij}(\tilde{q}_{\lb{s}}^{(2)}))  \partial_{v^{\ell}} \tilde{q}_{\lb{s}}^{(1),j}    \right] ds \\
& + \int\displaylimits_0^T \frac{T - s}{2} \sum_{j,m}  \left[\mathbf{H}_{ij}(\tilde{q}_{\lb{s}}^{(2)}) \partial_{v^{m}} \tilde{q}_{\lb{s}}^{(2),j} + \mathbf{H}_{ij}(\tilde{q}_{\ub{s}}^{(2)}) \partial_{v^{m}} \tilde{q}_{\ub{s}}^{(2),j} \right] (\partial_{v^{\ell}} \Phi^m(v) - \delta_{m \ell} I_k) ds  \\
 & -  \int\displaylimits_0^T \frac{s - \lb{s}}{2} \sum_{j,m}  \left[\mathbf{H}_{ij}(\tilde{q}_{\ub{s}}^{(2)}) \partial_{v^{m}} \tilde{q}_{\ub{s}}^{(2),j} - \mathbf{H}_{ij}(\tilde{q}_{\lb{s}}^{(2)}) \partial_{v^{m}} \tilde{q}_{\lb{s}}^{(2),j} \right] (\partial_{v^{\ell}} \Phi^m(v) - \delta_{m \ell} I_k) ds.
 \end{aligned}
 \end{equation*}
By using \eqref{eq:L_mf}, \eqref{eq:dH_ij_mf},  \eqref{apriori:5_mf}, \eqref{apriori:5b_mf}, \eqref{apriori:7_mf}, and $(L + 4 \epsilon \tilde{L})  T^2 \le 1/6$, we get \begin{align*}
  & (4/5) T \sum_{i,\ell} \mnorm{\partial_{v^{\ell}} \Phi^i(v) - \delta_{i \ell} I_k}   
    \le  (L + 4 \epsilon \tilde{L}) T^2 \sum_{j,\ell} \max_{s \le T}  \mnorm{\partial_{v^{\ell}} \tilde{q}_{s}^{(2),j} - \partial_{v^{\ell}} \tilde{q}_{s}^{(1),j} }  \, \\
    & \qquad\qquad + (7/5) T (L_H + 8 \epsilon \tilde{L}_H) T^2 \sum_{\ell=1}^n \max_{s \le T} \norm{\tilde{q}_{s}^{(1),\ell} - \tilde{q}_{s}^{(2),\ell} } \\
    & \qquad\le (49/25)   (L_H + 8 \epsilon \tilde{L}_H) T^3 \sum_{\ell=1}^n ( \norm{ x^{\ell} - y^{\ell} } + T \norm{\Phi^{\ell}(v) - v^{\ell}}  ) \\
    &\qquad \le  (49/10) (L_H + 8 \epsilon \tilde{L}_H) T^3 \sum_{\ell=1}^n \norm{ x^{\ell} - y^{\ell} }
    \end{align*}
where in the last step we inserted \eqref{Phiv_mf}.  Simplifying gives \eqref{DPhiv_mf}.  
\end{proof}

\subsection{One-shot coupling bounds for $\tilde q_T(x,\xi) = q_T(x, \Phi(\xi))$}

The following lemmas are the mean-field analogs of Lemmas~\ref{lem:oneshot:2:a} and~\ref{lem:oneshot:2:b}. 

\begin{lem}\label{lem:oneshot:2:a_mf}
For any $x,v \in \mathbb{R}^{n k}$ such that $\tilde q_T(x,v) = q_T(x,\Phi(v))$, we have
\begin{equation}
\label{Phiv2_mf}
\begin{aligned} 
& T \norm{\Phi(v)-v}   \, \le \, T \sum_{\ell=1}^n \norm{\Phi^{\ell}(v) - v^{\ell}} \ \le \ \frac{72}{65}   h^2 \Big(    (L+4 \epsilon \tilde{L}) \sum_{\ell=1}^n \norm{x^{\ell}}  
\\
& \quad + (L+4 \epsilon \tilde{L}) T \sum_{\ell=1}^n 
  \norm{v^{\ell}}
 +  (L_H +8 \epsilon \tilde{L}_H)  \sum_{\ell=1}^n \norm{x^{\ell}}^2 
 +  (L_H +8 \epsilon \tilde{L}_H) T \sum_{\ell=1}^n \norm{v^{\ell}}^2  \Big) .
\end{aligned}
\end{equation}
\end{lem}

\begin{proof}
Introduce the shorthand $\tilde q_T = \tilde q_T(x,v)$, $q_T^{(1)} = q_T(x,v)$ and $q_T^{(2)} = q_T(x,\Phi(v))$.
Noting that $ \norm{\Phi(v) - v} \le  \sum_{\ell=1}^n \norm{\Phi^{\ell}(v) - v^{\ell}}$, and using  $\tilde q_T = q_T^{(2)}$  or $q_T^{(2)}  - q_T^{(1)} = \tilde q_T - q_T^{(1)}$ and then applying \eqref{eq:dG_i_mf}, we obtain \begin{align*}
 & T \sum_{\ell=1}^n \norm{\Phi^{\ell}(v) - v^{\ell}}  \le  \sum_{\ell=1}^n  \norm{ \int_0^T ( T - s) [ \nabla_{\ell} U(q_{s}^{(2)}) - \nabla_{\ell} U(q_{s}^{(1)}) ] ds + \tilde{q}^{\ell}_T - q_T^{(1),\ell} } \nonumber \\
 & \qquad \le 
(L+4 \epsilon \tilde{L})  \frac{T^2}{2}  \sum_{\ell=1}^n \max_{s \le T}  \norm{q_s^{(2),\ell} - q_s^{(1),\ell} }
 + \sum_{\ell=1}^n \norm{\tilde{q}^{\ell}_T - q_T^{(1),\ell} } \\
 & \qquad \le \frac{7}{72} T \sum_{\ell=1}^n \norm{\Phi^{\ell}(v) - v^{\ell}}
 + \sum_{\ell=1}^n \norm{\tilde{q}^{\ell}_T - q_T^{(1),\ell} } \le \frac{72}{65} \sum_{\ell=1}^n  \norm{q_T^{(1),\ell} - \tilde{q}^{\ell}_T }
\end{align*}
where in the next to last step we used \eqref{apriori:3_mf} and $(L + 4 \epsilon \tilde{L}) T^2 \le 1/6$.  Inserting \eqref{verlet_error:1_mf} into this last inequality gives \eqref{Phiv2_mf}.
\end{proof}

\begin{lem}\label{lem:oneshot:2:b_mf}
For any $x,v \in \mathbb{R}^{n k}$ such that $\tilde q_T(x,v) = q_T(x,\Phi(v))$, we have that $\mnorm{D \Phi(v)- I_d} \le 1/2$ and  \begin{equation} \label{DPhiv2_mf}
\begin{aligned} 
& T \mnorm{D \Phi(v)- I_d}_F \, \le \,  \sqrt{k} \sum_{\ell=1}^n \sum_{i=1}^{n}  \mnorm{\partial_{v^{\ell}} \Phi^i(v) - \delta_{i \ell} I_k}  \\
& \le \   \sqrt{k} h^2 \bigg( (L+4 \epsilon \tilde{L}) T n + (L_H + 8 \epsilon \tilde{L}_H) T \sum_{\ell=1}^n \norm{x^{\ell}}   + 3 (L_H + 8 \epsilon \tilde{L}_H) T^2 \sum_{\ell=1}^n 
  \norm{v^{\ell}} \\
& \qquad + \big( (L_I + 14 \epsilon \tilde{L}_I) T + 2 (L_H + 8 \epsilon \tilde{L}_H)^2 T^3 \big) \sum_{\ell=1}^n \norm{x^{\ell}}^2  \\
& \qquad +
\big( (L_I + 14 \epsilon \tilde{L}_I) T^3 + 3 (L_H + 8 \epsilon \tilde{L}_H)^2 T^5 \big)  \sum_{\ell=1}^n \norm{v^{\ell}}^2 \bigg) \;.
\end{aligned} 
\end{equation}
\end{lem}

\begin{proof}
The proof that $\mnorm{D \Phi(v)- I_d} \le 1/2$ is similar to the proof in Lemma~\ref{lem:oneshot:1:b_mf} and therefore omitted.  Note that \[
\mnorm{D \Phi(v)- I_d}_F 
\, = \, \left( \sum_{\ell=1}^n \sum_{i=1}^{n}  \mnorm{\partial_{v^{\ell}} \Phi^i(v) - \delta_{i \ell} I_k}_F^2 \right)^{1/2} 
\, \le \,  \sqrt{k} \sum_{\ell=1}^n \sum_{i=1}^{n}  \mnorm{\partial_{v^{\ell}} \Phi^i(v) - \delta_{i \ell} I_k} .
\]  Next, we upper bound $\sum_{\ell=1}^n \sum_{i=1}^{n}  \mnorm{\partial_{v^{\ell}} \Phi^i(v) - \delta_{i \ell} I_k} $.    Introduce the shorthand $\tilde q_T = \tilde q_T(x,v)$, $q_T^{(1)} = q_T(x,v)$ and $q_T^{(2)} = q_T(x,\Phi(v))$. The derivative of $q_T^{(2),i}  - q_T^{(1),i} = \tilde q_T^i - q_T^{(1),i}$ with respect to $v^{\ell}$ yields \begin{align}
 & T \, (\partial_{v^{\ell}} \Phi^i(v) - \delta_{i \ell} I_k)  \ = \  
\partial_{v^{\ell}} \tilde q_T^i - \partial_{v^{\ell}} q_T^{(1),i} \nonumber \\
& \quad + \int_0^T ( T - s) \sum_{j=1}^n [ \mathbf{H}_{ij}(q_{s}^{(2)}) - \mathbf{H}_{ij} (q_{s}^{(1)}) ] \partial_{v^{\ell}} q_{s}^{(1),j} ds \nonumber \\ 
 & \quad + \int_0^T ( T - s) \sum_{j=1}^n \mathbf{H}_{ij}(q_{s}^{(2)}) [ \partial_{v^{\ell}} q_s^{(2),j} - \partial_{v^{\ell}} q_s^{(1),j} ] ds \nonumber \\
& \quad + \int_0^T ( T - s) \sum_{j=1}^n \sum_{m=1}^n \mathbf{H}_{ij}(q_{s}^{(2)}) \partial_{v^{m}} q_s^{(2),j} ( \partial_{v^{\ell}} \Phi^m(v) - \delta_{\ell m} I_k)]  ds . \nonumber 
\end{align}
By \eqref{eq:L_mf}, \eqref{eq:dH_ij_mf}, \eqref{apriori:5_mf}, \eqref{apriori:5b_mf} and $(L+4 \epsilon \tilde{L}) T^2 \le 1/6$,  \begin{align*}
 & T \, \sum_{i,\ell} \mnorm{\partial_{v^{\ell}} \Phi^i(v) - \delta_{i \ell} I_k} \\
 & \le
  \sum_{i,\ell}  \mnorm{\partial_{v^{\ell}} \tilde q_T^i - \partial_{v^{\ell}} q_T^{(1),i}  }    + \frac{7}{5} T (L_H + 8 \epsilon \tilde{L}_H) \frac{T^2}{2} \sum_{\ell} \max_{s \le T}  \norm{q_s^{(2),\ell} - q_s^{(1),\ell} }  \\
  & \quad +  (L+ 4 \epsilon \tilde{L}) \frac{T^2}{2} \sum_{i,\ell} \norm{\partial_{v^{\ell}} q_s^{(2),i} - \partial_{v^{\ell}} q_s^{(1),i} }  + \frac{6}{5}   T (L+ 4 \epsilon \tilde{L}) \frac{T^2}{2} \sum_{i,\ell} \mnorm{\partial_{v^{\ell}} \Phi^i(v) - \delta_{i \ell} I_k} \\
&\le \frac{10}{9} \sum_{i,\ell}  \mnorm{\partial_{v^{\ell}} \tilde q_T^i - \partial_{v^{\ell}} q_T^{(1),i}  }   + \frac{10}{9} \frac{7}{5} T (L_H + 8 \epsilon \tilde{L}_H) \frac{T^2}{2} \sum_{\ell} \max_{s \le T}  \norm{q_s^{(2),\ell} - q_s^{(1),\ell} }  \\
&\quad + \frac{10}{9} \frac{1}{12} \sum_{i,\ell} \mnorm{\partial_{v^{\ell}} q_s^{(2),i} - \partial_{v^{\ell}} q_s^{(1),i} } \\
&\le \frac{10}{9} \sum_{i,\ell}  \mnorm{\partial_{v^{\ell}} \tilde q_T^i - \partial_{v^{\ell}} q_T^{(1),i}  }   + \frac{49}{45} (L_H + 8 \epsilon \tilde{L}_H) T^3 \left( T \sum_{\ell=1}^n \norm{\Phi^{\ell}(v) - v^{\ell} } \right) 
 \end{align*}
 where in the last step we used
 \eqref{apriori:3_mf} and \eqref{apriori:7_mf}.  Inserting  \eqref{verlet_error:2_mf} and \eqref{Phiv2_mf} and simplifying gives \eqref{DPhiv2_mf}.
\end{proof}

\section*{Acknowledgements}
We wish to acknowledge Katharina Schuh for useful discussions.

N.~Bou-Rabee has been supported by the Alexander von Humboldt foundation and the National Science Foundation under Grant No.~DMS-1816378. 
 
A.~Eberle has been supported by the Hausdorff Center for Mathematics. Gef\"ordert durch die Deutsche Forschungsgemeinschaft (DFG) im Rahmen der Exzellenzstrategie des Bundes und der L\"ander - GZ 2047/1, Projekt-ID 390685813.

\bibliographystyle{amsplain}
\bibliography{nawaf}

\providecommand{\bysame}{\leavevmode\hbox to3em{\hrulefill}\thinspace}
\providecommand{\MR}{\relax\ifhmode\unskip\space\fi MR }
\providecommand{\MRhref}[2]{%
  \href{http://www.ams.org/mathscinet-getitem?mr=#1}{#2}
}
\providecommand{\href}[2]{#2}
\begin{thebibliography}{10}

\bibitem{ArmstrongMourrat}
Scott Armstrong and Jean-Christophe Mourrat, \emph{Variational methods for the
  kinetic {F}okker-{P}lanck equation}, 2019.

\bibitem{bakry2013analysis}
D.~Bakry, I.~Gentil, and M.~Ledoux, \emph{Analysis and geometry of {M}arkov
  diffusion operators}, Grundlehren der mathematischen Wissenschaften,
  Springer, 2013.

\bibitem{BoSaActaN2018}
N.~B{ou-Rabee} and J.~M. Sanz-Serna, \emph{Geometric integrators and the
  {H}amiltonian {M}onte {C}arlo method}, Acta Numerica \textbf{27} (2018),
  113--206.

\bibitem{BoEb2020}
Nawaf Bou-Rabee and Andreas Eberle, \emph{Two-scale coupling for preconditioned
  {H}amiltonian {M}onte {C}arlo in infinite dimensions}, Stochastics and
  Partial Differential Equations: Analysis and Computations \textbf{9} (2021),
  no.~1, 207–242.

\bibitem{BoEbZi2020}
Nawaf Bou-Rabee, Andreas Eberle, and Raphael Zimmer, \emph{Coupling and
  convergence for {H}amiltonian {M}onte {C}arlo}, Ann. Appl. Probab.
  \textbf{30} (2020), no.~3, 1209--1250.

\bibitem{BoSc2020}
Nawaf Bou-Rabee and Katharina Schuh, \emph{Convergence of unadjusted
  {H}amiltonian {M}onte {C}arlo for mean-field models}, arXiv preprint
  arXiv:2009.08735, 2020.

\bibitem{chen2020fast}
Yuansi Chen, Raaz Dwivedi, Martin~J Wainwright, and Bin Yu, \emph{{Fast mixing
  of Metropolized Hamiltonian Monte Carlo: Benefits of multi-step gradients}},
  Journal of Machine Learning Research \textbf{21} (2020), no.~92, 1--72.

\bibitem{chen2019optimal}
Z.~Chen and S.~S. Vempala, \emph{{Optimal convergence rate of Hamiltonian Monte
  Carlo for strongly logconcave distributions}}, arXiv preprint
  arXiv:1905.02313 (2019).

\bibitem{chewi2020optimal}
Sinho Chewi, Chen Lu, Kwangjun Ahn, Xiang Cheng, Thibaut~Le Gouic, and Philippe
  Rigollet, \emph{{Optimal dimension dependence of the Metropolis-Adjusted
  Langevin Algorithm}}, arXiv preprint arXiv:2012.12810 (2020).

\bibitem{DolbeaultMouhotSchmeiser}
Jean Dolbeault, Cl\'{e}ment Mouhot, and Christian Schmeiser,
  \emph{Hypocoercivity for linear kinetic equations conserving mass}, Trans.
  Amer. Math. Soc. \textbf{367} (2015), no.~6, 3807--3828.

\bibitem{DuKePeRo1987}
S.~Duane, A.~D. Kennedy, B.~J. Pendleton, and D.~Roweth, \emph{Hybrid
  {M}onte-{C}arlo}, Phys Lett B \textbf{195} (1987), 216--222.

\bibitem{DurmusEberle2021}
A.~Durmus and A.~Eberle, \emph{{Asymptotic bias for Markov Chain Monte Carlo
  Methods in High Dimension}}, arXiv:, 2021.

\bibitem{durmus2019}
Alain Durmus and Éric Moulines, \emph{High-dimensional {B}ayesian inference
  via the unadjusted {L}angevin algorithm}, Bernoulli \textbf{25} (2019),
  no.~4A, 2854--2882.

\bibitem{Eb2016A}
A.~Eberle, \emph{Reflection couplings and contraction rates for diffusions},
  Probability theory and related fields \textbf{166} (2016), no.~3-4, 851--886.

\bibitem{eberle2019couplings}
A.~Eberle, A.~Guillin, and R.~Zimmer, \emph{Couplings and quantitative
  contraction rates for {L}angevin dynamics}, Ann. Probab. \textbf{47} (2019),
  no.~4, 1982--2010.

\bibitem{EberleMajka2019}
Andreas Eberle and Mateusz~B. Majka, \emph{Quantitative contraction rates for
  {M}arkov chains on general state spaces}, Electronic Journal of Probability
  \textbf{24} (2019), 1--36.

\bibitem{friedland2018nuclear}
Shmuel Friedland and Lek-Heng Lim, \emph{Nuclear norm of higher-order tensors},
  Mathematics of Computation \textbf{87} (2018), no.~311, 1255--1281.

\bibitem{HaLuWa2010}
E.~Hairer, C.~Lubich, and G.~Wanner, \emph{Geometric numerical integration},
  Springer, 2010.

\bibitem{hairer2006ergodicity}
M.~Hairer and J.~C. Mattingly, \emph{Ergodicity of the {2D} {N}avier-{S}tokes
  equations with degenerate stochastic forcing}, Annals of Mathematics (2006),
  993--1032.

\bibitem{HairerMattinglyScheutzow}
M.~Hairer, J.~C. Mattingly, and M.~Scheutzow, \emph{Asymptotic coupling and a
  general form of {H}arris' theorem with applications to stochastic delay
  equations}, Probab. Theory Related Fields \textbf{149} (2011), no.~1-2,
  223--259.

\bibitem{heng2019unbiased}
J.~Heng and P.~E. Jacob, \emph{{Unbiased Hamiltonian Monte Carlo with
  couplings}}, Biometrika \textbf{106} (2019), no.~2, 287--302.

\bibitem{kac1956}
M.~Kac, \emph{Foundations of kinetic theory}, Proceedings of the Third Berkeley
  Symposium on Mathematical Statistics and Probability, Volume 3: Contributions
  to Astronomy and Physics (Berkeley, Calif.), University of California Press,
  1956, pp.~171--197.

\bibitem{LeRe2004}
B.~Leimkuhler and S.~Reich, \emph{Simulating {H}amiltonian dynamics}, Cambridge
  Monographs on Applied and Computational Mathematics, Cambridge University
  Press, 2004.

\bibitem{levin2009markov}
David~Asher Levin, Yuval Peres, and Elizabeth~Lee Wilmer, \emph{Markov chains
  and mixing times}, American Mathematical Soc., 2009.

\bibitem{Lindvall}
Torgny Lindvall, \emph{Lectures on the coupling method}, Dover Publications,
  Inc., Mineola, NY, 2002, Corrected reprint of the 1992 original.

\bibitem{LindvallRogers}
Torgny Lindvall and L.~C.~G. Rogers, \emph{Coupling of multidimensional
  diffusions by reflection}, Ann. Probab. \textbf{14} (1986), no.~3, 860--872.

\bibitem{LuWang}
Jianfeng Lu and Lihan Wang, \emph{{On explicit $L^2$-convergence rate estimate
  for piecewise deterministic Markov processes in MCMC algorithms}}, 2021.

\bibitem{madras2010quantitative}
Neal Madras and Deniz Sezer, \emph{Quantitative bounds for {M}arkov chain
  convergence: {W}asserstein and total variation distances}, Bernoulli
  \textbf{16} (2010), no.~3, 882--908.

\bibitem{mangoubi2017rapid}
Oren Mangoubi and Aaron Smith, \emph{{Rapid mixing of Hamiltonian Monte Carlo
  on strongly log-concave distributions}}, arXiv preprint arXiv:1708.07114
  (2017).

\bibitem{MaWe2001}
J.~E. Marsden and M.~West, \emph{Discrete mechanics and variational
  integrators}, Acta Numerica \textbf{10} (2001), 357--514.

\bibitem{mattingly2010convergence}
J.~C. Mattingly, A.~M. Stuart, and M.~V. Tretyakov, \emph{Convergence of
  numerical time-averaging and stationary measures via {P}oisson equations},
  SIAM J Num Anal \textbf{48} (2010), no.~2, 552--577.

\bibitem{MeTw2009}
S.~P. Meyn and R.~L. Tweedie, \emph{Markov chains and stochastic stability},
  2nd ed., Cambridge University Press, New York, NY, 2009.

\bibitem{KacProgram2013}
St{\'e}phane Mischler and Cl{\'e}ment Mouhot, \emph{Kac's program in kinetic
  theory}, Inventiones mathematicae \textbf{193} (2013), no.~1, 1--147.

\bibitem{monmarche2020high}
Pierre Monmarch{\'e}, \emph{High-dimensional {MCMC} with a standard splitting
  scheme for the underdamped langevin}, arXiv preprint arXiv:2007.05455 (2020).

\bibitem{MontenegroTetali}
Ravi Montenegro and Prasad Tetali, \emph{Mathematical aspects of mixing times
  in {M}arkov chains}, Found. Trends Theor. Comput. Sci. \textbf{1} (2006),
  no.~3, x+121.

\bibitem{MueOr2004}
Stefan M{\"u}ller and Michael Ortiz, \emph{{On the $\Gamma$-convergence of
  discrete dynamics and variational integrators}}, Journal of Nonlinear Science
  \textbf{14} (2004), no.~3, 279--296.

\bibitem{Ne2011}
R.~M. Neal, \emph{{MCMC} using {H}amiltonian dynamics}, Handbook of {M}arkov
  {C}hain {M}onte {C}arlo \textbf{2} (2011), 113--162.

\bibitem{rump2018estimates}
Siegfried~M Rump, \emph{Estimates of the determinant of a perturbed identity
  matrix}, Linear algebra and its applications \textbf{558} (2018), 101--107.

\bibitem{SaloffCoste}
Laurent Saloff-Coste, \emph{Lectures on finite {M}arkov chains}, Lectures on
  probability theory and statistics ({S}aint-{F}lour, 1996), Lecture Notes in
  Math., vol. 1665, Springer, Berlin, 1997, pp.~301--413.

\bibitem{Vempala}
Santosh Vempala, \emph{Geometric random walks: a survey}, Combinatorial and
  computational geometry, Math. Sci. Res. Inst. Publ., vol.~52, Cambridge Univ.
  Press, Cambridge, 2005, pp.~577--616.

\bibitem{villani2008optimal}
C{\'e}dric Villani, \emph{Optimal transport: old and new}, vol. 338, Springer
  Science \& Business Media, 2008.

\bibitem{VillaniHypo}
C\'{e}dric Villani, \emph{Hypocoercivity}, Mem. Amer. Math. Soc. \textbf{202}
  (2009), no.~950.

\end{thebibliography}
\end{document}